\PassOptionsToPackage{dvipsnames}{xcolor}
\documentclass[pdflatex,sn-mathphys-num]{sn-jnl}

\usepackage{amsmath}
\usepackage{changes}
\usepackage{amsfonts}
\usepackage{amssymb}
\usepackage{natbib}
\usepackage{stmaryrd}

\usepackage{graphicx}%
\graphicspath{{../}{./}}
\usepackage{multirow}%
\usepackage{amsmath,amssymb,amsfonts}%
\usepackage{amsthm}%
\usepackage{mathrsfs}%
\usepackage[title]{appendix}%
\usepackage{xcolor}%
\usepackage{textcomp}%
\usepackage{manyfoot}%
\usepackage{booktabs}%
\usepackage{algorithm}%
\usepackage{algorithmicx}%
\usepackage{algpseudocode}%
\usepackage{listings}%
\usepackage{stmaryrd}
\usepackage{graphicx}
\usepackage{epstopdf}
\usepackage{caption}
\usepackage{subcaption}
\usepackage{tikz}
\usepackage{algorithm}
\usepackage{algpseudocode}
\usepackage[shortlabels]{enumitem}

\usepackage{mathrsfs}
\usepackage{colortbl}
\usepackage{titlesec}
\usepackage{geometry}

\definechangesauthor[color=blue]{HL}

\tikzset{
  frame/.style={thin},
  axis/.style={->,thin},
  char/.style={ultra thin},
  shock/.style={line width=1.5pt},
}
\newcommand{\xtframe}[2]{%
  \def\Xmax{#1}%
  \def\Tmax{#2}%
  \draw[axis] (0,0) -- (0,\Tmax+0.45) node[above] {$t$};
  \draw[axis] (0,0) -- (\Xmax+0.45,0) node[right] {$x$};
  \draw[frame] (0,\Tmax) -- (\Xmax,\Tmax);
  \draw[frame] (\Xmax,0) -- (\Xmax,\Tmax);
  \draw[frame] (-0.10,0) -- (0.10,0);
  \draw[frame] (-0.10,\Tmax) -- (0.10,\Tmax);
  \node[left] at (0,0) {0};
  \node[left] at (0,\Tmax) {$T$};
}

\theoremstyle{thmstyleone}%
\newtheorem{theorem}{Theorem}
\newtheorem{proposition}[theorem]{Proposition}%

\theoremstyle{thmstyletwo}%
\newtheorem{remark}{Remark}%

\theoremstyle{thmstylethree}%

\newtheorem{lemma}[theorem]{Lemma}
\newtheorem{assumption}{Assumption}

\def\R {\mathbb{R}}
\def\L{\mathcal{L}}
\def\J{\mathcal{J}}
\def\d{\mathrm{d}}
\def\sgn{\operatorname{sgn}}
\def\dist{\operatorname{dist}}
\def\x{\mathbf{x}}
\def\z{\mathbf{z}}
\def\n{\mathbf{n}}
\DeclareMathOperator\supp{supp}
\def\relu{\mathrm{ReLU}}
\def\sech{\mathrm{sech}}

\begin{document}
\newgeometry{
  left=1.05in,
  right=1.05in,
  top=1.05in,
  bottom=1.05in
}

\title[Neural network for conservation laws]{\bf\Large A neural network method for scalar conservation laws with convergence rates for shock waves}


\author[1]{\fnm{Jiachuan} \sur{Cao}}\email{jiachuan.cao@kit.edu}

\author[2]{\fnm{Buyang} \sur{Li}}\email{buyang.li@polyu.edu.hk}

\author[2]{\fnm{Hao} \sur{Li}}\email{hao94.li@polyu.edu.hk}

\affil[1]{\orgdiv{Institute for Applied and Numerical Mathematics}, \orgname{Karlsruhe Institute of Technology}, \orgaddress{\city{Karlsruhe}, \country{Germany}}\vspace{5pt}}

\affil[2]{\orgdiv{Department of Applied Mathematics}, \orgname{The Hong Kong Polytechnic University}, \orgaddress{\city{Hong Kong}}}


\abstract{
{\unboldmath
We propose a new entropy-compatible neural network method for scalar hyperbolic conservation laws and establish, to our knowledge, the first explicit \(L^1\) convergence rates in this setting for piecewise smooth entropy solutions, including those with discontinuities. The method is based on a computable approximation of the Kru\v{z}kov entropy residual that sits between the strong and weak forms of the entropy inequality.
For a single regular shock in one dimension and finitely many smooth shock hypersurfaces with transversal interactions in several dimensions, we construct explicit neural networks with provably small loss by combining shock-adapted continuous piecewise linear functions with known approximation properties of \(\tanh\) neural networks. The one-dimensional analysis also covers rarefactions, compound waves, regular shock interactions, and nondegenerate shock formation from smooth initial data. Together with entropy-based stability estimates, this gives rigorous \(L^1\) error bounds for minimizers of the proposed loss. In particular, when the network size grows in proportion to the number of degrees of freedom of a space--time mesh of size \(h\), the analysis recovers the classical Kuznetsov rate \(O(h^{1/2})\) in shock-dominated cases.
Numerical experiments in one and two space dimensions illustrate the performance of the method.
}
}



\maketitle

\section{Introduction}
This paper is devoted to the neural network approximation of scalar hyperbolic conservation laws (HCLs), a canonical class of partial differential equations that model transport and wave propagation phenomena arising in applications such as compressible fluid flow and traffic dynamics. Specifically, we consider the Cauchy problem
\begin{equation}\label{eq:scalar-HCL}
	\left\{
	\begin{aligned}
		&\partial_t u + \nabla_{\x}\cdot \mathbf{f}(u)=0,
		&&(\x,t)\in \mathbb{R}^d\times (0,T),\\
		&u(\x,0)=u_0(\x),
		&&\x\in \mathbb{R}^d,
	\end{aligned}
	\right.
\end{equation}
where \(u: \mathbb{R}^d\times (0,T) \to \mathbb{R}\) denotes the unknown conserved quantity, \(\mathbf{f}:\mathbb{R}\to\mathbb{R}^d\) is a prescribed flux function, and \(u_0\) is a given initial datum. It is well known that solutions to \eqref{eq:scalar-HCL} propagate with finite speed; consequently, compactly supported initial data give rise to solutions whose support remains confined to a bounded region over any finite time interval.
Throughout the paper we assume that \(u_0\in L^\infty_c(\mathbb R^d)\).
By finite speed of propagation, we choose a rectangular box
 \(\Omega\subset\mathbb R^d\) so large that $\rho_\Omega:=\dist\!\left(
\bigcup_{0\le t\le T}\supp u(\cdot,t),\partial\Omega
 \right)>0.$ Here
\(\dist(A,B):=\inf\{|a-b|:a\in A,\ b\in B\}\), with
\(|\cdot|\) the Euclidean norm in \(\mathbb R^d\).
In particular, the entropy solution has zero trace on the lateral
boundary \(\partial\Omega\times(0,T)\). This assumption allows the weak formulation and the associated numerical schemes to be posed on the bounded computational domain \(\Omega_T := \Omega\times(0,T)\). In particular, a function \(u \in L^{\infty}(\Omega_T)\) is called a \emph{weak solution} of \eqref{eq:scalar-HCL} with initial datum \(u_0\) if $\forall\,\varphi\in C_c^\infty \left(\Omega\times [0,T)\right)$
\begin{equation}\label{eq:weak_formulation}
	\int_{\Omega_T} \big(u \,\partial_t \varphi +  \mathbf{f}(u)\cdot \nabla_{\x}\varphi\big) \, \d \x\,\d t
	+ \int_{\Omega}  u_0(\x)\,\varphi(\x, 0) \, \d \x = 0.
\end{equation}

We use the convention \(\sgn(0)=0\). A fundamental difficulty in the analysis of \eqref{eq:scalar-HCL} is that, even for smooth initial data, shocks may develop in finite time, so classical solutions generally fail to exist globally. One must therefore work with weak solutions, supplemented by an entropy condition to single out the physically relevant one. A standard choice of entropy condition is the Kru\v{z}kov family \(\eta_k(u)=|u-k|=(u-k)\sgn\!\left(u-k\right)\), \(k\in \mathbb{R}\), with associated entropy fluxes \(\boldsymbol{\psi}_k\) satisfying \(\boldsymbol{\psi}_k^{\prime}(u)=\eta_k^{\prime}(u)\mathbf{f}^{\prime}(u)\). The corresponding entropy inequalities
\begin{equation}\label{eq:entropy_inequality}
	\partial_t \eta_k(u) + \nabla_{\x}\cdot \boldsymbol{\psi}_k(u) \leq 0, \quad \forall\, k\in\R,
\end{equation}
are required to hold in the distributional sense on \(\Omega_T\). Equivalently, \(u\) is an \emph{entropy solution} of \eqref{eq:scalar-HCL} with initial datum \(u_0\) if, for every \(k\in\mathbb{R}\) and every nonnegative test function \(\varphi\in C_c^\infty(\Omega\times[0,T))\),
\begin{equation}\label{eq:entropy_solution}
	\int_0^T\!\!\int_{\Omega}
	\Big( |u-k|\,\partial_t\varphi
	+\sgn\!\left(u-k\right)\big(\mathbf{f}(u)-\mathbf{f}(k)\big)\cdot\nabla_{\x}\varphi\Big)\,\d\x\,\d t
	+\int_{\Omega}|u_0-k|\,\varphi(\x,0)\,\d\x
	\ge 0,
\end{equation}
which is precisely the distributional formulation of the entropy inequalities \eqref{eq:entropy_inequality}.

Over the past several decades, considerable effort has been devoted to the development of numerical schemes for hyperbolic conservation laws. Early foundational methods, such as the Lax--Friedrichs and Godunov
schemes, provide robust shock capturing but are restricted to first-order accuracy: numerical diffusion smears discontinuities, and Godunov's theorem rules out linear monotone schemes of higher order. Increasing the formal order of accuracy improves the resolution of smooth regions, but generally introduces spurious oscillations near discontinuities. Reconciling these two requirements---high accuracy in smooth regions and non-oscillatory behavior near shocks---has been a central theme in the field. Early remedies include flux- and slope-limited approaches such as total variation diminishing (TVD) schemes and the monotonic upstream-centered schemes for conservation laws (MUSCL). Subsequent developments led to a family of nonlinear high-resolution methods, including essentially non-oscillatory (ENO) schemes~\cite{Harten1987}, weighted ENO (WENO) schemes~\cite{liu1994weighted, JiangShu1996}, Runge--Kutta discontinuous Galerkin (RKDG) methods~\cite{CockburnShu2001}, and entropy-stable schemes~\cite{tadmor2016entropy}. On the theoretical side, the convergence analysis of numerical schemes can be traced back to Kru\v{z}kov's seminal work~\cite{kruvzkov1970first}, which established the notion of entropy solutions ensuring uniqueness and stability. Building upon this, Kuznetsov~\cite{kuznetsov1976accuracy}
derived the first a~priori \(L^\infty(L^1)\) error estimate of order \(\mathcal{O}(h^{1/2})\) for monotone schemes on uniform grids, with Sanders~\cite{sanders1983convergence} extending the result to non-uniform Cartesian meshes. Related \(\mathcal{O}(h^{1/2})\) estimates were obtained for numerical schemes by Cockburn and Gremaud~\cite{cockburn1996priori} and for vanishing viscosity by Makridakis and Perthame~\cite{makridakis2003optimal}. Moreover, Nessyahu and Tadmor \cite{nessyahu1992convergence} proved that several finite volume schemes
converge at rate $\mathcal{O}(h)$ in the dual Lipschitz norm, for Lip$^+$-bounded initial data. Another way to prove convergence to the entropy solution is DiPerna's theory of measure-valued solutions \cite{diperna1985measure}, which was first applied by Szepessy in 1989 \cite{szepessy1989convergence} to the streamline diffusion method and was later applied to finite difference, finite volume, and various space--time DG methods.

Despite the remarkable maturity of these classical schemes, the rapid
rise of modern artificial intelligence has created a strong impetus to
develop neural-network-based numerical methods for partial differential
equations. As AI-specific hardware (GPUs, TPUs, and emerging neural
accelerators), automatic differentiation frameworks, and large-scale
foundation models become the dominant computational infrastructure, it
is increasingly desirable to have PDE solvers that are natively
compatible with these platforms. Neural-network solvers are mesh-free,
differentiable end-to-end, and can be seamlessly integrated into larger
AI pipelines---for example, as differentiable physics modules inside
foundation models, as surrogates for inverse problems and optimal
control, or as building blocks for scientific large models. From this
perspective, the goal of developing neural-network methods for
hyperbolic conservation laws is not necessarily to outperform highly
optimized classical schemes on standard benchmarks, but rather to
provide reliable, theoretically grounded PDE solvers that fit naturally
into the emerging AI-centric computational ecosystem. This perspective
motivates a careful study of neural-network approximations to entropy
solutions, supported by rigorous error analysis, which is the focus of
the present work.

A wide range of neural-network methods for PDEs has emerged in recent
years. Representative examples include physics-informed neural networks
(PINNs) and their variants~\cite{Raissi2019, akrivis2025runge}, the deep Ritz
method~\cite{E2018, cai2025efficient}, operator learning approaches
such as the Fourier neural operator (FNO)~\cite{Li2021FNO} and
DeepONet~\cite{Lu2021}, and dynamically parametrized neural-network
and kernel methods for time-evolving problems~\cite{du2021evolutional, feischl2024regularized, su2025spike}.
Among these, PINNs are particularly natural for boundary- and
initial-value problems, as they directly penalize the PDE residual;
their theoretical properties have been investigated
in~\cite{de2024numerical, gazoulis2025stability}. For hyperbolic
conservation laws, however, this strong-form formulation runs into a
fundamental difficulty. Letting $U_{\mathrm{NN}}$ denote a neural
network approximation, standard PINN losses penalize the strong-form
residual
$\partial_{t} U_{\mathrm{NN}} + \nabla_{\mathbf{x}}\!\cdot
\mathbf{f}(U_{\mathrm{NN}})$ in the $L^{2}$ norm, which cannot in
general be made small along a sequence approximating a discontinuous
entropy solution; see~\cite{chaumet2024improving}. This obstruction has
motivated alternative loss functionals for discontinuous solutions.
One approach uses a least-squares formulation with a discrete
divergence operator~\cite{cai2023least}; another enforces conservation
over space--time control volumes and incorporates entropy constraints
into the loss~\cite{patel2022thermodynamically}. Entropy conditions
select the physically relevant weak solution when weak solutions are
non-unique. On the analytical side,
rigorous approximation and generalization results for neural-network
approximations of parametric hyperbolic conservation laws were
established in~\cite{de2024error}. More recently, under the assumption
that the loss can be made arbitrarily small, \cite{oubarka2026weak}
derived an a posteriori error estimate by invoking the Kru\v{z}kov-type
stability estimates of Bouchut and Perthame~\cite{bouchut1998kruzkov}.

For scalar HCLs, the Kru\v{z}kov entropy family becomes particularly important. Specifically, the associated Kru\v{z}kov entropy residual plays two roles simultaneously. First, it recovers the conservation law in the weak sense. Indeed, by taking $k\in \mathbb{R}$ large or small enough, the full family of Kru\v{z}kov entropy inequalities \eqref{eq:entropy_inequality} recovers the weak formulation \eqref{eq:weak_formulation}. Second, for general $k\in\R$,
\eqref{eq:entropy_inequality} enforces entropy admissibility and thereby selects the physically relevant solution, restoring uniqueness and the classical $L^1$-stability property. Thus, the Kru\v{z}kov entropy residual simultaneously measures \emph{weak consistency} and \emph{entropy admissibility}.
This observation is the common starting point of both weak PINNs and the present work.

The weak PINN (wPINN) framework \cite{de2024wpinns} exploits these dual effects by enforcing the entropy inequality in a distributional sense. Schematically, upon inserting a neural network ansatz \(U_{\mathrm{NN}}\) into \eqref{eq:entropy_solution} and temporarily suppressing the initial and boundary terms, one seeks to enforce \eqref{eq:entropy_inequality} weakly in the form
\begin{equation}\label{wpinn_loss}
	-\int_0^T\!\!\int_{\Omega}
	\Big( \eta_k(U_{\mathrm{NN}})\,\partial_t\varphi
	+\boldsymbol{\psi}_k(U_{\mathrm{NN}})\cdot\nabla_{\x}\varphi\Big)\,\d\x\,\d t
	\le o(1),
	\quad \forall\, k\in \mathbb{R},
\end{equation}
for every nonnegative test function \(\varphi\in C_c^\infty(\Omega_T)\) subject to suitable normalization constraints. In practice, the test functions \(\varphi\) are themselves parametrized by auxiliary neural networks, which gives rise to a min--max optimization problem. Error estimates have been established for neural network solutions that remain bounded in \(C^1\). A modified wPINN was subsequently proposed in \cite{chaumet2024improving} to decouple weak consistency from entropy admissibility, thereby simplifying the training procedure and rendering the framework more amenable to systems.

The objective of this paper is to develop a novel neural network method that
\emph{(i)} matches the computational cost of wPINN and \emph{(ii)} is provably convergent, with an \emph{explicit convergence rate}, for the shock configurations considered in this work. To the best of our knowledge, no existing neural network method for scalar hyperbolic conservation laws satisfies both properties simultaneously.

Our starting point is a simple observation: since the neural network ansatz $U_{\mathrm{NN}}$ is at least continuous, fully relaxing the entire Kru\v{z}kov family \eqref{eq:entropy_inequality} into the distributional form \eqref{wpinn_loss} is excessive in the neural network setting. Motivated by this, we propose a new enforcement strategy that sits
\emph{between} the strong and weak formulations of
\eqref{eq:entropy_inequality}. Its key ingredient is the following
discrete surrogate of the Kru\v{z}kov entropy residual:
\begin{equation}\label{eq:residual_entropy}
	\mathcal{L}_{\mathrm{ent}}(U_{\mathrm{NN}})
	:= \sup_{k_h\in V_h^c}
	\int_{\Omega_T}
	\Big(\partial_t U_{\mathrm{NN}}+\nabla_{\x}\cdot \mathbf f(U_{\mathrm{NN}})\Big)\,
	\sgn\!\left(U_{\mathrm{NN}}-k_h\right)\,\d\x\,\d t,
\end{equation}
where the bounded comparison class $V_h^c$ (defined in
Section~\ref{sec:notation}) is chosen rich enough to probe local entropy
inequality violations. The motivation comes from the pointwise identity
\begin{equation}\label{eq:entropy_production}
	\begin{aligned}
	\partial_t \eta_k(U_{\mathrm{NN}})
	+ \nabla_{\x}\cdot \boldsymbol \psi_k(U_{\mathrm{NN}})
	&= \Big(\partial_t U_{\mathrm{NN}}+\nabla_{\x}\cdot
	\mathbf f(U_{\mathrm{NN}})\Big)\sgn\!\left(U_{\mathrm{NN}}-k\right),
	\end{aligned}
\end{equation}
for $(\x,t)\in\Omega_T$ and $k\in\mathbb R$, which shows that the
integrand in \eqref{eq:residual_entropy} is the
entropy production associated with the Kru\v{z}kov entropy pair
$(\eta_k,\boldsymbol{\psi}_k)$.

The crucial design choice is to replace the scalar parameter $k\in\mathbb{R}$ by a space--time function $k_h\in V_h^c$. This serves two purposes. First, it keeps the loss computable while making the constraint local: the supremum over $k_h$ provides a discrete surrogate for testing the continuum family $k\in\mathbb{R}$ \emph{at every space--time point}, so that the locality and richness of $V_h^c$ directly control how effectively
\eqref{eq:residual_entropy} detects local entropy violations. Second,
this choice aligns with the doubling-of-variables technique used in our later stability analysis, where $k_h$ is taken to approximate the
entropy solution itself. Accordingly, the integrand
in~\eqref{eq:residual_entropy} should be viewed as an
\emph{entropy residual surrogate} rather than as the exact Kru\v{z}kov entropy residual associated with a constant state.

In the idealized limit $h\to 0$, in which $V_h^c$ becomes sufficiently
rich and local, driving
$\mathcal{L}_{\mathrm{ent}}(U_{\mathrm{NN}})$ to zero suppresses
violations of the entropy inequalities
\eqref{eq:entropy_inequality} and steers $U_{\mathrm{NN}}$ toward the
entropy solution in the sense of \eqref{eq:entropy_solution}. Hence
minimizing \eqref{eq:residual_entropy} simultaneously promotes weak
consistency and entropy admissibility, which is precisely what
underlies the entropy-based stability estimate developed later
(cf.\ Remark~\ref{rem:motivation}).

Because the proposed loss requires only the governing equation and prescribed initial and boundary data, rather than labeled solution trajectories, it could also provide a physics-informed objective for unsupervised operator learning over families of initial data~\cite{wang2021learning}. This potential extension further motivates the design of neural network methods with provable convergence rates, although operator learning and its generalization analysis are beyond the scope of this paper.

The main contributions of this paper are as follows.
\begin{enumerate}
	\item[(C1)] \textbf{A new entropy-residual loss between strong and weak
		enforcement.} We introduce the entropy component of the computable loss
	\eqref{eq:loss_func}, which is a surrogate of the Kru\v{z}kov entropy
	residual obtained by replacing the constant parameter $k\in\mathbb{R}$ by a space--time comparison function $k_h$ in the finite-dimensional space $V_h^c$. The resulting loss promotes both weak consistency and entropy admissibility, while retaining a comparable computational cost to wPINN.

	\item[(C2)] \textbf{A constructive bridge between finite element
		approximations and neural networks.} ReLU representations of finite element functions have been studied by He et al.~\cite{he2020relu,he2022hierarchical} and, in a weak sense, by Jin~\cite{jin2026two}. We construct explicit
	continuous piecewise linear (CPwL) competitors on shock-adapted
	simplicial meshes that produce provably small loss, and then transfer these constructions to the neural setting via representability and approximation results for $\tanh$ neural networks. This yields, for the first time in this context, an a priori upper bound
	on the loss attained by a minimizer of the training problem.

	\item[(C3)] \textbf{Explicit $L^{1}$ convergence rates for shock solutions.} Combining (C1)--(C2) with entropy stability, we prove the rate $O(h^{1/2})$ for a single regular shock in one dimension and finitely many smooth shock hypersurfaces with transversal interactions in several dimensions, when the architecture scales like $O(h^{-(d+1)})$. In one dimension, the additional analyses below cover rarefaction and compound waves, regular shock interactions, and nondegenerate shock formation from smooth initial data; the corresponding rates are stated in Theorem~\ref{thm:main_1d}.

	\item[(C4)] \textbf{Numerical validation.} We provide numerical experiments in one and two space dimensions that illustrate sharp shock resolution.
\end{enumerate}

The rest of this paper is organized as follows. Section~\ref{sec:notation} introduces the problem setting, notation, and the proposed loss functional. Sections~\ref{sec:1d} and~\ref{sec:multiD} develop the convergence analysis in one and several space dimensions, respectively, under piecewise smoothness assumptions. Section~\ref{sec:numerics} reports numerical experiments. The additional one-dimensional wave configurations are treated in Section~\ref{sec:1d}. Some technical proofs and auxiliary approximation results are collected in the appendix.

\section{Neural network methods for scalar conservation laws}
\label{sec:notation}

This section sets the stage for our neural network methods for scalar
hyperbolic conservation laws (HCLs). We first recall the continuous
problem and the entropy framework on which the analysis is built
(Section~\ref{subsec:scl}), and then introduce the discrete spaces and
mesh notation that underpin the entropy-residual loss
\eqref{eq:residual_entropy} and its analysis
(Section~\ref{subsec:notation}).

\subsection{Scalar conservation laws and entropy solutions}
\label{subsec:scl}

Let $\x\in\mathbb{R}^d$ denote the spatial variable and $t\in[0,T]$
the time. Throughout the paper, $\Omega\subset\mathbb{R}^d$ is a
 rectangular box satisfying the uniform support-buffer condition stated in
 the introduction. We write
$\Omega_T:=\Omega\times(0,T)$ for the space--time domain and
$\z:=(\x,t)$ for a generic space--time point. The flux
$\mathbf{f}\in C^2(\mathbb{R};\mathbb{R}^d)$ is fixed, and it is convenient to introduce the \emph{space--time flux}
\[
\mathbf{F}(v):=\bigl(\mathbf{f}(v),\,v\bigr)\in\mathbb{R}^{d+1},
\qquad
\nabla\cdot \mathbf{F}(v)=\partial_t v+\nabla_{\x}\cdot \mathbf{f}(v),
\]
where \(\nabla\) denotes the full space--time gradient
\(\nabla_{(\x,t)}\).

Entropy solutions of \eqref{eq:scalar-HCL} are generally only
$L^\infty$ and may contain shocks, so in
Sections~\ref{sec:1d}--\ref{sec:multiD} we work under the
piecewise smoothness hypotheses in
Assumptions~\ref{assump:piecewise_smooth}, \ref{assump:single_shock}, and~\ref{assump:shock_3d}, as appropriate to the wave configuration.
Let $\Gamma\subset\Omega_T$ denote a regular $C^2$ shock hypersurface, that is, a
$d$-dimensional embedded $C^2$ manifold in space--time, and let
$\mathbf{n}_\Gamma(\z)$ be a chosen unit normal at $\z\in\Gamma$. We
define the one-sided traces and the jump of a function $v$ across
$\Gamma$ by
\[
v^{\pm}(\z) := \lim_{\varepsilon\to 0^+}
v\bigl(\z\pm \varepsilon\, \mathbf{n}_\Gamma(\z)\bigr),
\qquad
\llbracket v \rrbracket := v^{+}-v^{-}.
\]

In this piecewise smooth setting, the entropy inequality
 \eqref{eq:entropy_inequality} reduces to an admissibility condition at
 each regular point of a shock sheet: for every $k\in\mathbb R$,
\begin{equation}\label{eq:reformulated_entropy_condition}
	\llbracket
	\bigl(\mathbf{F}(u) - \mathbf{F}(k)\bigr)\sgn(u-k)
	\rrbracket \cdot \mathbf{n}_\Gamma(\z) \le 0.
\end{equation}
The condition is independent of the orientation of $\mathbf{n}_\Gamma$,
since reversing $\mathbf{n}_\Gamma$ exchanges the roles of $+$ and $-$.
 Throughout the paper we assume that the entropy solution $u$ admits
 one-sided traces along the regular part of each shock sheet satisfying
 both the
Rankine--Hugoniot relation and the
admissibility condition~\eqref{eq:reformulated_entropy_condition}.

\subsection{Discrete setting and notation}\label{subsec:notation}

We now introduce the discrete objects on which the entropy-residual
loss \eqref{eq:residual_entropy} is built. The guiding principle is
that the supremum over $k_h\in V_h^c$ in
\eqref{eq:residual_entropy} should be both \emph{rich enough} to
detect entropy violations and \emph{computable} on a mesh; both
features are encoded in the choice of $V_h^c$ below.

\paragraph{Space--time meshes}
Let $\Lambda_h$ be a uniform Cartesian partition of
$\overline{\Omega_T}$ into closed $(d{+}1)$-dimensional
hyperrectangles, with all space--time side lengths comparable to the
mesh size
\[
h:=\max_{K\in\Lambda_h} h_K,
\qquad
h_K:=\operatorname{diam}(K).
\]
We denote the uniform coordinate widths by
$\Delta x_1,\ldots,\Delta x_d,\Delta t$; thus each is comparable to
$h$, with constants independent of $h$.
We write
$\partial\Lambda_h:=\bigcup_{K\in\Lambda_h}\partial K$ for the
space--time skeleton of the mesh and, for each cell $K$, we denote by
$\mathbf{n}$ the unit outward normal on $\partial K$.

\paragraph{Discrete comparison space}
The comparison functions $k_h$ in the entropy-residual loss
\eqref{eq:residual_entropy} are drawn from the discontinuous
piecewise multilinear space
\[
V_h := \Bigl\{ v_h\in L^\infty(\Omega_T) :
v_h|_K\in\mathbb{Q}^1(K), \forall K\in\Lambda_h \Bigr\},
\]
where $\mathbb{Q}^1(K)$ is the space of polynomials of degree at most
 one in each space--time variable on $K$. For a continuous function $w$ on $K$, let $I_hw$ denote its cellwise
 $\mathbb Q^1$ nodal interpolant. To control the size of \(k_h\) and its variation inside each cell, while
still allowing jumps across mesh faces, we equip \(V_h\) with the
mesh-dependent norm
\begin{equation}\label{eq:W1infty_kh}
	\|k_h\|_{W_h^{1,\infty}}
	:= \max_{K\in\Lambda_h}\Bigl(
	\|k_h\|_{L^\infty(K)} + h\,\|\nabla k_h\|_{L^\infty(K)}
	\Bigr).
\end{equation}
For the scalar conservation law under consideration, by the maximum principle, we fix once and for all
a number \(c>0\), independent of \(h\), such that
\[
\|u\|_{L^\infty(\Omega_T)}
\le \|u_0\|_{L^\infty(\Omega)} < \frac c2 .
\]
In each convergence theorem, after the relevant one-sided extensions are
fixed, \(c\), independently of \(h\), is chosen larger than twice their
uniform \(L^\infty\) bound. Their fixed \(W^{1,\infty}\) bounds and the
Cartesian interpolation constant then ensure, for all sufficiently small \(h\),
that the corresponding interpolants belong to \(V_h^c\).
Motivated by this bound, in the stability analysis, we restrict the class of discontinuous
piecewise polynomial comparison functions to
\[
V_h^{c} := \bigl\{ k_h\in V_h : \|k_h\|_{W_h^{1,\infty}}\le c \bigr\}.
\]

The bound \eqref{eq:W1infty_kh} ensures that $k_h$ remains
\emph{$L^\infty$-bounded uniformly in $h$} while still allowing
$O(1/h)$ slopes per cell, exactly the regime needed to approximate
the entropy solution and its discontinuities.

\paragraph{Terminology and refinements}
We refer to a function $k_h\in V_h^c$ as a \emph{discontinuous
	piecewise polynomial} (DPwP). We also use auxiliary \emph{continuous
piecewise linear} (CPwL) functions on a separate conforming simplicial
mesh \(\mathcal T\). This shock-adapted mesh need not refine
\(\Lambda_h\) and may contain anisotropic simplices in thin shock layers.

\paragraph{Conventions}
For a subset of cells $\Lambda\subset\Lambda_h$ we write $\#\Lambda$
for its cardinality and, when no confusion can arise, identify
$\Lambda$ with the associated union of cells
$\bigcup_{K\in\Lambda} K$. Throughout the paper, generic constants
 are denoted by $C$ and may depend on the fixed dimension $d$, domain
 $\Omega$, final time $T$, flux $\mathbf f$, comparison bound $c$, mesh
 comparability constants, exact entropy solution $u$, and the geometric
 constants in the standing assumptions, but never on the mesh size $h$ or
 on the network parameters. We write $A\lesssim B$ if $A\le CB$ with
 such a constant, and $A\simeq B$ if both $A\lesssim B$ and
 $B\lesssim A$.
For \(A\subset\overline{\Omega_T}\) and \(r>0\), with
\(\dist(\z,A):=\inf_{\mathbf y\in A}|\z-\mathbf y|\), define
\[
	N_r(A):=\bigl\{\z\in\Omega_T:\dist(\z,A)<r\bigr\},
	\qquad
	\overline N_r(A):=\bigl\{\z\in\overline{\Omega_T}:
	\dist(\z,A)\leq r\bigr\}.
\]
Here \(|\cdot|\) is the Euclidean norm on
\(\mathbb R^{d+1}\), and \(\dist(\z,\emptyset):=+\infty\).

\subsection{Neural networks and the minimization problem}

\paragraph{Neural networks with bounded output}
We use the same truncation level \(c\) fixed in Section~\ref{sec:notation}
 to define a clipped neural network. For \(\z=(\x,t)\in\Omega_T\), with activation functions $\rho = \relu$ and $\sigma=\tanh$, set
\begin{align}\label{def:clipped-NN}
	&   u_\theta(\z):= \Pi_c\bigl(U_\theta^{\rm raw}(\z)\bigr),   \\
	& \Pi_c(r):  = \min\left\{\max\left\{r,-\frac c2\right\},\frac c2\right\} = \rho\!\left(c-\rho\!\left(\frac c2-r\right)\right)-\frac c2 ,  \nonumber\\
	\label{def:NN}
	&  U_\theta^{\rm raw}(\z)
	:=
	\bigl(A_L^\theta\circ \sigma\circ A_{L-1}^\theta
	 \circ\cdots\circ \sigma\circ A_1^\theta\bigr)(\z),
	\quad A_\ell^\theta(y)=W_\ell y+b_\ell,
\end{align}
where \(\theta=\{(W_\ell,b_\ell)\}_{\ell=1}^L\) denotes the collection of trainable parameters.
Since \(\Pi_c\) is Lipschitz, the derivatives of \(u_\theta=\Pi_c(U_\theta^{\rm raw})\) are defined almost everywhere, and
\[
\nabla u_\theta
=
\mathbf 1_{\{-c/2<U_\theta^{\rm raw}<c/2\}}
\nabla U_\theta^{\rm raw}
\qquad\text{a.e. in }\Omega_T .
\]
Accordingly, differential expressions such as \(\nabla\!\cdot\mathbf F(u_\theta)\) are understood in the a.e. sense.

The trainable part of the
architecture is the raw $\tanh$ neural network \(U_\theta^{\rm raw}\). The clipping layer is fixed and introduces
no additional trainable parameters. We denote by
\[
\mathcal{N}
:=\bigl\{\,u_\theta =\Pi_c\bigl(U_\theta^{\rm raw}\bigr)  : \theta\in\Theta\,\bigr\}
\]
the resulting network class, where $\Theta$ encodes the chosen architecture (depth and width).

\paragraph{Loss functional}
 For \(v\in W^{1,1}(\Omega_T)\cap L^\infty(\Omega_T)\) with well-defined traces on
\(\Omega\times\{0\}\) and \(\partial\Omega\times(0,T)\), we introduce
the total loss
\begin{equation}\label{eq:loss_func}
	\begin{aligned}
		\L(v)\ &:= \L_{\mathrm{ent}}(v)+\L_{\mathrm{reg}}(v)+\L_{\mathrm{ibc}}(v),\\[2pt]
		\L_{\mathrm{ent}}(v)\ &:=\ \sup_{k_h\in V_h^c}\J_{\mathrm{ent}}(v;k_h),
		\quad\J_{\mathrm{ent}}(v;k_h):=\int_{\Omega_T}\nabla\!\cdot\mathbf{F}(v)\,
		\sgn\!\bigl(v-k_h\bigr)\,\d\z,\\[2pt]
		\L_{\mathrm{reg}}(v)\ &:=\ h\!\int_{\Omega_T}\bigl|\nabla\!\cdot\mathbf{F}(v)\bigr|\,\d\z,
		\quad \L_{\mathrm{ibc}}(v)\ :=\ \mathcal L_{\mathrm{ibc}}^{0}(v),
	\end{aligned}
\end{equation}
where, for a prescribed lateral trace \(g\in L^1(\partial\Omega\times(0,T))\), we define
\begin{equation}\label{eq:ibc_nonhom}
 \mathcal L_{\mathrm{ibc}}^{g}(v):=
 \|v(\cdot,0)-u_0\|_{L^1(\Omega)}
 +\|v-g\|_{L^1(\partial\Omega\times(0,T))}.
\end{equation}
 The analysis uses
 \(\mathcal L_{\mathrm{ibc}}=\mathcal L_{\mathrm{ibc}}^{0}\), as ensured by the
 support buffer, whereas the numerical benchmarks use their prescribed
 exact trace \(g\).

The three contributions play complementary roles:
$\L_{\mathrm{ent}}$ enforces the entropy admissibility condition through a Kru\v{z}kov-type residual,
 $\L_{\mathrm{ibc}}$ matches the prescribed initial and lateral data, and
 \(\mathcal L_{\mathrm{reg}}\) is a mild \(L^1\)-residual regularization
 term suppressing spurious oscillatory behavior.
Composition with the network gives the parameter-space loss
\(
\mathscr{L}(\theta):=\L(u_\theta).
\)
In addition, we emphasize that $\L_{\mathrm{ent}} \geq 0$ when $c$ is taken large enough.
\begin{lemma}[Nonnegativity of the entropy loss for a clipped neural network]\label{lem:nonnegativity}
	 Let \(v\in W^{1,1}(\Omega_T)\) satisfy $v(\z) \in [-\frac c2 , \frac c2 ]$, for a.e. $\z\in\Omega_T$. Then
	\[
	 \L_{\mathrm{ent}}(v)\ge 0 .
	\]
\end{lemma}
\begin{proof}
	The functions $k_h\equiv\pm c$ belong to $V_h^c$ and give the two
	values $\pm\int_{\Omega_T}\nabla\!\cdot\mathbf F(v)\,\d\z$ in the
	supremum. Their maximum is nonnegative.
\end{proof}
\begin{remark}[Effects of the clipping layer]
	Clipping supplies the uniform range on which all flux constants are
	taken.
\end{remark}

\paragraph{The neural network scheme}
The numerical scheme proposed in this paper consists of solving the minimization problem
\begin{equation}\label{eq:NN_min_problem}
	\theta^{*}\ \in\ \underset{\theta\in\Theta}{\rm argmin}\,\mathscr{L}(\theta),
	\qquad
	u_{\theta^{*}} \in\mathcal N .
\end{equation}
The analysis in Sections~\ref{sec:1d}--\ref{sec:multiD} establishes quantitative
$L^{1}$ error estimates between $u_{\theta^*}$ and the entropy solution
$u$ of \eqref{eq:scalar-HCL}.

\paragraph{Idealized setting for the analysis}
For a neural network method, the total error is classically split
into approximation, optimization, and sampling/quadrature errors;
see, e.g., \cite{mishra2023estimates}. Since our focus is the
PDE-specific difficulty of approximating entropy solutions, we work
throughout in an idealized setting:
\begin{itemize}[label=(A\arabic*),leftmargin=2.5em,itemsep=2pt,topsep=4pt]
	\item[(A1)] the supremum in $\L_{\mathrm{ent}}$ is taken exactly over
	$V_h^c$, and all integrals in \eqref{eq:loss_func} are evaluated
	exactly;
	\item[(A2)] the minimization problem \eqref{eq:NN_min_problem} attains
	a (global) minimizer $\theta^{*}\in\Theta$.
\end{itemize}
 Under (A1)--(A2) we do not track optimization, sampling, or quadrature
 errors.
The sampled, fully discrete version of \eqref{eq:loss_func} used in
the experiments is described in
Section~\ref{sec:numerics}.
\begin{lemma}[Sign chain rule and Gauss--Green formula]
\label{lem:sign_chain_rule}
Let $D\subset\mathbb R^{d+1}$ be bounded and Lipschitz and let
$a,b\in W^{1,\infty}(D)$. Then
\begin{equation}\label{eq:def_Q}
\mathcal Q(a,b):=(\mathbf F(a)-\mathbf F(b))\sgn(a-b) \in W^{1,\infty}(D;\mathbb R^{d+1})
\end{equation}
and, for scalars \(r,s,\widetilde r,\widetilde s\) in any fixed
bounded interval,
\begin{equation}\label{eq:Lipschitz_Q}
\bigl|\mathcal Q(r,s)-\mathcal Q(\widetilde r,\widetilde s)\bigr|
\le C\bigl(|r-\widetilde r|+|s-\widetilde s|\bigr).
\end{equation}
Moreover,
\begin{equation}\label{eq:sign_chain_rule}
\nabla\!\cdot\mathcal Q(a,b)
=\bigl(\nabla\!\cdot\mathbf F(a)-\nabla\!\cdot\mathbf F(b)\bigr)
 \sgn(a-b)\qquad\text{a.e. in }D.
\end{equation}
Consequently, for every Lipschitz subdomain $E\subset D$,
\begin{equation}\label{eq:sign_gauss_green}
\int_E\bigl(\nabla\!\cdot\mathbf F(a)-\nabla\!\cdot\mathbf F(b)\bigr)
\sgn(a-b)\,\d\z
=\int_{\partial E}\mathcal Q(a,b)\cdot\mathbf n_E\,\d s .
\end{equation}
The same identity holds for a set $E\subset D$ of finite perimeter whose
closure is contained in a Lipschitz subdomain of $D$, with $\partial E$
replaced by the reduced boundary $\partial^*E$ and the Sobolev trace of
$\mathcal Q(a,b)$.
\end{lemma}
\begin{proof}
On the bounded range of $(a,b)$, the map
$(r,s)\mapsto(\mathbf F(r)-\mathbf F(s))\sgn(r-s)$ is Lipschitz: write
$\mathbf F(r)-\mathbf F(s)=(r-s)\int_0^1\mathbf F'(s+\lambda(r-s))\,\d\lambda$.
This proves \eqref{eq:Lipschitz_Q} and
\(\mathcal Q(a,b)\in W^{1,\infty}\). On $\{a\ne b\}$,
\eqref{eq:sign_chain_rule} is the ordinary chain rule. On $\{a=b\}$,
$\mathcal Q(a,b)=0$, and the Sobolev constancy-on-level-sets property
gives $\nabla\mathcal Q(a,b)=0$ a.e. The equality-set property also gives
$\nabla a=\nabla b$ a.e., so both sides of
\eqref{eq:sign_chain_rule} vanish there. Formula
\eqref{eq:sign_gauss_green} follows from the classical Gauss--Green theorem.
\end{proof}
\begin{remark}[Why this loss? --- stability mechanism]\label{rem:motivation}
	Suppose that $u$ is a piecewise $C^1$ entropy solution with a
	single $C^2$ shock hypersurface $\Gamma$ separating $\Omega_T$ into
	$\Omega_T^{\pm}$, and let $v\in W^{1,\infty}(\Omega_T)$ satisfy $\|v\|_{L^\infty(\Omega_T)}\le c/2$. Lemma~\ref{lem:sign_chain_rule} and Gauss--Green on the two
	sides of $\Gamma$ give
	\begin{equation}\label{eq:motivation_J_identity}
		\begin{aligned}
			\J_{\mathrm{ent}}(v;u)=&\Bigl(\int_{\Omega_T^{-}}+\int_{\Omega_T^{+}}\Bigr)
			\nabla\!\cdot\!\bigl((\mathbf F(v)-\mathbf F(u))\,\sgn(v-u)\bigr)\,\d\z\\
			=&\|v(\cdot,T)-u(\cdot,T)\|_{L^1(\Omega)}
			-\|v(\cdot,0)-u(\cdot,0)\|_{L^1(\Omega)}\\
			&+
			\int_{\partial\Omega\times(0,T)}
			\bigl(\mathbf f(v)-\mathbf f(u)\bigr)
			\sgn(v-u)\cdot\mathbf n_{\partial \Omega}\,\d s \\
            &- \int_{\Gamma}
			\bigl\llbracket(\mathbf F(v)-\mathbf F(u))\,\sgn(v-u)\bigr\rrbracket
			\cdot\mathbf n_\Gamma\,\d s,
		\end{aligned}
	\end{equation}
		 Here $\mathbf n_{\partial\Omega}$ is the outward spatial
unit normal. We used $\nabla\!\cdot\mathbf F(u)=0$ on
each smooth side. By
\eqref{eq:reformulated_entropy_condition} with $k=v(\z)$,
the shock contribution in \eqref{eq:motivation_J_identity},
including its minus sign, is nonnegative.
Rearranging gives
	\[
	\|v(\cdot,T)-u(\cdot,T)\|_{L^1(\Omega)}
	\;\le\;
	\|v(\cdot,0)-u(\cdot,0)\|_{L^1(\Omega)}
	+ C_{\mathbf f}\|v-u\|_{L^1(\partial\Omega\times(0,T))}
	+\J_{\mathrm{ent}}(v;u),
	\]
	using the Lipschitz continuity of $\mathbf f$. Since $u$ is unknown,
	we replace it by the bounded computable family $V_h^c$; its supremum
	is a discrete doubling-of-variables surrogate for
	$\J_{\mathrm{ent}}(v;u)$.
\end{remark}

\paragraph{A cellwise identity}
Lemma~\ref{lem:sign_chain_rule} gives, for every $K\in\Lambda_h$ and
$v,k_h\in W^{1,\infty}(K)$,
\begin{equation}\label{eq:cellwise_entropy_ibp}
	\begin{aligned}
		\int_{K}\nabla\!\cdot\mathbf{F}(v)\,\sgn(v-k_h)\,\d\z
		=&\ \int_{\partial K}
		\bigl(\mathbf{F}(v)-\mathbf{F}(k_h)\bigr)\sgn(v-k_h)\cdot\mathbf{n}\,\d s\\
		&+\int_{K}\nabla\!\cdot\mathbf{F}(k_h)\,\sgn(v-k_h)\,\d\z.
	\end{aligned}
\end{equation}
For a CPwL function $v$ the identity is understood cellwise after a
conforming simplicial refinement; interior face contributions cancel
because $v-k_h$ is continuous within $K$.

\paragraph{Strategy of the analysis}
With \eqref{eq:NN_min_problem} fixed, by the minimizing property,
$\L(u_{\theta^*})\le\L(u_\theta)$ for every $\theta\in\Theta$.
We therefore bound the minimizer's loss by constructing
a competitor in two steps:
\begin{itemize}[leftmargin=4.2em,itemsep=2pt,topsep=4pt]
	\item[\textbf{Step 1.}] Construct a continuous piecewise linear
	(CPwL) function $\hat u$ with provably small loss $\L(\hat u)$
	(carried out in Sections~\ref{sec:1d}--\ref{sec:multiD} according to
	the wave structure of $u$).
	\item[\textbf{Step 2.}] Use the following lemma to convert $\hat u$
	into a $\tanh$ neural network $u_{\hat\theta}\in\mathcal N$ whose loss
	matches $\L(\hat u)$ up to a controlled error.
\end{itemize}
Combining the two steps produces an admissible competitor in
\eqref{eq:NN_min_problem}, which in turn yields the desired bound on
$\L(u_{\theta^*})$ and hence on $\|u_{\theta^*}-u\|_{L^1}$.

\begin{lemma}[CPwL-to-network loss approximation]
\label{lem:loss_cpwl_approx}
	Let \(0<h\le1\), let \(\mathcal T\) be a finite conforming simplicial
	partition of \(\overline{\Omega_T}\), and set \(N:=\#\mathcal T\).
	Let \(\hat u\in C(\overline{\Omega_T})\) be affine on each simplex,
	and assume that every vertex belongs to at most \(N_{\rm patch}\)
	simplices. Suppose that, for some \(\delta_c>0\) independent of \(h\)
	and some open set $\Omega_s\subset\Omega_T$ with piecewise $C^1$ boundary,
	\begin{align}
		&\|\hat u\|_{L^\infty(\Omega_T)}
		\le \frac c2-\delta_c,
		\qquad
		\|\nabla\hat u\|_{L^1(\Omega_T)}\lesssim1,
		\label{eq:cpwl-boundness-requirement}\\
		&|\Omega_s|\lesssim h^2,
		\qquad
		\mathcal H^d(\partial\Omega_s)\lesssim1,
		\qquad
		\Omega_r:=\Omega_T\setminus\Omega_s .
		\label{eq:assum_Omega_s}
	\end{align}
	Then there exists a clipped \(\tanh\) network
	\(u_{\hat\theta}=\Pi_c(U_{\hat\theta}^{\rm raw})\) of the form
	\eqref{def:clipped-NN} such that
	\[
		\max\!\left\{
		\|u_{\hat\theta}-\hat u\|_{L^\infty(\Omega_T)},
		\|\nabla u_{\hat\theta}-\nabla\hat u\|_{L^1(\Omega_T)}
		\right\}
		\le \min\{h^2,\delta_c/2,1/2\},
	\]
	and
	\[
		\mathcal L(u_{\hat\theta})
		\le \mathcal L(\hat u)
		+C\left(
		h+\int_{\Omega_r}
		|\nabla\!\cdot\mathbf F(\hat u)|\,\d\z
		\right).
	\]
	The network has depth \(O(1+\log N_{\rm patch})\) and width
	\(O(N\,N_{\rm patch})\). The constants may depend on the fixed data,
	\(\delta_c\), and the uniform bounds above, but not on \(h\), \(N\),
	or the network parameters.
\end{lemma}
The proof of Lemma~\ref{lem:loss_cpwl_approx} is provided in
Appendix~\ref{app:nn-approx}.

\section{Error estimate for one-dimensional problems with piecewise smooth solutions}\label{sec:1d}
This section focuses on the one-dimensional scalar HCL~\eqref{eq:scalar-HCL}
with $d=1$. The flux is scalar, so we write $f$ instead of
$\mathbf f$. We consider either a convex flux or a nonconvex flux with finitely many inflection points. The extended one-dimensional analysis uses the following piecewise smoothness assumption.

\begin{assumption}[Finite space--time decomposition into regular-wave and event pieces]
\label{assump:piecewise_smooth}
There exists a finite decomposition of $\Omega_T$ into open, connected pieces $\{\mathcal P_q\}_{q=1}^Q$ with pairwise disjoint interiors and piecewise $C^2$ boundaries, in the sense that \(\Omega_T=\bigcup_{q=1}^Q \overline{\mathcal P_q}.\)
Each piece is one of the following types: a smooth piece, a regular shock piece, a rarefaction/compound-wave piece, a shock-shock interaction piece, or a nondegenerate shock-birth piece. By finiteness of the decomposition, all constants in the assumptions below may be taken uniformly in $q$.
\begin{itemize}
\item[(i)] \emph{Smooth pieces and regular shock pieces.}
If $\mathcal P_q$ contains no rarefaction region and no local event, then either
$u\in W^{2,\infty}(\mathcal P_q)\cap C^2(\mathcal P_q)$, 
or $\mathcal P_q$ contains a single $C^2$ entropy shock curve $\Gamma_q^s=\{(x,t):x=\gamma_q^s(t)\}\subset \mathcal P_q$ 
which divides $\mathcal P_q$ into two open subregions $\mathcal P_{q,-}$ and $\mathcal P_{q,+}$. In the latter case there exist $u_{q,\pm}\in W^{2,\infty}(\mathcal P_q)\cap C^2(\mathcal P_{q,\pm})$
such that $u=u_{q,\pm}$ on $\mathcal P_{q,\pm}$, with
\[
\|u_{q,-}\|_{W^{2,\infty}(\mathcal P_q)}
+\|u_{q,+}\|_{W^{2,\infty}(\mathcal P_q)}
\le C_*,
\]
and the left/right traces satisfy the Rankine-Hugoniot condition and entropy condition \eqref{eq:reformulated_entropy_condition} along $\Gamma_q^s$.

\item[(ii)] \emph{Rarefaction and compound-wave pieces.}
If $\mathcal P_q$ is a rarefaction piece or a compound-wave piece, then the rarefaction part inside $\mathcal P_q$ is the union of finitely many pairwise disjoint connected open sets $\mathcal R_q=\cup_{m=1}^{M_q}\mathcal R_{q,m}$, $M_q\ge 1$. More precisely, there exist numbers $0\le \tau_q<t_q^+\le T$, a point $x_q \in \Omega$, and for each $m=1,\dots,M_q$, two $C^2$ curves $\Gamma_{q,m,i}
=\{(x,t):x=\gamma_{q,m,i}(t),\ t\in[\tau_q,t_{q}^+)\}$, $i=1,2$, such that $\gamma_{q,m,1}(\tau_q)=\gamma_{q,m,2}(\tau_q)=x_q$, $\gamma_{q,m,1}(t)<\gamma_{q,m,2}(t)$, $t\in(\tau_q,t_q^+)$, and such that the open set
\begin{equation*}
\begin{aligned}
&\mathcal R_{q,m}
:= 
\{(x,t)\in\mathcal P_q:\ \tau_q<t<t_{q}^+,\ \gamma_{q,m,1}(t)<x<\gamma_{q,m,2}(t)\}, \\
&c(t-\tau_q)\le  \gamma_{q,m,2}(t)-\gamma_{q,m,1}(t)\le C(t-\tau_q),
\quad t\in[\tau_q,t_{q}^+).
\end{aligned}
\end{equation*}
Inside the rarefaction region one has $u\in C^2(\mathcal R_q)$, and
\[
(t-\tau_q)\,|\nabla u(x,t)|+(t-\tau_q)^2\,|\nabla^2u(x,t)|\le C,
\quad (x,t)\in\mathcal R_q.
\]
On each connected component $\mathcal O$ of $\mathcal P_q\setminus \overline{\mathcal R_q}$, the solution satisfies the regularity assumptions from item~(i). If a shock curve bounds $\mathcal R_q$, then on its non-rarefaction side the corresponding one-sided $W^{2,\infty}$ bound from item~(i) holds.

\item[(iii)] \emph{Regular interaction pieces.}
Each shock-shock merging point is contained in a single event piece $\mathcal P_q$ centered at $(x_\beta,\tau_\beta)$, in which exactly two incoming $C^2$ entropy shock curves and one outgoing $C^2$ entropy shock curve meet, and no other local event occurs. On every connected component of the complement of these three curves, the solution belongs to $W^{2,\infty}\cap C^2$, 
with the same uniform bound $C_*$. Any other regular local interaction (for instance, a shock-rarefaction interaction) is isolated in a single event piece and is assumed to satisfy the same piecewise regularity as in items (i)–(ii).

\item[(iv)] \emph{Nondegenerate shock-birth pieces and uniform regularity away from birth.}
Each shock-formation point is contained in a single event piece $\mathcal P_q$
which contains no local event other than the shock birth at $(x_q,\tau_q)$.
In the neighborhood of $(x_q,\tau_q)$ near the shock curve
$x=\gamma_q^b(t)$, for $x\neq \gamma_q^b(t)$, one has
\[
|\nabla u(x,t)|
= O\!\left(\bigl((t-\tau_q)^3+(x-x_q)^2\bigr)^{-1/3}\right),
\qquad
|\nabla^2 u(x,t)|
= O\!\left(\bigl((t-\tau_q)^3+(x-x_q)^2\bigr)^{-5/6}\right).
\]
Finally, there exists a uniform waiting time $\delta_*>0$ such that the
regular shock piece adjacent to $\gamma_q^b$ starts no earlier than
$t=\tau_q+\delta_*$. Consequently, on that adjacent regular shock piece
the two one-sided states satisfy the uniform $W^{2,\infty}$ bound from
item~\textup{(i)}, with the same constant $C_*$.
\end{itemize}
\end{assumption}
\begin{remark}
Assumption~\ref{assump:piecewise_smooth}(iv) models the generic nondegenerate shock-birth structure for one-dimensional scalar conservation laws; see, for instance, \cite[Thm.~1.1]{huicheng2022shock}.
\end{remark}

As the entropy solution consists of a finite number of elementary waves and shock interaction, the corresponding piecewise smoothness patterns are illustrated in Figure~\ref{fig:smooth-pattern-illustration}. In particular, Figure~\ref{fig:rarefaction-compound-illustration} includes both single rarefaction waves and compound waves (the latter illustrated in the case of flux $f(u)$ with three inflection points). 
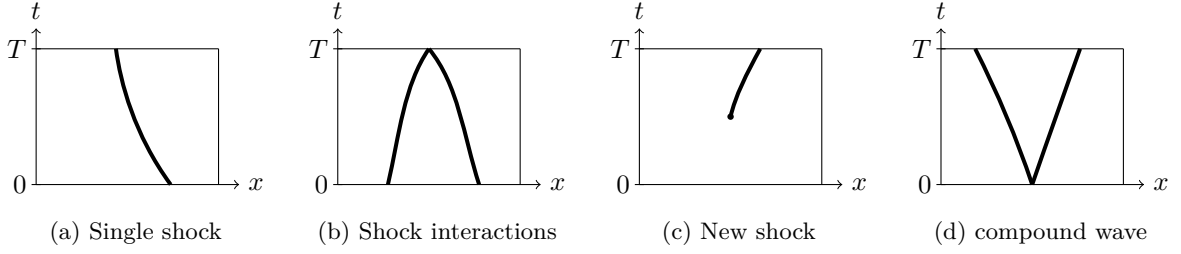
\begin{figure}[ht]
\centering
\begin{subfigure}{0.24\linewidth}
\centering
\begin{tikzpicture}[x=0.7cm,y=0.7cm]
  \xtframe{4}{3}
  \coordinate (S0) at (1.75,3);
  \coordinate (S1) at (1.85,2.25);
  \coordinate (S2) at (2.15,1.10);
  \coordinate (S3) at (2.95,0);
  \draw[shock] (S0) .. controls (S1) and (S2) .. (S3);
\end{tikzpicture}
\caption{Single shock}
\end{subfigure}
\hfill
\begin{subfigure}{0.24\linewidth}
\centering
\begin{tikzpicture}[x=0.7cm,y=0.7cm]
  \xtframe{4}{3}
  \coordinate (A)  at (2.00,3);
  \coordinate (L1) at (1.45,2.20);
  \coordinate (L2) at (1.35,1.10);
  \coordinate (L3) at (1.10,0);
  \draw[shock] (A) .. controls (L1) and (L2) .. (L3);
  \coordinate (R1) at (2.65,2.20);
  \coordinate (R2) at (2.75,1.10);
  \coordinate (R3) at (3.10,0);
  \draw[shock] (A) .. controls (R1) and (R2) .. (R3);
\end{tikzpicture}
\caption{Shock interactions}
\end{subfigure}
\hfill
\begin{subfigure}{0.24\linewidth}
\centering
\begin{tikzpicture}[x=0.7cm,y=0.7cm]
 \xtframe{4}{3}
 \coordinate (C)  at (2.00,1.50);
 \coordinate (S1) at (2.10,1.95);
 \coordinate (S2) at (2.35,2.45);
 \coordinate (S3) at (2.65,3.00);
 \draw[shock] (C) .. controls (S1) and (S2) .. (S3);
 \fill (C) circle (1.2pt); 
\end{tikzpicture}
\caption{Newly formed shock}
\end{subfigure}
\hfill
\begin{subfigure}{0.24\linewidth}
\centering
\begin{tikzpicture}[x=0.7cm,y=0.7cm]
  \xtframe{4}{3}
  \coordinate (B)  at (2.00,0);
  \coordinate (L1) at (1.65,1.05);
  \coordinate (L2) at (1.25,2.0);
  \coordinate (L3) at (0.75,3);
  \draw[shock] (B) .. controls (L1) and (L2) .. (L3);
  \coordinate (R1) at (2.36,1.05);
  \coordinate (R2) at (2.73,2.10);
  \coordinate (R3) at (3.05,3);
  \draw[shock] (B) .. controls (R1) and (R2) .. (R3);
\end{tikzpicture}
\caption{compound wave}
\end{subfigure}
\caption{Illustration of typical piecewise smooth patterns}
\label{fig:smooth-pattern-illustration}
\end{figure}
\begin{figure}[!ht]
\tikzset{rarEdge/.style={char, line width=0.9pt}}

\begin{subfigure}{0.45\linewidth}
\centering
\begin{tikzpicture}[x=0.9cm, y=0.9cm]
  \xtframe{4}{3}
  \coordinate (B) at (2.00,0);

  \coordinate (Ltop) at (1.10,3);
  \coordinate (Rtop) at (3.00,3);

  \draw[rarEdge] (B) -- (Ltop);
  \draw[rarEdge] (B) -- (Rtop);

  \foreach \xTop in {1.3, 1.6, 1.9, 2.2, 2.5,2.80}{
    \draw[char] (B) -- (\xTop,3);
  }

  \foreach \xB/\xT in {0.90/0.00, 1.30/0.40, 1.70/0.80}{
    \draw[char] (\xB,0) -- (\xT,3);
  }

  \foreach \xB/\xT in {2.30/3.30, 2.60/3.60, 2.90/3.90}{
    \draw[char] (\xB,0) -- (\xT,3);
  }
\end{tikzpicture}
\caption{Single rarefaction wave}
\label{fig:rarefaction-illustration-single}
\end{subfigure}
\hfill
\begin{subfigure}{0.45\linewidth}
\centering
\begin{tikzpicture}[x=0.9cm, y=0.9cm]
  \xtframe{4}{3}
  \coordinate (B)  at (2.00,0);
  \coordinate (L1) at (1.65,1.05);
  \coordinate (L2) at (1.25,2.0);
  \coordinate (L3) at (0.75,3);
  \draw[shock] (B) .. controls (L1) and (L2) .. (L3);

  \coordinate (R1) at (2.36,1.05);
  \coordinate (R2) at (2.73,2.10);
  \coordinate (R3) at (3.05,3);
  \draw[shock] (B) .. controls (R1) and (R2) .. (R3);

  \foreach \xTop in {1.05,1.35,1.75, 2.15,2.5,2.85}{
    \draw[char] (B) -- (\xTop,3);
  }
  \draw[char] (0.0, 0) -- (L3);
  \draw[char] (0.7, 0) -- (L2);
  \draw[char] (1.4, 0) -- (L1);

  \draw[char] (4.0, 0) -- (R3);
  \draw[char] (3.4, 0) -- (R2);
  \draw[char] (2.7, 0) -- (R1);
\end{tikzpicture}
\caption{Compound wave}
\label{fig:shock-illustration-compound-wave}
\end{subfigure}
\caption{Thin lines indicate characteristics; thick lines indicate shocks.}
\label{fig:rarefaction-compound-illustration}
\end{figure}
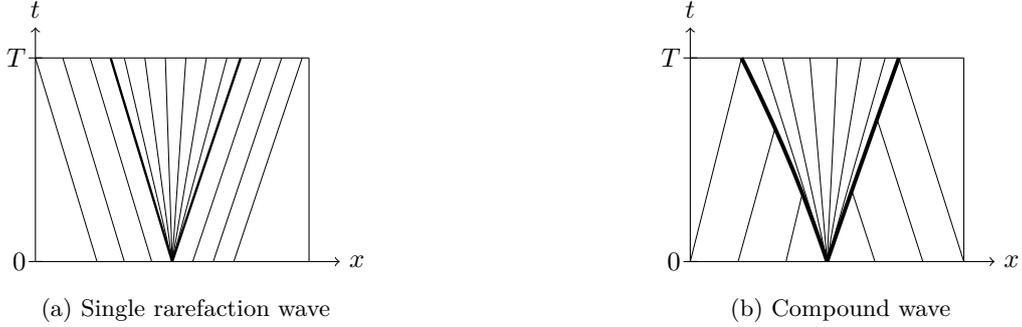

\subsection{Single shock waves and shock interactions}\label{sec:shock_wave}
We first treat a single regular shock under the following assumption.

\begin{assumption}[Single regular shock]
\label{assump:single_shock}
Let $d=1$. There exists a curve $\gamma\in C^2([0,T])$ such that $\gamma([0,T])\subset\Omega$ and the shock curve $\Gamma:=\{(\gamma(t),t):0<t<T\}$ divides $\Omega_T$ into the two open sets
\[
\Omega_T^-:=\{(x,t)\in\Omega_T:x<\gamma(t)\},
\qquad
\Omega_T^+:=\{(x,t)\in\Omega_T:x>\gamma(t)\}.
\]
There exist $u^\pm\in W^{2,\infty}(\Omega_T)\cap C^2(\Omega_T)$ such that
\[
u=u^\pm\quad\text{in }\Omega_T^\pm,
\qquad
\|u^-\|_{W^{2,\infty}(\Omega_T)}
+\|u^+\|_{W^{2,\infty}(\Omega_T)}
\le C_*.
\]
Set $B:=\max_{\pm}\|u^\pm\|_{L^\infty(\Omega_T)}$, $G:=\max_{\pm}\|\nabla u^\pm\|_{L^\infty(\Omega_T)}$
and choose the fixed comparison-class bound $c>3B$. The one-sided traces of $u$ satisfy the Rankine--Hugoniot condition and the entropy condition \eqref{eq:reformulated_entropy_condition} along $\Gamma$.
\end{assumption}

\begin{remark}
For the cellwise interpolant $I_h$ defined in Section~\ref{subsec:notation}, for
all sufficiently small $h$, we have
\begin{equation}\label{eq:interp_Vhc_1d}
\|I_hu^\pm\|_{L^\infty(K)}
+h\|\nabla I_hu^\pm\|_{L^\infty(K)}
\le B+ 2B<c,\qquad K\in\Lambda_h.
\end{equation}
\end{remark}

This section establishes the a priori error estimate for the single-shock case. The proof has two main ingredients. First, we construct a neural network competitor with \(O(h^{-(d+1)})\) neurons whose loss is sufficiently small, which is highly nontrivial.	Indeed, we will convert derivatives in the loss into jumps across the shock, which has favorable signs for estimation due to the entropy condition \eqref{eq:reformulated_entropy_condition}.
Second, using a DPwP approximation \(k_h\) of the entropy solution, we convert this loss bound into an error estimate for the NN minimizer \(u_{\theta^*}\).
\begin{lemma}\label{lem:Lu_shock}
	Let $u$ be the entropy solution of the scalar conservation law \eqref{eq:scalar-HCL} that contains a single shock and satisfies Assumption~\ref{assump:single_shock}. There exists a clipped $\tanh$ neural network $u_{\hat\theta} = \Pi_c\bigl(U_{\hat\theta}^{\rm raw}\bigr)$ of the form \eqref{def:clipped-NN} with depth $O(d+1)$ and at most $O(h^{-(d+1)})$ neurons in each hidden layer, such that
	\begin{equation}\label{eq:Lu_shock}
		\mathscr{L}(\hat{\theta}) \lesssim h.
	\end{equation}
\end{lemma}

\begin{proof}

	Set $\varepsilon:=h^2$. We first construct a continuous piecewise linear function $\hat u$ on a simplicial partition $\mathcal T$ of $\Omega_T$ with $N:=\#\mathcal T = O(h^{-(d+1)})$ such that
	\begin{equation}\label{eq:hat_u_basic}
    \begin{aligned}
         &  \|\hat u\|_{L^1(\partial\Omega\times(0,T))}=0, &&\| \hat u(\cdot,0)-u_{0}(\cdot)\|_{L^1(\Omega)}\lesssim h,\\
         & \int_{\Omega_{T}}|\nabla\!\cdot \mathbf{F}(\hat u)|\,\d \z\le C,
		&&\sup_{k_{h}\in V_{h}^{c}}\J_{\mathrm{ent}}(\hat u;k_{h})\lesssim h,
    \end{aligned}
    \end{equation}
	with $C$ independent of $h$. Then
	\[
	\L(\hat u)=\sup_{k_h\in V_h^c}\J_{\mathrm{ent}}(\hat u;k_h)+\L_{\mathrm{ibc}}(\hat u)+\L_{reg}(\hat u)\lesssim h.
	\]

	\noindent\textbf{Step 1: Construction of $\hat u$.}
	Let $\Gamma:=\{(x,t)\in\Omega_T:\ x=\gamma(t),\ 0\le t \le T\}$ be the shock curve. Assume first that $\Gamma$ does not significantly intersect the mesh
	skeleton $\partial\Lambda_{h}$ in the sense of
	\begin{equation}\label{eq:condition_Gamma}
		\mathcal{H}^1\Big(\big\{(x,t)\in\Gamma:\dist\big((x,t),\partial\Lambda_h\big)< C_{\rm geo} h^2\big\}\Big) \lesssim h,
	\end{equation}
	where $\mathcal{H}^1$ denotes the one-dimensional Hausdorff (arc-length) measure on $\Gamma$, and $C_{\mathrm{geo}}$ is a constant independent of $h$.

    Define a tubular neighborhood $\widetilde\Lambda_{\Gamma}$ of $\Gamma$ of width $O(\varepsilon)$ whose left and right boundaries
	$\widetilde\Gamma^{-},\widetilde\Gamma^{+}$ are piecewise linear, with each segment lying in a single mesh cell
	(cf.\ Figure~\ref{fig:mesh_for_NN}). We triangulate each strip $\widetilde\Lambda_{\Gamma}$ by anisotropic triangles whose tangential
	length is $O(h)$ and normal thickness is $O(\varepsilon)$; hence each such triangle has area
	$O(h\varepsilon)=O(h^3)$ and there are $O(h^{-1})$ triangles. Then we triangulate $\Omega_{T}\setminus \widetilde\Lambda_{\Gamma}$ in a shape-regular way conforming to the triangulation of $\widetilde\Lambda_{\Gamma}$ with mesh size $O(h)$.

	Define $\hat u$ by:
	(i) $\hat u$ equals the nodal $P^1$ interpolant of $u$ on $\Omega_T\setminus \widetilde\Lambda_\Gamma$;
	(ii) on $\widetilde\Gamma^\pm$, $\hat u$ is the nodal $P^1$
interpolant of $u^\pm$; inside $\widetilde\Lambda_\Gamma$,
we extend these boundary values by continuous piecewise
affine interpolation on the strip triangulation.
The boundary interpolants agree with those in (i), so
$\hat u\in L^\infty(\Omega_T)$ is continuous. Moreover,
\[
\|\hat u-u^\pm\|_{L^\infty(\widetilde\Gamma^\pm)}
\lesssim h^2.
\]
	The initial trace $\hat{u}(\cdot,0)$ can be regarded as a piecewise linear approximation
	of the piecewise smooth function $u_0$ with a jump discontinuity.
	In particular, $\hat{u}(\cdot,0)$ varies sharply only within an $O(\varepsilon)$-neighborhood
	of each initial discontinuity. Therefore,
	\[
	\|\hat u(\cdot,0)-u_0(\cdot)\|_{L^1(\Omega)} \lesssim h+\varepsilon \lesssim h.
	\]
	By the assumption that the exact solution $u(\cdot, t)$, $t\in [0,T]$ is compactly supported in $\Omega$, for all sufficiently small $h$, we have $\hat u=0$ on $\partial\Omega\times(0,T)$.
	Consequently, we have the first bound in \eqref{eq:hat_u_basic}.
	Moreover, on $\Omega_T\setminus \widetilde\Lambda_\Gamma$ we have $u\in C^2$ and $\nabla\cdot \mathbf F(u)=0$, so standard interpolation estimates yield
	$\int_{\Omega_T\setminus \widetilde\Lambda_\Gamma}|\nabla\cdot \mathbf F(\hat u)|\lesssim h$.
	Inside $\widetilde\Lambda_\Gamma$, $\hat u$ varies by $O(1)$ across width $O(\varepsilon)$, hence $|\nabla \hat u|\lesssim \varepsilon^{-1}$ and $|\widetilde\Lambda_\Gamma|\lesssim \varepsilon$, which gives
	$\int_{\widetilde\Lambda_\Gamma}|\nabla\cdot \mathbf F(\hat u)|\lesssim 1$.
	Thus $\int_{\Omega_T}|\nabla\cdot \mathbf F(\hat u)|\le C$, proving the second bound in \eqref{eq:hat_u_basic}.
	\begin{figure}[htbp!]
		\centering
		\begin{subfigure}[b]{0.3\textwidth}
\centering
\includegraphics[width=\linewidth]{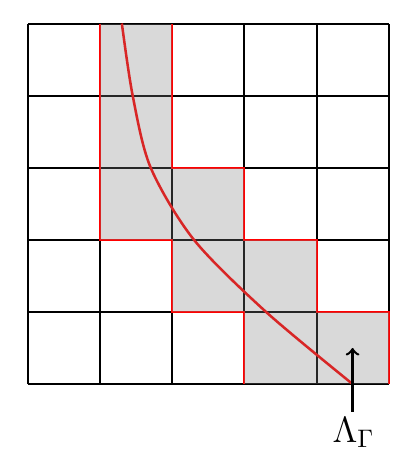}
\caption{Shock and the mesh for DPwP \( k_h \)}
\label{fig:mesh_for_DG}
\end{subfigure}
		\hfill
		\begin{subfigure}[b]{0.29\textwidth}
\centering
\includegraphics[width=\linewidth]{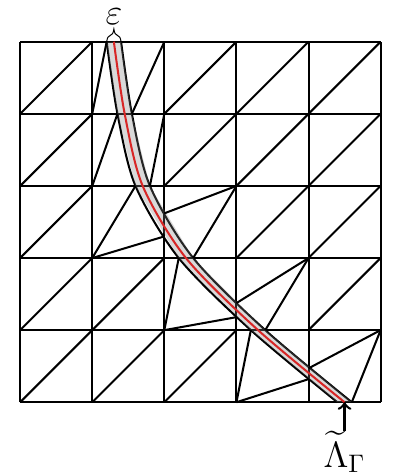}
\caption{Shock and the mesh for CPwL \( \hat{u} \)}
\label{fig:mesh_for_NN}
\end{subfigure}
		\hfill
		\begin{subfigure}[b]{0.29\textwidth}
\centering
\includegraphics[width=\linewidth]{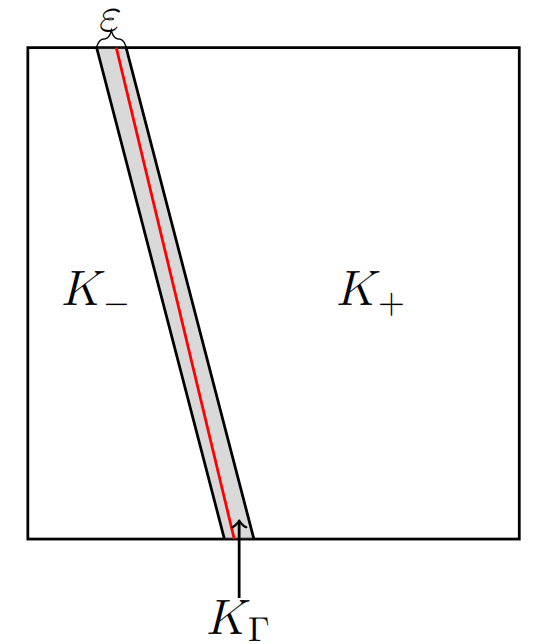}
\caption{A representative cell $K$ in $\Lambda_{\Gamma}$}
\label{fig:K_gamma}
\end{subfigure}
		\caption{Comparison of the partition $\Lambda_h$ and the mesh partition for the CPwL $\hat{u}$.}
	\end{figure}

	\noindent\textbf{Step 2: Estimate of $\J_{\mathrm{ent}}(\hat u;k_{h})$.}
	Fix $k_h\in V_h^c$ and define $\Lambda_\Gamma:=\{K\in\Lambda_h:K\cap \widetilde{\Lambda}_\Gamma\neq\emptyset\}$, as illustrated in Figure~\ref{fig:mesh_for_DG}. Decompose
	\[
	\J_{\mathrm{ent}}(\hat u;k_h)=\sum_{K\in\Lambda_h\setminus \Lambda_\Gamma}\int_K \nabla\cdot \mathbf F(\hat u)\,\sgn\!\left(\hat u-k_h\right)\,\d \z
	+\sum_{K\in\Lambda_\Gamma}\int_K \nabla\cdot \mathbf F(\hat u)\,\sgn\!\left(\hat u-k_h\right)\,\d \z.
	\]

	If $K \notin \Lambda_{\Gamma}$, then $u$ is smooth on $K$ and $\hat{u}$ is
	its linear interpolant. Combining the identity
	$\nabla \!\cdot \mathbf{F}(u) = 0$ with the Lipschitz continuity of the flux $\mathbf{f}$, we obtain
	\begin{equation}\label{eq:away_shock}
		\sum_{K\in\Lambda_{h}\setminus\Lambda_{\Gamma}}\int_{K}\nabla\!\cdot \mathbf{F}(\hat u)\,\sgn\!\left(\hat u-k_{h}\right) \,\d \z
		\le
		\sum_{K\in\Lambda_{h}\setminus\Lambda_{\Gamma}}\int_{K}\big|\nabla\!\cdot \mathbf{F}(\hat u)-\nabla\!\cdot \mathbf{F}(u)\big|\,\d \z
		\lesssim h.
	\end{equation}

	Now let $K\in\Lambda_\Gamma$. By condition \eqref{eq:condition_Gamma}, most of the shock curve $\Gamma$ lies in the interior of mesh cells rather than near the mesh skeleton. In particular, the contribution from those cells that intersect the strip $\widetilde{\Lambda}_\Gamma$ but not $\Gamma$ itself is controlled by \eqref{eq:condition_Gamma} and is therefore small. Hence, we mainly focus on a representative element $K\in\Lambda_\Gamma$, as illustrated in Figure~\ref{fig:K_gamma}. Decompose $K=K_-\cup K_\Gamma\cup K_+$ with $K_\Gamma:=K\cap \widetilde\Lambda_\Gamma$ and $K_\pm:=K\setminus \widetilde\Lambda_\Gamma$ on the two sides. On $K_-\cup K_+$ we use the same interpolation argument. On $K_\Gamma$ we integrate by parts with $k_h|_K\in \mathbb{Q}^1(K)$:
	\begin{equation}\label{eq:decompose_loss_K}
		\begin{aligned}
			&\int_K \nabla \cdot \mathbf{F}(\hat{u}) \sgn\!\left(\hat{u} - k_h\right)\, \d \z
			= \int_{K_- \cup K_+} \nabla \cdot \mathbf{F}(\hat{u}) \sgn\!\left(\hat{u} - k_h\right)\, \d \z  \\
			&\quad\quad\quad\quad\quad+ \int_{\partial K_\Gamma} \left(\mathbf{F}(\hat{u}) - \mathbf{F}(k_h)\right)\sgn\!\left(\hat{u}-k_h\right)\cdot\mathbf{n}\,\d s
			+ \int_{K_\Gamma} \nabla \cdot \mathbf{F}(k_h)\sgn\!\left(\hat{u}-k_h\right)\,\d \z,
		\end{aligned}
	\end{equation}
	where $\mathbf{n}$ denotes the outward unit normal vector on $\partial K_\Gamma$ in the second term on the right-hand side.

	For the first term on the right-hand side of \eqref{eq:decompose_loss_K},
	since $u$ is smooth on $K_- \cup K_+$ and $\hat{u}$ is its linear interpolant,
	by the same argument as in \eqref{eq:away_shock} we deduce that
	\begin{align}\label{eq:1st_K}
		\left|\int_{K_- \cup K_+} \nabla \cdot \mathbf{F}(\hat{u}) \sgn\!\left(\hat{u} - k_h\right)\, \d \z\right|\lesssim |K_- \cup K_+|\cdot h\lesssim h^3.
	\end{align}
	For the third term on the right-hand side of \eqref{eq:decompose_loss_K},
	using the bound $\|k_h\|_{W_h^{1,\infty}} \le c$ together with
	$|K_\Gamma| \lesssim \varepsilon h$, we obtain
	\begin{align}\label{eq:3rd_K}
		\left|\int_{K_\Gamma}\nabla\cdot \mathbf F(k_h)\,\sgn\!\left(\hat u-k_h\right)\,\d \z\right|\lesssim \varepsilon.
	\end{align}

	The next and most important step is to estimate the second term on the right-hand side
	of \eqref{eq:decompose_loss_K}, which involves the boundary integral of
	\(\left(\mathbf{F}(\hat{u}) - \mathbf{F}(k_h)\right)
	\operatorname{sgn}\!\left(\hat{u} - k_h\right)\!\cdot\!\mathbf{n}\).
	As shown in Figure~\ref{fig:K_gamma}, the region $K_{\Gamma}$ is a very thin strip
	with upper and lower boundaries whose measures are of size $O(\varepsilon)$.
	Hence, the dominant contribution of the boundary integral comes from
	the two lateral sides of $K_{\Gamma}$, denoted by
	\((\partial K_\Gamma)_-\) and \((\partial K_\Gamma)_+\).

Now we replace the lateral boundary integral by the corresponding shock-jump integral. Recall from \eqref{eq:def_Q} that
    \[
    \mathcal Q(a,k) =\bigl(\mathbf F(a)- \mathbf F(k)\bigr)\sgn(a-k).
    \]
    By \eqref{eq:Lipschitz_Q}, the map \(\mathcal Q\) is Lipschitz on the fixed bounded range under consideration.

	Let
	\[
	B_K:=
	\bigl\{
	z\in\Gamma:\ \dist(z,\partial K)<C_{\rm geo} h^2
	\bigr\}.
	\]
	For $z\in(\Gamma\cap K)\setminus B_K$, the two lateral sides $(\partial K_\Gamma)_\pm$
	can be parameterized by $\Gamma\cap K$ in a one-to-one manner. Let
	$z_\pm\in(\partial K_\Gamma)_\pm$ denote the points corresponding to $z$.
	Then
	\[
	|z_\pm-z|\lesssim \varepsilon,
	\qquad
	\mathbf n_\pm(z_\pm)=\pm \mathbf n_\Gamma(z)+O(h).
	\]
	where $\mathbf n_\pm$ are the outward normals of $K_\Gamma$ on
	$(\partial K_\Gamma)_\pm$. Moreover, since $\Gamma$ is $C^2$, the discrepancy between the line element on
	$(\partial K_\Gamma)_\pm$ and the arc-length element on $\Gamma$ contributes only
	an $O(h^2)$ error.
	Furthermore, using $\|\nabla k_h\|_{L^\infty(K)}\lesssim h^{-1}$,
	\[
	|k_h(z_\pm)-k_h(z)|\lesssim \varepsilon h^{-1}.
	\]
	Together with the interpolation and normal-offset estimate for $\hat u$, this gives
	\[
	\bigl|
	\mathcal Q(\hat u(z_\pm),k_h(z_\pm))
	-
	\mathcal Q(u^\pm(z),k_h(z))
	\bigr|\lesssim |\hat u(z_\pm)-u^\pm(z)|+|k_h(z_\pm)-k_h(z)|
	\lesssim \varepsilon h^{-1}.
	\]
	Since the length of $\Gamma\cap K$ is bounded by $O(h)$, the accumulated error from this
	part is $O(\varepsilon)$. The arc-length and normal-vector perturbations contribute
	$O(h^2)$. The above correspondence may break down only for points $z\in \Gamma\cap K$ lying in $B_K$,
	i.e., where $\Gamma$ comes within distance $Ch^2$ of $\partial K$. Since the integrand is
	uniformly bounded, the contribution of the corresponding exceptional portions is bounded by
	$C\mathcal H^1(B_K)$. Consequently,
	\begin{align}\label{eq:2nd_K}
		&\int_{(\partial K_\Gamma)_-\cup(\partial K_\Gamma)_+}
		\bigl(\mathbf F(\hat u)-\mathbf F(k_h)\bigr)
		\sgn(\hat u-k_h)\cdot\mathbf n\,\d s \notag\\
        &\quad=
		\int_{\Gamma\cap K}
		\llbracket
		\bigl(\mathbf F(u)-\mathbf F(k_h)\bigr)\sgn(u-k_h)
		\rrbracket\cdot\mathbf n_\Gamma\,\d s
		+\mathcal E_K,
	\end{align}
	with
	\begin{align*}
		|\mathcal E_K|
		\le
		C\left(
		h^2+\varepsilon+
		\mathcal H^1\Bigl(
		\bigl\{z\in\Gamma:\dist(z,\partial K)<C_{\rm geo} h^2\bigr\}
		\Bigr)
		\right).
		\notag
	\end{align*}
	The integral in \eqref{eq:2nd_K} is non-positive by the reformulated
	entropy condition \eqref{eq:reformulated_entropy_condition}, applied pointwise
	with $k=k_h(z)$.

	Therefore, by combining \eqref{eq:decompose_loss_K}-\eqref{eq:2nd_K} we obtain for all $K\in\Lambda_\Gamma$:
	\begin{align}\label{eq:Lu_shock_K}
		\int_K \nabla\cdot \mathbf F(\hat u)\,\sgn\!\left(\hat u-k_h\right)\,\d \z\lesssim h^2+\varepsilon+\mathcal{H}^1\Big(\big\{\z\in\Gamma:\dist\big(\z,\partial K\big)< C_{\rm geo} h^2\big\}\Big).
	\end{align}

The configuration shown in Figure~\ref{fig:K_gamma} is only representative. For every \(K\in\Lambda_\Gamma\), the portions of the transition layer whose fixed-time projections lie outside \(B_K\) admit the same two-sided parametrization and hence the same entropy-jump argument, whereas every unpaired or spill portion must project into \(B_K\), since the segment joining it to \(\Gamma\) crosses \(\partial K\) and has length \(O(\varepsilon)=O(h^2)\). By Fubini's theorem and \(|\nabla\hat u|\lesssim\varepsilon^{-1}\), the latter portions contribute at most \(C(\varepsilon+\mathcal H^1(B_K))\);  therefore \eqref{eq:Lu_shock_K} holds for all \(K\in\Lambda_\Gamma\).

	Since $\#\Lambda_\Gamma\lesssim h^{-1}$ (for a single $C^2$ shock curve in space--time domain), summing \eqref{eq:Lu_shock_K} yields
	\begin{align}\label{eq:Lu_shock_Lambda_Gamma}
		\sum_{K\in \Lambda_{\Gamma}}\int_{K}\nabla \cdot \mathbf{F}(\hat{u})\sgn\!\left(\hat{u}-k_h\right)\,\d \z
		\lesssim h+\frac{\varepsilon}{h}+\mathcal{H}^1\Big(\big\{\z\in\Gamma:\dist\big(\z,\partial \Lambda_{h}\big)< C_{\rm geo} h^2\big\}\Big).
	\end{align}
	With $\varepsilon=h^2$, \eqref{eq:condition_Gamma} and \eqref{eq:away_shock}, we conclude $\J_{\mathrm{ent}}(\hat u;k_h)\lesssim h$ uniformly in $k_h\in V_h^c$, proving the third bound in \eqref{eq:hat_u_basic}.

	By identifying the transition region \(\widetilde{\Lambda}_\Gamma\) as the singular region, it remains to verify the \(L^1\)-bound on the gradient required in Lemma~\ref{lem:loss_cpwl_approx}. On \(\Omega_T\setminus\widetilde{\Lambda}_\Gamma\), \(\hat u\) is the nodal interpolant of a piecewise \(C^2\) function on a mesh of size \(O(h)\), and hence \(|\nabla \hat u|\le C\) there. Inside \(\widetilde{\Lambda}_\Gamma\), \(\hat u\) varies by \(O(1)\) across a layer of thickness \(O(\varepsilon)=O(h^2)\), so \(|\nabla \hat u|\lesssim \varepsilon^{-1}\), while \(|\widetilde{\Lambda}_\Gamma|\lesssim \varepsilon\). Moreover, by the standing choice of \(c\), the CPwL function \(\hat u\) constructed above satisfies
	\[
	\|\hat u\|_{L^\infty(\Omega_T)}
	\le \frac c2-\delta_c
	\]
	for some \(\delta_c>0\) independent of \(h\). In addition, the gradient bound
	\[
	\|\nabla\hat u\|_{L^1(\Omega_T)}\le C
	\]
	verifies the corresponding assumption in Lemma~\ref{lem:loss_cpwl_approx}.
	From the smooth-region interpolation estimate established above, we have
	\[
	\int_{\Omega_T \setminus \widetilde{\Lambda}_\Gamma}
	\big|\nabla\!\cdot\mathbf{F}(\hat{u})\big|\, \mathrm{d}\boldsymbol{z}
	\lesssim h,
	\]
	which leads to
	\(\mathscr{L}(\hat{\theta}) = \mathcal{L}(u_{\hat{\theta}}) \le \mathcal{L}(\hat{u}) + C h \lesssim h,\)
	thus proving \eqref{eq:Lu_shock}.

	From the construction of the piecewise linear function $\hat{u}$,
	its mesh is obtained by a direct simplicial refinement of a structured canonical grid. The corresponding patch complexity satisfies
	\(N_{\text{patch}} = O(2^{d+1})\),
	and the total number of mesh elements is \(O(h^{-(d+1)})\).
	Hence, by Lemma~\ref{lem:loss_cpwl_approx},
	there exists a neural network of depth \(O(d+1)\)
	with at most \(O(h^{-(d+1)})\) neurons in each hidden layer.

\noindent\textbf{Step 3: Treatment of the general geometric case.} If condition~\eqref{eq:condition_Gamma} fails, choose \[ M\simeq h^{-1}, \qquad \delta_j=j\kappa h^2,\quad j=0,\ldots,M-1, \] where $\kappa>2C_{\rm geo}$ is fixed and $M$ is chosen so that $\delta_{M-1}\le\Delta x_1$. In particular, $0\le\delta_j\lesssim h$. Let \[ \Gamma_{\delta_j} =\{(\gamma(t)-\delta_j,t):0\le t\le T\}. \] For each fixed $t$, the uniform spacing of the Cartesian grid and of the shifts implies that only $O(1)$ indices $j$ satisfy \[ \dist\bigl((\gamma(t)-\delta_j,t), \text{vertical mesh faces}\bigr)<C_{\rm geo}h^2. \] Since arc length on each $\Gamma_{\delta_j}$ is uniformly comparable with $\d t$, integration in $t$ gives \[ \sum_{j=0}^{M-1} \mathcal H^1\!\left( \left\{\z\in\Gamma_{\delta_j}: \dist(\z,\text{vertical mesh faces})<C_{\rm geo}h^2 \right\}\right) \lesssim1. \] The pigeonhole principle therefore yields an index $j_*$ for which the corresponding length is $O(M^{-1})=O(h)$. The horizontal mesh faces also contribute only $O(h)$, since there are $O(h^{-1})$ such faces and the portion of the curve lying within distance $C_{\rm geo}h^2$ of each face has length $O(h^2)$. Consequently, $\Gamma_{\delta_{j_*}}$ satisfies condition~\eqref{eq:condition_Gamma}. Set $\delta:=\delta_{j_*}$ and define \[ u_\delta(x,t):=u(x+\delta,t), \qquad u_{\delta,0}(x):=u_0(x+\delta). \] The preceding construction applied to $u_\delta$ yields a network $u_{\theta_\delta}$ whose loss with respect to $u_{\delta,0}$ is $O(h)$. Moreover, since the zero extension of $u_0$ belongs to $BV(\mathbb R)$ and $|\delta|\lesssim h$, \[ \|u_{\delta,0}-u_0\|_{L^1(\Omega)} \lesssim|\delta|\lesssim h. \] Replacing the translated initial datum by $u_0$ therefore changes only the initial-data term by $O(h)$. Hence \( \mathcal L(u_{\theta_\delta})\lesssim h, \) which completes the proof.
\end{proof}
\begin{remark}[Parameter count]\label{rem:parameter_count}
For the sparse constructive realization used here, bounded
patch complexity and $N=O(h^{-(d+1)})$ give depth $O(d+1)$,
width $O(N)$, and $O(N)$ nonzero weights and biases.
Thus the neuron and nonzero-parameter counts have the
same asymptotic order for this construction. The same scaling applies to the rarefaction, compound-wave, and interaction constructions below.
\end{remark}

Then by taking $k_h$ as a DPwP approximation of the entropy solution, we derive an error estimate for the NN minimizer $u_{\theta^*}$.
\begin{theorem}[Single shock]\label{thm:shock}
	Let $u$ be the entropy solution of the scalar conservation law \eqref{eq:scalar-HCL} that has discontinuous initial data $u_0$, contains a single shock, and satisfies Assumption~\ref{assump:single_shock}. Let $h$ be the mesh size of the background mesh $\Lambda_h$.  Let \(u_{\theta^*}\) be a clipped \(\tanh\) neural network of the form \eqref{def:clipped-NN} with depth $O(d+1)$ and at most $O(h^{-(d+1)})$ neurons in each hidden layer, where $\theta^*$ is a minimizer of the minimization problem \eqref{eq:NN_min_problem} over this class of networks. The network class is chosen to contain the competitor constructed in Lemma~\ref{lem:Lu_shock}. Then there exists a constant $C > 0$, independent of $h$, such that
	\begin{equation}\label{eq:shock_error}
		\|u_{\theta^*}(\cdot, T) - u(\cdot, T)\|_{L^1(\Omega)}\leq C h^{1/2}.
	\end{equation}
\end{theorem}

\begin{proof}
	By Lemma~\ref{lem:Lu_shock} there exists $\hat\theta$ (in the same network class) such that $\mathscr{L}(\hat\theta)\lesssim h$.
	Since $\theta^{*}$ minimizes $\mathscr{L}$, we have
	\begin{equation}\label{eq:min_loss}
		\J_{\mathrm{ent}}(u_{\theta^{*}};k_{h})\le \mathscr{L}(\theta^{*})\le \mathscr{L}(\hat\theta)\lesssim h,
		\quad \forall k_{h}\in V_{h}^{c}.
	\end{equation}
	We now choose a particular $k_h\in V_h^c$ (denoted $\tilde u_h$) adapted to the shock.
	\begin{figure}[htbp!]
		\centering
		\includegraphics[width=0.8\linewidth]{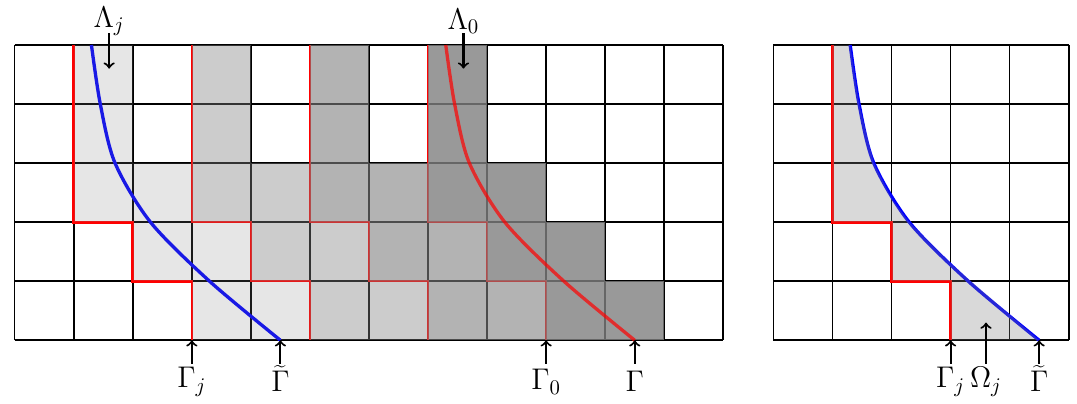}
		\caption{Construction of the DPwP \( \tilde{u}_h \) based on the extended exact solution. Left: The strategy for extending the exact solution. Right: The local cells near $\widetilde{\Gamma}_j$.}
		\label{fig:DPwP_construction}
	\end{figure}

	\noindent\textbf{Step 1: A shifted strip with small $\int |\nabla\cdot\mathbf F(u_{\theta^*})|$.}
	Let $\Lambda_0:=\cup_{K\in\Lambda_\Gamma}\overline K$ be the $O(h)$-width strip of elements intersecting $\widetilde\Lambda_\Gamma$, and let $\Gamma_0$ denote its left boundary (see Figure \ref{fig:DPwP_construction}).
	To control $\int_{\Lambda_0}|\nabla\cdot\mathbf F(u_{\theta^*})|$ we use a strip-selection argument.

	Choose \(L\simeq h\), an integer multiple of \(\Delta x_1\), larger than the maximum width of the spatial sections of \(\Lambda_0\).
	Define the shifted strips $\Lambda_j$ ($j = 1, \dots, M$) by shifting $\Lambda_0$ to the left by $jL$. Let $M\simeq h^{-1/2}$ and denote by $\Gamma_j$ the left boundary of $\Lambda_j$. Let $\widetilde\Gamma:=\left\{(x,t)\in\Omega_T:\ x=\gamma(t)-jL,\ 0\le t \le T\right\}$ be the shifted copy of the shock curve lying inside $\Lambda_j$ (see Figure~\ref{fig:DPwP_construction}).
	Since the strips are disjoint, Lemma~\ref{lem:nonnegativity} yields
	\[
	\sum_{j=1}^M \int_{\Lambda_j}\big|\nabla\cdot\mathbf F(u_{\theta^*})\big|\,\d \z
	\le \int_{\Omega_T}\big|\nabla\cdot\mathbf F(u_{\theta^*})\big|\,\d \z=h^{-1}\L_{reg}(u_{\theta^*})
	\le h^{-1}\L(u_{\theta^*}) \le C.
	\]
	Hence there exists an index $j\in\{1,\dots,M\}$ such that
	\begin{equation}\label{eq:condition_u_theta}
		\int_{\Lambda_j} \left| \nabla \cdot \mathbf{F}(u_{\theta^*}) \right|\, \d \z \leq \frac{C}{M}.
	\end{equation}
	In particular, the region $\Omega_j\subset \Lambda_j$ between $\Gamma_j$ and $\widetilde\Gamma$ satisfies
	$\int_{\Omega_j}|\nabla\cdot\mathbf F(u_{\theta^*})|\lesssim M^{-1}$.

	\noindent\textbf{Step 2: Construction of the DPwP test function $\tilde u_h$.}
	Since $\Lambda_h$ does not align with $\Gamma$, we define $\tilde u_h$ to have its single discontinuity along $\Gamma_j$.
	Let $u_{\rm ext}^+$ be a smooth extension of the right trace of $u$ to the right of $\Gamma_j$ (up to $\Gamma_j$), agreeing with $u$ on the true right side of $\Gamma$.
	Define $\tilde u_h$ as the (elementwise) $\mathbb{Q}^1$ Lagrange interpolant of $u$ on the left of $\Gamma_j$ and of $u_{\rm ext}^+$ on the right of $\Gamma_j$. By construction $\tilde u_h\in V_h^c$ and is continuous across all mesh faces except those lying on $\Gamma_j$.

	Taking $k_h=\tilde u_h$ in \eqref{eq:min_loss} and using the elementwise integration-by-parts identity \eqref{eq:cellwise_entropy_ibp} yields
	\begin{equation}\label{eq:error_estimate_shock}
\begin{aligned}
\sum_{K\in\Lambda_h}\Bigl(
\int_{\partial K}\bigl(\mathbf F(u_{\theta^*})-\mathbf F(\tilde u_h)\bigr)
\sgn(u_{\theta^*}-\tilde u_h)\cdot\mathbf n\,\d s+
\int_K\nabla\!\cdot\mathbf F(\tilde u_h)
\sgn(u_{\theta^*}-\tilde u_h)\,\d\z\Bigr)\lesssim h.
\end{aligned}
\end{equation}

	\noindent\textbf{Step 3: estimate of the volume term.}
	Using $|\sgn|\le 1$,
	\[
	\sum_{K\in \Lambda_{h}} \left|\int_{K}\nabla\cdot\mathbf F(\tilde u_h)\,\sgn\!\left(u_{\theta^*}-\tilde u_h\right)\,\d\z\right|
	\le \sum_{K\in \Lambda_{h}} \int_{K}\big|\nabla\cdot\mathbf F(\tilde u_h)\big|\,\d\z.
	\]
	Let $D_j:=\bigcup_{i=0}^j \Lambda_i$, which satisfies $|D_j|\lesssim jh \lesssim Mh$.
	On $\Omega_T\setminus D_j$, $\tilde u_h$ is the $\mathbb{Q}^1$ interpolant of the exact smooth solution $u$, hence
	$\int_{\Omega_T\setminus D_j}|\nabla\cdot\mathbf F(\tilde u_h)|\,\d\z\lesssim h$ as in Lemma~\ref{lem:Lu_shock}.
	On $D_j$, $\tilde u_h$ is the $\mathbb{Q}^1$ interpolant of a smooth bounded function ($u_{\rm ext}^+$), so
	$|\nabla\cdot\mathbf F(\tilde u_h)|\lesssim 1$ there and thus $\int_{D_j}|\nabla\cdot\mathbf F(\tilde u_h)|\,\d\z\lesssim |D_j|\lesssim Mh$.
	Therefore,
	\begin{equation}\label{eq:estimate_2nd_term}
		\left|\int_{\Omega_T}\nabla\cdot\mathbf F(\tilde u_h)\,\sgn\!\left(u_{\theta^*}-\tilde u_h\right)\,\d\z\right|
		\;\lesssim\; h + Mh.
	\end{equation}

	\noindent\textbf{Step 4: Reformulation of the boundary term.}
	Since $\tilde u_h$ is continuous across all faces of $\partial\Lambda_h$ except those on $\Gamma_j$, the mesh-skeleton contributions cancel pairwise. The uncancelled contributions are the final-time face, the
	initial-time face, the lateral boundary $\partial\Omega\times(0,T)$, and the jump interface
	$\Gamma_j$. Then
	\begin{equation}\label{eq:1st_term_reformulated}
		\begin{aligned}
			&\sum_{K\in \Lambda_{h}}\int_{\partial K}\left(\mathbf{F}(u_{\theta^{*}}) - \mathbf{F}(\tilde{u}_h)\right)\sgn(u_{\theta^{*}}-\tilde{u}_h) \cdot\mathbf{n}\,\d s\\
			=&\int_{\Omega}|u_{\theta^{*}}-\tilde{u}_h|(x, T)\,\d x - \int_{\Omega}|u_{\theta^{*}}-\tilde{u}_h|(x, 0)\,\d x \\
			 &+ \int_{\partial\Omega\times(0,T)}
			\left(\mathbf f(u_{\theta^*})-\mathbf f(\tilde u_h)\right)
			\sgn(u_{\theta^*}-\tilde u_h)\cdot\mathbf n_{\partial\Omega}\,\d s\\
            &-\int_{\Gamma_{j}}\llbracket \left(\mathbf{F}(\tilde{u}_h) - \mathbf{F}(u_{\theta^{*}}) \right)\sgn(\tilde{u}_h - u_{\theta^{*}}) \rrbracket\cdot\mathbf{n}_{\Gamma_j}\,\d s.
		\end{aligned}
	\end{equation}
	By the Lipschitz continuity of $\mathbf f$ and the construction of $\tilde u_h$ on the boundary,
	\begin{equation}\label{eq:lateral_boundary_bound}
\begin{aligned}
&\left|\int_{\partial\Omega\times(0,T)}
\bigl(\mathbf f(u_{\theta^*})-\mathbf f(\tilde u_h)\bigr)
\sgn(u_{\theta^*}-\tilde u_h)\cdot\mathbf n_{\partial\Omega}\,\d s\right|\\
&\qquad\le C\|u_{\theta^*}\|_{L^1(\partial\Omega\times(0,T))}
\le C\mathcal L(u_{\theta^*})\lesssim h.
\end{aligned}
\end{equation}

	\noindent\textbf{Step 5: Control of the jump term on $\Gamma_j$ via entropy condition.}
	Extend $\tilde u_h$ from the left of $\Gamma_j$ across $\Lambda_j$ by interpolating the exact solution $u$; denote this extension by $\tilde u_h^-$. Denote the original $\tilde u_h$ by $\tilde u_h^+$.
	Let $\Omega_j\subset\Lambda_j$ be the region between $\Gamma_j$ and $\widetilde\Gamma$. Integration by parts gives
	\begin{equation}\label{eq:reformulate_int_Gamma_j}
		\begin{aligned}
			& \int_{\Gamma_j}\llbracket \left(\mathbf{F}(\tilde{u}_h) - \mathbf{F}(u_{\theta^{*}}) \right)\sgn(\tilde{u}_h - u_{\theta^{*}}) \rrbracket\cdot\mathbf{n}_{\Gamma_j}\,\d s \\
			= & \int_{\widetilde{\Gamma}} \llbracket \left(\mathbf{F}(\tilde{u}_h) - \mathbf{F}(u_{\theta^{*}}) \right)\sgn(\tilde{u}_h - u_{\theta^{*}}) \rrbracket \cdot\mathbf{n}_{\widetilde{\Gamma}}\,\d s + \mathcal{R}, \quad
			|\mathcal{R}| \leq \mathcal{R}_1 + \mathcal{R}_2, \\
			\mathcal{R}_1 = & \int_{\Omega_j} \left( 2\,|\nabla \cdot \mathbf{F}(u_{\theta^*})| + |\nabla \cdot \mathbf{F}(\tilde{u}_h^-)| + |\nabla \cdot \mathbf{F}(\tilde{u}_h^+)| \right) \d \z, \\
			\mathcal{R}_2 = & \int_{\partial \Omega_j \setminus (\Gamma_j \cup \widetilde{\Gamma})}
			\big| \llbracket \left(\mathbf{F}(u_{\theta^{*}}) - \mathbf{F}(\tilde{u}_h)\right)\sgn(u_{\theta^{*}}-\tilde{u}_h) \rrbracket\big|\,\d s.
		\end{aligned}
	\end{equation}

	For $(x,t)\in\widetilde{\Gamma}$, we have $(x+jL,t)\in\Gamma$,
	where $L=O(h)$ and $j\le M$.
	Using the smoothness of $u^\pm$ together with the interpolation error estimate, we obtain
	\begin{align}\label{eq:oneside_smoothness}
		| \tilde{u}_h^-(x,t) - u^-(x + jL, t) | \lesssim Mh, \quad
		| \tilde{u}_h^+(x,t) - u^+(x + jL, t) | \lesssim Mh.
	\end{align}
	Since $\widetilde{\Gamma}$ is a translated copy of $\Gamma$, for each $(x,t)\in \widetilde{\Gamma}$ we associate the point $(x+jL,t)\in\Gamma$. Moreover, the corresponding unit normals agree:
	\[
	\mathbf{n}_{\widetilde{\Gamma}}(x,t)=\mathbf{n}_{\Gamma}(x+jL,t).
	\]
	By \eqref{eq:oneside_smoothness}, we have
	\[
	|\tilde{u}_h^\pm(x,t)-u^\pm(x+jL,t)|\lesssim Mh.
	\]
	Recalling the definition of $\mathcal Q$ in \eqref{eq:def_Q},
	and using the Lipschitz continuity of $\mathcal Q$ for the first variable on bounded sets in \eqref{eq:Lipschitz_Q}, we obtain
	\[
	\bigl|
	\mathcal Q(\tilde{u}_h^\pm(x,t),u_{\theta^*}(x,t))
	-
	\mathcal Q(u^\pm(x+jL,t),u_{\theta^*}(x,t))
	\bigr|
	\lesssim Mh .
	\]
Therefore,
\[
\begin{aligned}
&\llbracket
\mathcal Q(\tilde u_h,u_{\theta^*})
\rrbracket(x,t)\cdot\mathbf n_{\widetilde\Gamma}(x,t) \\
= &
\Bigl[
\mathcal Q(u^+(x+jL,t),u_{\theta^*}(x,t)) -
\mathcal Q(u^-(x+jL,t),u_{\theta^*}(x,t))
\Bigr]\cdot\mathbf n_\Gamma(x+jL,t)
+O(Mh)
\le CMh,
\end{aligned}
\]
uniformly on $\widetilde\Gamma$.
The inequality follows from
\eqref{eq:reformulated_entropy_condition} at $(x+jL,t)$
with $k=u_{\theta^*}(x,t)$.
Integrating over $\widetilde\Gamma$, we obtain
	\begin{equation}\label{eq:estimate_int_Gamma}
		\int_{\widetilde{\Gamma}} \llbracket ( \mathbf{F}(\tilde{u}_h) - \mathbf{F}(u_{\theta^*}) ) \sgn\!\left(\tilde{u}_h - u_{\theta^*}\right) \rrbracket \cdot \mathbf{n}_{\widetilde{\Gamma}}\,\d s \lesssim M h.
	\end{equation}
	The cellwise gradients of $\tilde u_h^\pm$ are uniformly
bounded. Together with \eqref{eq:condition_u_theta} and
$|\Omega_j|\lesssim h$, this gives
	\begin{equation}\label{eq:estimate_R_1}
		\mathcal{R}_1 \lesssim \int_{\Omega_j}|\nabla \cdot \mathbf{F}(u_{\theta^*})|\d \z+|\Omega_j|\lesssim h + M^{-1}.
	\end{equation}
	For $\mathcal R_2$ in \eqref{eq:reformulate_int_Gamma_j}, since $u_{\theta^*}$ is continuous across the interface, we have
	\[
	\llbracket \left(\mathbf{F}(u_{\theta^{*}}) - \mathbf{F}(\tilde{u}_h)\right)\sgn(u_{\theta^{*}}-\tilde{u}_h) \rrbracket=
	\mathcal Q(u_{\theta^*},\tilde u_h^+)-\mathcal Q(u_{\theta^*},\tilde u_h^-).
	\]
	Moreover, by Lipschitz continuity of $\mathcal Q$ in \eqref{eq:Lipschitz_Q}, we have:
	\begin{align}\label{eq:estimate_jump}
		\bigl|\llbracket (\mathbf F(u_{\theta^*})-\mathbf F(\tilde u_h))
		\sgn(u_{\theta^*}-\tilde u_h)\rrbracket\bigr|
		&=
		\bigl|\mathcal Q(u_{\theta^*},\tilde u_h^+)-\mathcal Q(u_{\theta^*},\tilde u_h^-)\bigr| \lesssim |\tilde u_h^+-\tilde u_h^-|
		=|\llbracket \tilde u_h\rrbracket|.
	\end{align}

	Consequently, using the $L^\infty$ boundedness of $\tilde{u}_h^\pm$, and the length $|\partial \Omega_j \setminus (\Gamma_j \cup \widetilde{\Gamma})|\lesssim h$, we obtain
	\begin{equation}\label{eq:estimate_R_2}
		\mathcal{R}_2 \lesssim \int_{\partial \Omega_j \setminus (\Gamma_j \cup \widetilde{\Gamma})}
		|\llbracket \tilde{u}_h \rrbracket|\,\d s\lesssim h.
	\end{equation}
	Combining \eqref{eq:reformulate_int_Gamma_j}-\eqref{eq:estimate_R_2} yields
	\begin{equation}\label{eq:estimate_int_Gamma_new}
		\int_{\Gamma_j} \llbracket ( \mathbf{F}(\tilde{u}_h) - \mathbf{F}(u_{\theta^*}) ) \sgn\!\left(\tilde{u}_h - u_{\theta^*}\right) \rrbracket \cdot \mathbf{n}_{\Gamma_j}\,\d s
		\lesssim h + M^{-1} + M h.
	\end{equation}

	\noindent\textbf{Step 6: Conclusion and choice of $M$.}
	Combining  \eqref{eq:error_estimate_shock}, \eqref{eq:estimate_2nd_term}, \eqref{eq:1st_term_reformulated}, \eqref{eq:lateral_boundary_bound}, and  \eqref{eq:estimate_int_Gamma_new} gives
	\[
	\|u_{\theta^*}(\cdot, T) - \tilde{u}_h(\cdot, T)\|_{L^1(\Omega)}
	\le \|u_{\theta^*}(\cdot, 0) - \tilde{u}_h(\cdot, 0)\|_{L^1(\Omega)}
	+ C\big(h + M^{-1} + Mh\big).
	\]
	Choose $M\sim h^{-1/2}$ to balance $Mh$ and $M^{-1}$, so that
	\[
	\|u_{\theta^*}(\cdot, T) - \tilde{u}_h(\cdot, T)\|_{L^1(\Omega)}
	\le \|u_{\theta^*}(\cdot, 0) - \tilde{u}_h(\cdot, 0)\|_{L^1(\Omega)} + C h^{1/2}.
	\]
	Finally, the initial term is controlled by the boundary penalty and the fact that $\tilde u_h$ differs from $u_0$ only through a shift of the discontinuity by $O(Mh)$:
	\[
	\|u_{\theta^*}(\cdot, 0)-\tilde u_h(\cdot, 0)\|_{L^1(\Omega)}
	\le \|u_{\theta^*}(\cdot, 0)-u_0(\cdot)\|_{L^1(\Omega)} + \|u_0-\tilde u_h\|_{L^1(\Omega)}
	\lesssim h + Mh \lesssim h^{1/2},
	\]
	where we used $\|u_{\theta^*}(\cdot,0)-u_0(\cdot)\|_{L^1(\Omega)}\leq\L_{\mathrm{ibc}}(u_{\theta^*})\le \mathscr{L}(\theta^*)\lesssim h$.
	Moreover, $\tilde u_h$ approximates $u$ well away from an $O(Mh)$-neighborhood of the shock, and the interpolation error there is of order $O(h)$, so $\|u(\cdot, T)-\tilde u_h(\cdot, T)\|_{L^1(\Omega)}\lesssim Mh + h\lesssim h^{1/2}$.
	By the triangle inequality,
	\[
	\|u(\cdot, T)-u_{\theta^*}(\cdot, T)\|_{L^1(\Omega)}
	\le \|u(\cdot, T)-\tilde u_h(\cdot, T)\|_{L^1(\Omega)} + \|\tilde u_h(\cdot, T)-u_{\theta^*}(\cdot, T)\|_{L^1(\Omega)}
	\lesssim h^{1/2}.
	\]
	This completes the proof of Theorem~\ref{thm:shock}.

\end{proof}

For the case of shock interactions, we have the following convergence result. 
\begin{theorem}[Shock interaction]\label{thm:shock_interaction}
Let $u$ be the entropy solution of the scalar conservation law \eqref{eq:scalar-HCL} that has discontinuous initial data $u_0$, contains a single shock interaction occurring at the terminal time $T$, and satisfies Assumption~\ref{assump:piecewise_smooth}. Let $h$ be the mesh size of the background mesh $\Lambda_h$. Let \(u_{\theta^*}\) be a clipped \(\tanh\) neural network of the form \eqref{def:clipped-NN} with depth $O(d+1)$ and at most $O(h^{-(d+1)})$ neurons in each hidden layer, where $\theta^*$ is a minimizer of the minimization problem \eqref{eq:NN_min_problem} over this class of networks. Then there exists a constant $C > 0$, independent of $h$, such that
\begin{equation}
\|u_{\theta^*}(\cdot, T) - u(\cdot, T)\|_{L^1(\Omega)} \leq C h^{1/2}.
\end{equation}
\end{theorem}

\begin{proof}
	The proof follows the same overall construction as in Theorem~\ref{thm:shock}, with an additional local analysis near the interaction point \((x_\beta,T)\). A complete argument is contained in the proof of the more general higher-dimensional interaction result, Theorem~\ref{thm:more_general_case}; the present statement is a special case of that theorem.
\end{proof}


\subsection{Single rarefaction waves and compound waves}\label{sec:rarefaction_wave}
In this section, we extend our analysis to solutions featuring a single rarefaction wave. 
\begin{lemma}\label{lem:Lu_rarefaction}
Let $u$ be the entropy solution of the scalar conservation law \eqref{eq:scalar-HCL} that contains a single rarefaction wave  and satisfies Assumption~\ref{assump:piecewise_smooth}. Then there exists  a clipped $\tanh$ neural network $u_{\hat\theta} = \Pi_c\bigl(U_{\hat\theta}^{\rm raw}\bigr)$ of the form \eqref{def:clipped-NN} with depth $O(d+1)$ and at most $O(h^{-(d+1)})$ neurons in each hidden layer, such that
\begin{align}\label{eq:Lu_rarefaction}
\mathscr{L}(\hat{\theta})\lesssim h|\ln h|.
\end{align}
\end{lemma}
\begin{proof}
Let $A_0,\dots,A_3$ be the partition of $\Omega_T$ shown in Figure \ref{fig:grids_rarefaction} (left part); in particular,
\[
A_0=[\Gamma_1(0)-O(h),\,\Gamma_1(0)+O(h)]\times[0,t_1],
\quad
A_2=\{(x,t)\in\Omega_T:\ t\in[t_1,T],\ \Gamma_1(t)\le x\le \Gamma_2(t)\},
\]
with $t_1\simeq h$.
Let $\{\Lambda_i\}_{i=0}^3$ denote the induced decomposition of the background mesh $\Lambda_h$ shown in Figure \ref{fig:grids_rarefaction} (right part), i.e.\ $\Lambda_2:=\{K\in\Lambda_h:\ K\cap A_2 \neq \emptyset\}$, and $\Lambda_i:=\{K\in\Lambda_h:\ K\subset A_i \}$, $i=1,3$.

\smallskip
\noindent\textbf{Step 1: Construction of $\hat u$.}
We construct a shape-regular simplicial partition of size \(O(h)\) over $\Omega_T$ that aligns with boundaries of $\{\Lambda_i\}_{i=0}^3$ and define a CPwL \(\hat{u}\) as the linear interpolation of \(u\) on this partition. With $\hat u\in \mathcal P^1(\mathcal T_h)\cap C^0(\Omega_T)$, $\|\nabla \hat u\|_{L^\infty(A_0)}\lesssim h^{-1}$, $\hat u=0$ on $\partial\Omega\times(0,T)$, we have
\begin{equation}\label{eq:int_div_u_hat_rare}
    \|\hat u(\cdot,0)-u_0\|_{L^1(\Omega)}\lesssim h,\quad \|\hat u\|_{L^1(\partial\Omega\times(0,T))} =0,
    \quad
    \int_{\Omega_T} \big|\nabla \cdot \mathbf F(\hat u)\big|\,\d \z \lesssim 1.
\end{equation}

For any $k_h\in V_h^c$, since $|\sgn|\le 1$ we have
\begin{equation}\label{eq:estimate_Lu_rarefaction}
    \J_{\mathrm{ent}}(\hat u;k_h)
    =\int_{\Omega_T} \nabla \cdot \mathbf F(\hat u)\,\sgn\!\left(\hat u-k_h\right)\,\d \z
    \le \sum_{i=0}^3\sum_{K\in\Lambda_i}\int_K \big|\nabla \cdot \mathbf F(\hat u)\big|\,\d \z.
\end{equation}
We estimate the right-hand side region by region.

\smallskip
\noindent\emph{Smooth regions $A_1$ and $A_3$.}
On $A_1$ and $A_3$, the solution \(u\) is \(C^2\) with uniformly bounded derivatives. Since \(\nabla\cdot\mathbf F(u)=0\) and \(\hat u\) is the \(P^1\) interpolant of \(u\), standard interpolation estimates imply that \(\nabla\cdot\mathbf F(\hat u)=O(h)\) on \(K\in\Lambda_i\). Therefore,
\begin{equation}\label{eq:estimate_Lambda123}
    \sum_{K\in\Lambda_i}\int_K \big|\nabla \cdot \mathbf F(\hat u)\big|\,\d \z
    \lesssim h,
    \quad i\in\{1,3\}.
\end{equation}

\smallskip
\noindent\emph{Rarefaction fan region $A_2$.}
Assumption~\ref{assump:piecewise_smooth}(ii) implies $|\nabla^2 u(x,t)|\lesssim t^{-2}$ on $A_2$.
Consequently, on each cell $K\subset A_2$ at time level $t\sim t_K$,
\(
\|\nabla \cdot \mathbf F(\hat u)\|_{L^\infty(K)}\lesssim h\,t_K^{-2}.
\)
Using that $|\Gamma_2(t)-\Gamma_1(t)|\lesssim t$ for a rarefaction fan, we obtain
\begin{equation}\label{eq:estimate_Lambda4}
    \sum_{K\in\Lambda_2}\int_K \big|\nabla \cdot \mathbf F(\hat u)\big|\,\d \z
    \lesssim \int_{t_1}^{T}\frac{h}{t^2}\,|\Gamma_2(t)-\Gamma_1(t)|\,\d t
    \lesssim h\int_{t_1}^{T}\frac{\d t}{t}
    \lesssim h|\ln h|.
\end{equation}

\smallskip
\noindent\emph{Tip region $A_0$.}
In $A_0$ we have $\|\nabla \hat u\|_{L^\infty(A_0)}\lesssim h^{-1}$ and $|A_0|\lesssim h^{2}$, hence
\begin{equation}\label{eq:estimate_Lambda0}
    \sum_{K\in\Lambda_0}\int_K \big|\nabla \cdot \mathbf F(\hat u)\big|\,\d \z
    \lesssim h.
\end{equation}

Combining \eqref{eq:estimate_Lambda123}-\eqref{eq:estimate_Lambda0} in \eqref{eq:estimate_Lu_rarefaction} yields
\[
\sup_{k_h\in V_h^c}\J_{\mathrm{ent}}(\hat u;k_h)\lesssim h|\ln h|.
\]
Together with \eqref{eq:int_div_u_hat_rare} and the definition \eqref{eq:loss_func}, we conclude that
\begin{equation}\label{eq:Lu_rarefaction_cpwl}
    \L(\hat u)\lesssim h|\ln h|.
\end{equation}

\smallskip
\noindent\textbf{Step 2: Neural network approximation.}
In addition, the constructed CPwL function \(\hat u\) satisfies
\[
\|\nabla \hat u\|_{L^1(\Omega_T)} \le C.
\]
Indeed, this follows from the uniform \(W^{1,\infty}\)-bound on the smooth regions,
the estimate \(\|\nabla \hat u\|_{L^\infty(A_0)}\lesssim h^{-1}\) together with \(|A_0|\lesssim h^2\),
and the bound \(|\nabla \hat u(x,t)|\lesssim t^{-1}\) in the rarefaction region \(A_2\), whose integral is finite. By construction of the CPwL function $\hat u$, we have $\|\hat u\|_{L^\infty(\Omega_T)} \le c/2 - \delta_c$, 
so the assumptions of Lemma~\ref{lem:loss_cpwl_approx} are satisfied. 
Therefore, by setting the singular region to be empty in Lemma~\ref{lem:loss_cpwl_approx}, and using
\[
\int_{\Omega_T} \big|\nabla \cdot \mathbf F(\hat u)\big|\, d\mathbf z \lesssim h|\ln h|,
\]
we obtain that there exists a \(\tanh\) neural network \(u_{\hat\theta}\) of the stated size such that \eqref{eq:Lu_rarefaction} holds.
\end{proof}
\begin{figure}[htbp!]
\centering
\includegraphics[width=0.75\linewidth]{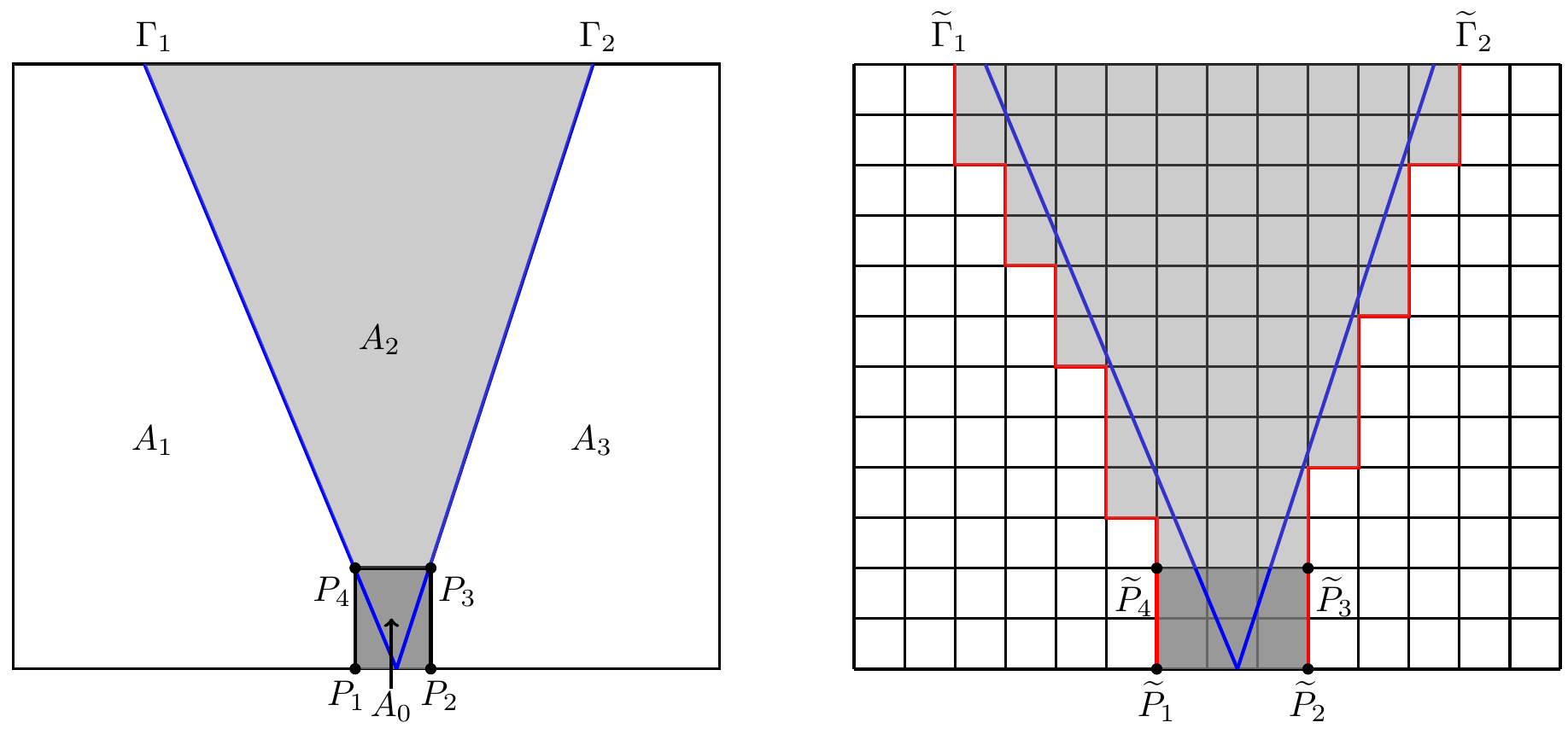}
\caption{Domain decomposition for rarefaction $u$ and DPwP $\tilde u_h$. Left: region partition of $\Omega_T$. Right: mesh partition for the DPwP $\tilde u_h$}
\label{fig:grids_rarefaction}
\end{figure}
Then we derive an error estimate for the neural network minimizer $u_{\theta^*}$.
\begin{theorem}[Single rarefaction]\label{thm:rarefaction}
Let $u$ be the entropy solution of the scalar conservation law \eqref{eq:scalar-HCL} that has discontinuous initial data $u_0$, contains a single rarefaction wave, and satisfies Assumption~\ref{assump:piecewise_smooth}. Let $h$ be the mesh size of the background mesh $\Lambda_h$. Let \(u_{\theta^*}\) be a clipped \(\tanh\) neural network of the form \eqref{def:clipped-NN} with depth $O(d+1)$ and at most $O(h^{-(d+1)})$ neurons in each hidden layer, where $\theta^*$ is a minimizer of the minimization problem \eqref{eq:NN_min_problem} over this class of networks. Then there exists a constant $C > 0$, independent of $h$, such that
\begin{align}\label{eq:single-rarefaction-rate}
\|u_{\theta^{*}}(\cdot, T)-u(\cdot, T)\|_{L^1(\Omega)} \leq C h|\ln h|.
\end{align}
\end{theorem}

\begin{proof}
By Lemma~\ref{lem:Lu_rarefaction} and $\theta^*$ being a minimizer of the minimization problem \eqref{eq:NN_min_problem}, we have 
\begin{align}\label{eq:recall_Lu}
\J_{\mathrm{ent}}(u_{\theta^*}; k_h) \leq \mathscr{L}(\theta^{*}) \leq \mathscr{L}(\hat{\theta}) \lesssim h|\ln h|, \quad \forall k_h \in V_h^c.
\end{align}
As before, by taking $\tilde{u}_h \in V_h^c$ as an approximation to the exact solution $u$, we derive an error estimate.

Based on the space–time decomposition and the partition $\Lambda_h$ illustrated in 
Figure~\ref{fig:grids_rarefaction}, with $\Lambda_i$ corresponding to $A_i$ for $i=0,1,2, 3$, 
we proceed as follows. 

First, the exact solution $u$ in $A_2 $ is smoothly extended 
to  $\Lambda_2 $. 
The function $\tilde{u}_h$ is then defined as the continuous piecewise 
$\mathbb{Q}_1$ interpolation of this extended function over 
$\Lambda_2 $. 
Over the regions $\Lambda_1 \cup \Lambda_3$, 
$\tilde{u}_h$ is taken as the continuous piecewise 
$\mathbb{Q}_1$ interpolation of the exact solution $u$. 
Finally, on $\Lambda_0$, $\tilde{u}_h$ is defined as a continuous 
piecewise $\mathbb{Q}_1$ function that remains continuous 
across the common boundary between $\Lambda_0$ and $\Lambda_2$. 
In addition, $\tilde{u}_h$ satisfies the bound
$|\tilde{u}_h| + h |\nabla \tilde{u}_h| \le C$ in $\Lambda_0$,
where the constant $C$ is independent of $h$.

By substituting $k_h = \tilde{u}_h$ into \eqref{eq:recall_Lu}, and noting the continuity of $u_{\theta^{*}}$ over $\Lambda_h$ and the constructed function $\tilde{u}_h$ within $\Lambda_{0} \cup \Lambda_2$, $\Lambda_1$, and $\Lambda_3$, we conclude that
\begin{equation}\label{eq:part1_Lu_rarefaction}
\begin{aligned}
&\sum_{K\in \Lambda_{h}}\int_{\partial K} \left(\mathbf{F}(u_{\theta^{*}})-\mathbf{F}(\tilde{u}_h)\right)\sgn\!\left(u_{\theta^{*}}-\tilde{u}_h\right)\cdot\mathbf{n}\,\d s \\
=&\int_{\Omega}\big|u_{\theta^{*}}-\tilde{u}_h\big|(x, T)\,\d x 
-\int_{\Omega}\big|u_{\theta^{*}}-\tilde{u}_h\big|(x, 0)\,\d x \\
+ & \int_{\partial\Omega\times(0,T)}
\left(\mathbf f(u_{\theta^*})-\mathbf f(\tilde u_h)\right)
\sgn(u_{\theta^*}-\tilde u_h)\cdot\mathbf n_{\partial\Omega}\,\d s +\int_{\widetilde{\Gamma}_1\cup\widetilde{\Gamma}_2} \llbracket\left(\mathbf{F}(u_{\theta^{*}})-\mathbf{F}(\tilde{u}_h)\right)\sgn\!\left(u_{\theta^{*}}-\tilde{u}_h\right)\rrbracket\cdot\mathbf{n}\,\d s.
\end{aligned}
\end{equation}
By the Lipschitz continuity of $\mathbf f$ and the construction of $\tilde u_h$ on the boundary,
\begin{equation}\label{eq:lateral_boundary_bound-1}
| \int_{\partial\Omega\times(0,T)}
\left(\mathbf f(u_{\theta^*})-\mathbf f(\tilde u_h)\right)
\sgn(u_{\theta^*}-\tilde u_h)\cdot\mathbf n_{\partial\Omega}\,\d s|
\le
C\|u_{\theta^*}\|_{L^1(\partial\Omega\times(0,T))}
\le
C\mathcal L(u_{\theta^*})
\lesssim h|\ln h| .
\end{equation}

Using the same argument as in \eqref{eq:estimate_jump} we have
\begin{align}
    \big| \llbracket (\mathbf{F}(u_{\theta^*}) - \mathbf{F}(\tilde{u}_h)) \sgn\!\left(u_{\theta^*} - \tilde{u}_h\right) \rrbracket \cdot \mathbf{n} \big| \leq \big| \llbracket \mathbf{F}(u_{\theta^*}) - \mathbf{F}(\tilde{u}_h) \rrbracket \big| \lesssim |\llbracket \tilde{u}_h \rrbracket|.
\end{align}

Since $\tilde{u}_h$ on each side of 
$\widetilde{\Gamma}_1(t) \cup \widetilde{\Gamma}_2(t)$ 
is obtained by piecewise $\mathbb{Q}_1$ interpolation of either the exact solution $u$ or its local smooth extension, and because for $\boldsymbol{z} = (x,t) \in \widetilde{\Gamma}_1(t) \cup \widetilde{\Gamma}_2(t)$ 
we have $\operatorname{dist}(\boldsymbol{z}, \Gamma_1 \cup \Gamma_2) \lesssim h$, 
it follows from the continuity of $u$ along $\Gamma_1$ and $\Gamma_2$ 
that the jump of $\tilde{u}_h$ over 
$\widetilde{\Gamma}_1(t) \cup \widetilde{\Gamma}_2(t)$ 
for all $t \ge t_1$ can be bounded by
\[
|\llbracket \tilde{u}_h \rrbracket| \lesssim h \sup_{x' \in B^h_x} |\nabla u(x', t)|\lesssim \left\{\begin{array}{cc}
    h & \text{for}\quad t>t_2, \\
    {\displaystyle\frac{h}{t}} & \quad\;\;\text{for} \quad t\in[t_1,t_2].
\end{array}\right.
\]
Denoting by \(\widetilde{\Gamma}_{\widetilde{P}_i \widetilde{P}_j}\) the line segment connecting \(\widetilde{P}_i\) and \(\widetilde{P}_j\), with \(|\widetilde{P}_1 \widetilde{P}_4| + |\widetilde{P}_2 \widetilde{P}_3| \lesssim h\), we have 
\begin{equation}\label{eq:Gamma_estimate_rarefaction}
\begin{aligned}
&\int_{\widetilde{\Gamma}_1 \cup \widetilde{\Gamma}_2} \llbracket \left( \mathbf{F}(u_{\theta^*}) - \mathbf{F}(\tilde{u}_h) \right) \sgn\!\left(u_{\theta^*} - \tilde{u}_h\right) \rrbracket \cdot \mathbf{n}\,\d s \\
\lesssim &\left( \int_{\widetilde{\Gamma}_{\widetilde{P}_1 \widetilde{P}_4} \cup \widetilde{\Gamma}_{\widetilde{P}_2 \widetilde{P}_3}}+\int_{(\widetilde{\Gamma}_1 \cup \widetilde{\Gamma}_2) \setminus (\widetilde{\Gamma}_{\widetilde{P}_1 \widetilde{P}_4} \cup \widetilde{\Gamma}_{\widetilde{P}_2 \widetilde{P}_3})} \right)|\llbracket \tilde{u}_h \rrbracket| \d s \\
\lesssim &\int_{\widetilde{\Gamma}_{\widetilde{P}_1 \widetilde{P}_4} \cup \widetilde{\Gamma}_{\widetilde{P}_2 \widetilde{P}_3}} \big( |\tilde{u}_h^+| + |\tilde{u}_h^-| \big)\,\d s + \int_{(\widetilde{\Gamma}_1 \cup \widetilde{\Gamma}_2) \setminus (\widetilde{\Gamma}_{\widetilde{P}_1 \widetilde{P}_4} \cup \widetilde{\Gamma}_{\widetilde{P}_2 \widetilde{P}_3})} \frac{h}{t}\,\d s
\lesssim  h + h |\ln h|.
\end{aligned}
\end{equation}

Meanwhile, since \(\tilde{u}_h\) is a DPwP of the exact solution \(u\), we follow the same reasoning as in the proof of Lemma \ref{lem:Lu_rarefaction} to obtain
\begin{equation}\label{eq:part2_Lu_rarefaction}
\left| \sum_{K \in \Lambda_h} \int_K \nabla \cdot \mathbf{F}(\tilde{u}_h) \sgn\!\left(u_{\theta^*} - \tilde{u}_h\right) \, \d \z \right| \leq \sum_{K \in \Lambda_h} \int_K \big| \nabla \cdot \mathbf{F}(\tilde{u}_h) \big| \, \d \z \lesssim h |\ln h|.
\end{equation}
Combining this estimate with \eqref{eq:recall_Lu}, \eqref{eq:part1_Lu_rarefaction}, \eqref{eq:lateral_boundary_bound-1}, \eqref{eq:Gamma_estimate_rarefaction}, and \eqref{eq:part2_Lu_rarefaction}, we conclude that
\[
\|u_{\theta^*}(\cdot, T) - \tilde{u}_h(\cdot, T)\|_{L^1(\Omega)} \leq \|u_{\theta^*}(\cdot, 0) - \tilde{u}_h(\cdot, 0)\|_{L^1(\Omega)} + C h |\ln h|.
\]
Since \(\tilde{u}_h\) is a piecewise linear approximation of the exact solution \(u\), we obtain \eqref{eq:single-rarefaction-rate}.

\end{proof}

For the case of compound waves,  we have the following convergence result.
\begin{theorem}[Single compound wave]\label{thm:compound_wave}
Let $u$ be the entropy solution of the scalar conservation law \eqref{eq:scalar-HCL} that has discontinuous initial data $u_0$, contains a single compound wave, and satisfies Assumption~\ref{assump:piecewise_smooth}. Let $h$ be the mesh size of the background mesh $\Lambda_h$. Let \(u_{\theta^*}\) be a clipped \(\tanh\) neural network of the form \eqref{def:clipped-NN} with depth $O(d+1)$ and at most $O(h^{-(d+1)})$ neurons in each hidden layer, where $\theta^*$ is a minimizer of the minimization problem \eqref{eq:NN_min_problem} over this class of networks. Then there exists a constant $C>0$, independent of $h$, such that
\begin{equation}\label{eq:compound_wave_result}
\|u_{\theta^*}(\cdot,T)-u(\cdot,T)\|_{L^1(\Omega)}
\le C(h|\ln h|)^{1/2}.
\end{equation}
\end{theorem}

The proof combines the shifted-interface argument from Theorem~\ref{thm:shock} with the rarefaction estimate from Theorem~\ref{thm:rarefaction}. The logarithmic factor comes from the approximation of the rarefaction part: as shown in Lemma~\ref{lem:Lu_rarefaction}, the contribution of the rarefaction region is \(O(h|\ln h|)\). Therefore, one can construct a $\tanh$ neural network \(u_{\hat\theta}\) such that
\(
\mathcal L(u_{\hat\theta}) \lesssim h|\ln h|.
\)
Hence \(\mathcal L(u_{\theta^*}) \lesssim h|\ln h|\), and therefore \(\int_{\Omega_T} |\nabla\!\cdot\mathbf F(u_{\theta^*})|\,d\mathbf z \lesssim |\ln h|.\)
The strip-selection argument gives
\(
\int_{\Lambda_j} |\nabla\!\cdot\mathbf F(u_{\theta^*})|\,d\mathbf z \lesssim |\ln h|/M,
\)
so balancing \(Mh\) and \(|\ln h|/M\) yields
\(
M\sim (|\ln h|/h)^{1/2}
\)
and thus the rate
\(
(h|\ln h|)^{1/2}.
\)
We omit further details.


\subsection{Shock formation with smooth initial data}\label{sec:shock-formation}

We next consider smooth initial data that generate a nondegenerate shock in finite time.

\begin{theorem}[Shock formation from smooth initial data]\label{thm:shock_smooth_initial}
Let \(u\) be the entropy solution of \eqref{eq:scalar-HCL} in one space dimension. Assume that \(u_0\) is smooth, that \(u\) develops one nondegenerate shock at time \(\tau_q<T-\delta_*\), and that Assumption~\ref{assump:piecewise_smooth} holds. Let \(h\) be the mesh size of the background mesh \(\Lambda_h\). Let \(u_{\theta^*}\) be a clipped \(\tanh\) neural network of the form \eqref{def:clipped-NN} with depth \(O(d+1)\) and at most \(O(h^{-(d+1)})\) neurons in each hidden layer, where \(\theta^*\) is a minimizer of \eqref{eq:NN_min_problem} over this class of networks. Then there exists a constant \(C>0\), independent of \(h\), such that
\begin{equation}\label{eq:shock_smooth_initial}
\|u_{\theta^*}(\cdot,T)-u(\cdot,T)\|_{L^1(\Omega)}
\le C h^{1/2}|\ln h|.
\end{equation}
\end{theorem}

The proof follows the shifted-interface argument used for shocks, with an additional local estimate near the shock-formation point. It is given in Appendix~\ref{appendix:proof_smooth_init}.

\subsection{Main result for one-dimensional HCL}
Combining the analysis strategies developed in Sections~\ref{sec:shock_wave}--\ref{sec:shock-formation}, we establish the following global convergence result under the general piecewise smooth setting.

\begin{theorem}[Main result in 1D]\label{thm:main_1d}
Let \(u\) be the entropy solution of the scalar conservation law \eqref{eq:scalar-HCL} that is piecewise smooth in the sense of Assumption~\ref{assump:piecewise_smooth}. Let \(h\) be the mesh size of the background mesh \(\Lambda_h\). Let \(u_{\theta^*}\) be a clipped \(\tanh\) neural network of the form \eqref{def:clipped-NN} with depth \(O(d+1)\) and at most \(O(h^{-(d+1)})\) neurons in each hidden layer, where \(\theta^*\) is a minimizer of the minimization problem \eqref{eq:NN_min_problem} over this class of neural networks. Then there exists a constant \(C>0\), independent of \(h\), such that
\begin{equation}\label{eq:main_1d}
\|u_{\theta^*}(\cdot,T)-u(\cdot,T)\|_{L^1(\Omega)}
\le C h^{1/2}|\ln h|.
\end{equation}
\end{theorem}

The proof is obtained by decomposing the entropy solution into finitely many elementary waves, shock interactions, compound waves, and nondegenerate shock-birth pieces, each of which is handled by the corresponding local analysis. The complete argument is given in Appendix~\ref{appendix:proof_1d_thm}.

\section{Error estimate for $d$-dimensional scalar hyperbolic conservation laws}\label{sec:multiD}
We now extend the analysis to several space dimensions. A single smooth
shock can be treated by the shifted-interface strategy of
Section~\ref{sec:1d}, together with the multidimensional fitted-interface
estimates stated below. We focus on the genuinely new feature considered
here: the codimension-two interaction set of several shock
hypersurfaces; see Figure~\ref{fig:3d_shock_intersect}. We first construct
a network competitor with small loss and then a DPwP comparison function
\(\tilde u_h\) for the stability argument.

\begin{figure}[htbp!]
\centering
\begin{subfigure}[b]{0.35\textwidth}
\includegraphics[width=\linewidth]{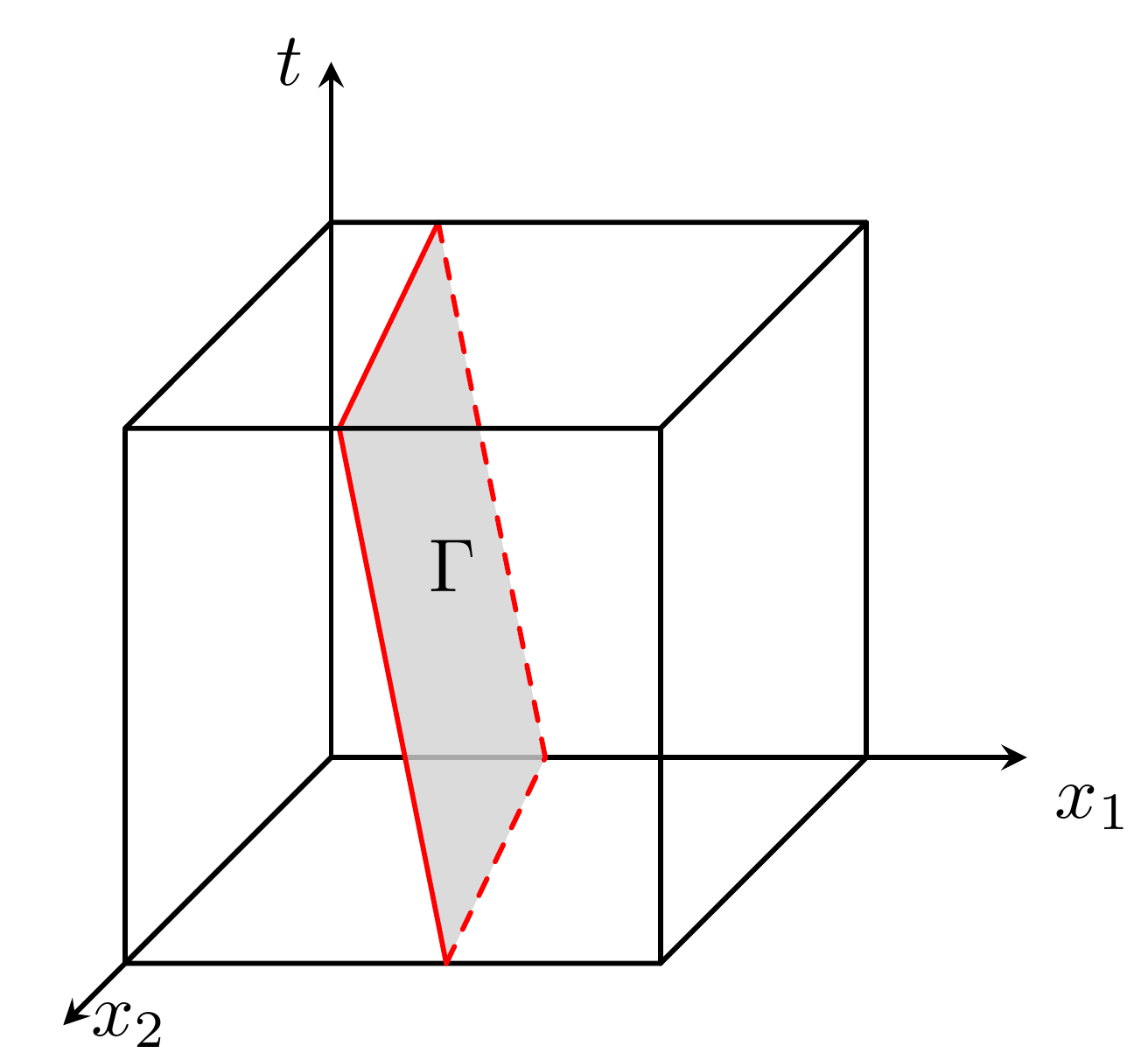}
\caption{Single shock wave solution with shock surface $\Gamma$}
\end{subfigure}
\qquad
\begin{subfigure}[b]{0.35\textwidth}
\includegraphics[width=\linewidth]{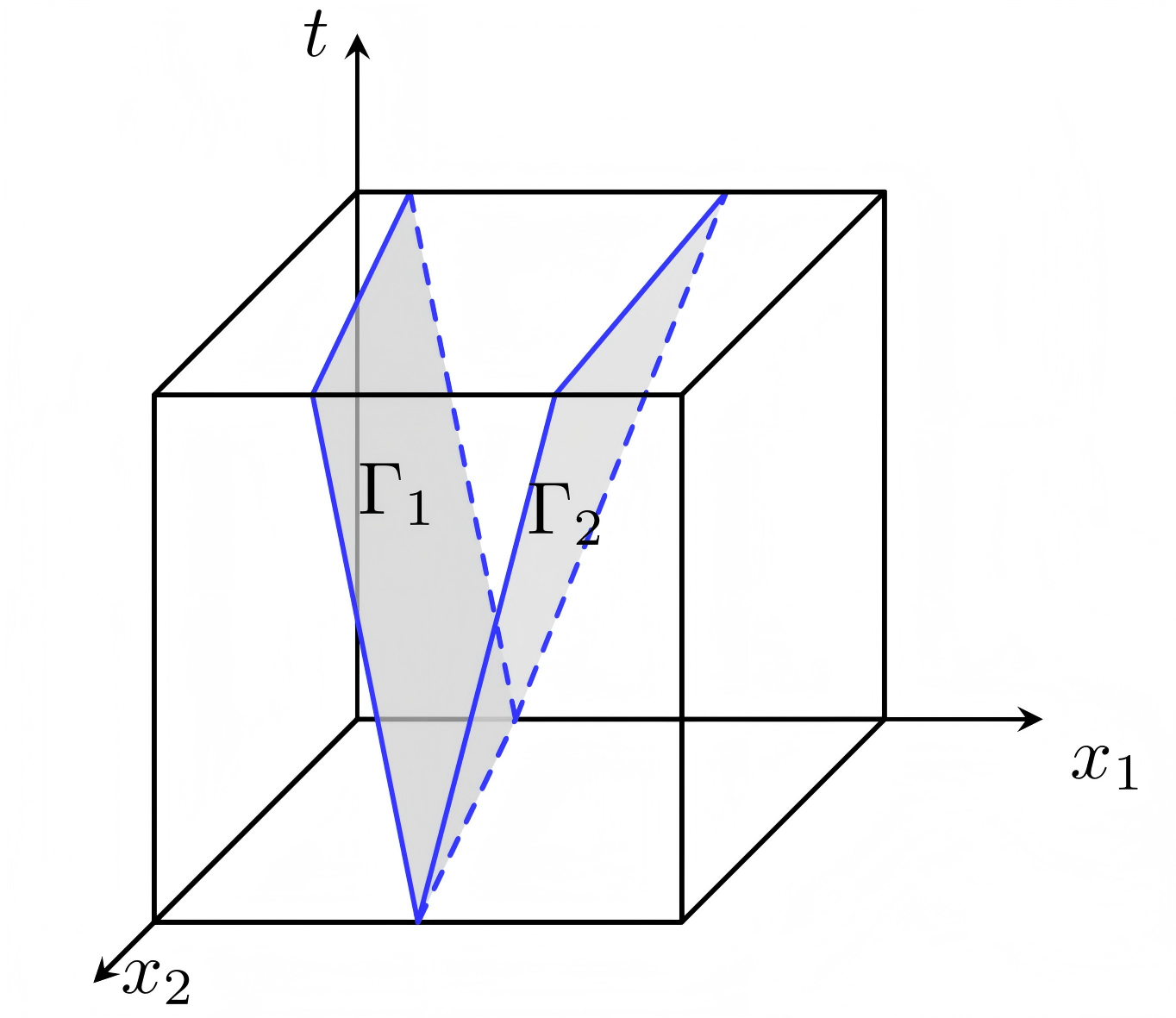}
\caption{Single rarefaction wave solution with rarefaction boundaries $\Gamma_1$ and $\Gamma_2$}
\end{subfigure}
\caption{Solutions in high-dimensional cases with a single shock and rarefaction waves.}
\label{fig:solution_3d}
\end{figure}

\begin{assumption}[Piecewise $C^2$ solution with finitely many shock hypersurfaces]\label{assump:shock_3d}
Let \(u\) be the entropy solution of \eqref{eq:scalar-HCL}. Assume that there exist finitely many relatively closed $C^2$ hypersurfaces $\Gamma_i\subset\Omega_T$, $i=1,\ldots,N_\Gamma$, whose closures $\overline{\Gamma_i}\subset\Omega\times[0,T]$ are compact $C^2$ hypersurfaces, possibly with boundary. Define $\Gamma:=\bigcup_{i=1}^{N_\Gamma}\Gamma_i$ and assume that $\Gamma$ has finite $d$-dimensional Hausdorff measure. The following conditions are assumed to hold:
\begin{enumerate}
\item On each connected component of \(\Omega_T\setminus\Gamma\), the solution \(u\) is \(C^2\) up to the shock hypersurfaces, with uniformly bounded one-sided \(W^{2,\infty}\) norms, and satisfies \(\nabla\cdot\mathbf F(u)=0\). Along each \(\Gamma_i\), the one-sided traces of \(u\) satisfy the Rankine--Hugoniot relation and the reformulated entropy condition \eqref{eq:reformulated_entropy_condition}.

\item For $i\neq j$, every nonempty intersection
$\overline{\Gamma_i}\cap\overline{\Gamma_j}$
is a compact $C^2$ $(d-1)$-dimensional manifold,
possibly with boundary. Define
\[
\Sigma:=\bigcup_{1\le i<j\le N_\Gamma}
\bigl(\overline{\Gamma_i}\cap\overline{\Gamma_j}\bigr),
\]
and assume that $\Sigma$ has finite
$(d-1)$-dimensional Hausdorff measure.

\item The intersections are transversal, and the following
quantitative separation holds: there exist constants $r_0>0$ and $c_{\rm sep}>0$ such that for every $0<\delta\le r_0$, if $S_1$ and $S_2$ are two distinct connected components of $\Gamma\setminus\Sigma_\delta$, where $\Sigma_\delta:=\big\{\z\in\Omega_T:\dist(\z,\Sigma)<\delta\big\}$, then $\dist(S_1,S_2)\ge c_{\rm sep}\delta$.
\end{enumerate}
\end{assumption}

\begin{figure}[htbp!]
	\centering
	\begin{subfigure}[b]{0.4\textwidth}
\centering
\includegraphics[width=\linewidth]{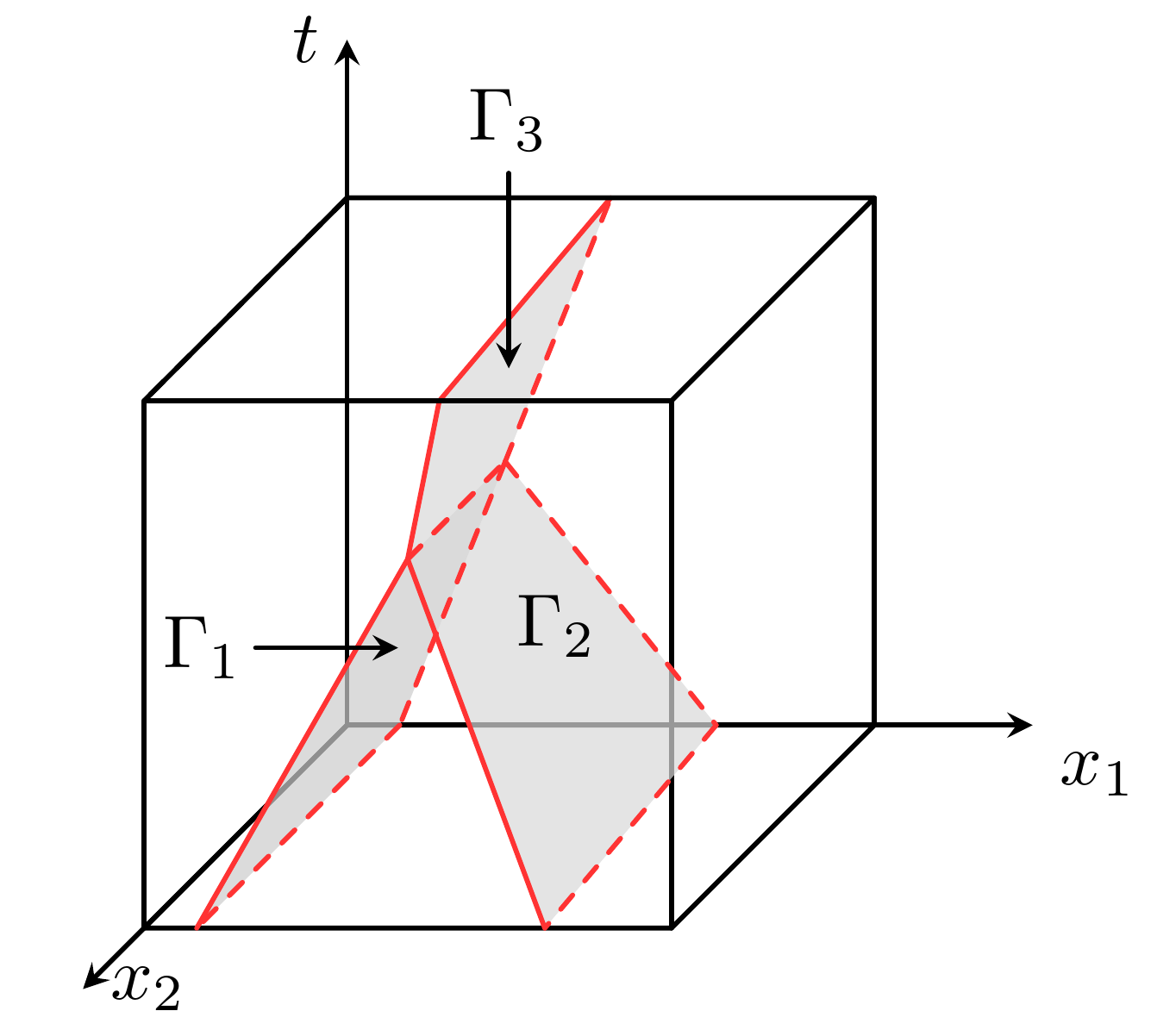}
\caption{Solution with multiple shock waves}
\label{fig:3d_shock_intersect}
\end{subfigure}
	\begin{subfigure}[b]{0.4\textwidth}
\centering
\includegraphics[width=\linewidth]{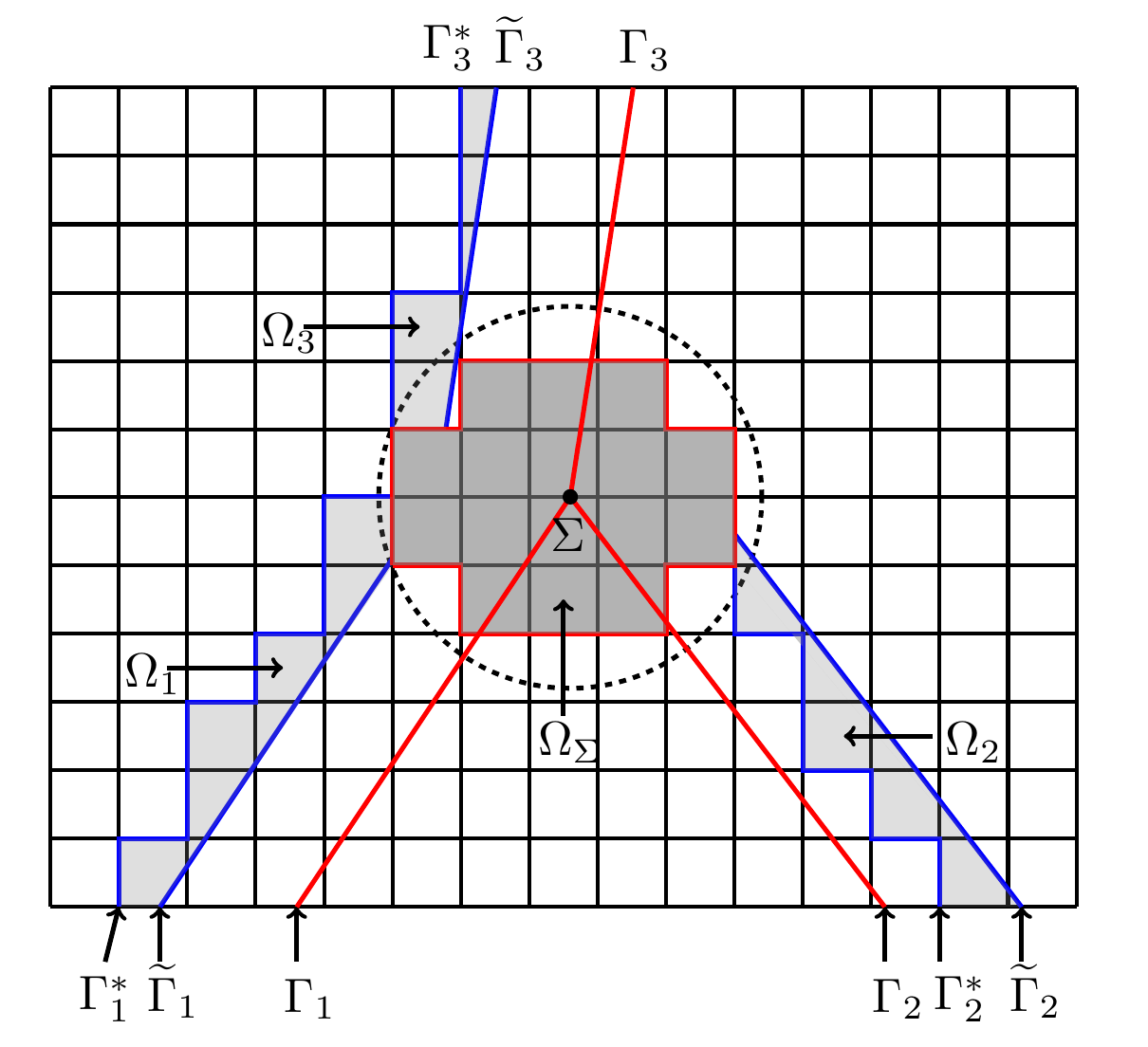}
\caption{Strategy for selecting the $k_h$}
\label{fig:construction_K_3d}
\end{subfigure}
	\caption{Schematic diagram of a solution with multiple shock waves in $d$ dimensions.}
\end{figure}

Following a construction similar to that in Lemma~\ref{lem:Lu_shock}, we can construct a specific CPwL $\hat{u}$, which can be approximated by a $\tanh$ neural network $u_{\hat{\theta}}$, such that the loss is sufficiently small.
\begin{lemma}\label{lem:Lu_shock_3d}
	Let $u$ be the entropy solution of the scalar conservation law \eqref{eq:scalar-HCL} satisfying Assumption~\ref{assump:shock_3d}. Then there exists a clipped $\tanh$ neural network $u_{\hat\theta} = \Pi_c\bigl(U_{\hat\theta}^{\rm raw}\bigr)$ of the form~\eqref{def:clipped-NN} with depth $O(d+1)$ and at most $O(h^{-(d+1)})$ neurons in each hidden layer, such that
	\begin{equation}\label{eq:Lu_shock_3d}
		\mathscr{L}(\hat{\theta}) \lesssim h.
	\end{equation}
\end{lemma}
The proof of this lemma is provided in Appendix~\ref{app:proof-lemma-nd}.

From Lemma~\ref{lem:Lu_shock_3d}, we know that the minimizer \(u_{\theta^*}\) of the minimization problem \eqref{eq:NN_min_problem} must achieve a loss no greater than that of the network constructed in Lemma~\ref{lem:Lu_shock_3d}. Hence,
\(
\mathscr{L}(\theta^*) \leq \mathscr{L}(\hat{\theta}) \lesssim h.
\)
Based on the stability provided by the Kru\v{z}kov-type loss, and following the arguments used in the proof of Theorem~\ref{thm:shock}, we obtain the following theorem.

\begin{theorem}[Multiple shock hypersurfaces]\label{thm:more_general_case}
	Let \(u\) satisfy Assumption~\ref{assump:shock_3d}, and let
\(\Lambda_h\) be the uniform Cartesian background mesh. Let \(u_{\theta^*}\) be a clipped \(\tanh\) neural network of the form \eqref{def:clipped-NN} with  depth $O(d+1)$ and at most $O(h^{-(d+1)})$ neurons in each hidden layer, where $\theta^*$ is a minimizer of the minimization problem \eqref{eq:NN_min_problem}. The network class is chosen to contain the competitor constructed in Lemma~\ref{lem:Lu_shock_3d}. Then there exists a constant $C>0$, independent of $h$, such that
	\begin{equation}\label{eq:more_general_case}
		\|u_{\theta^*}(\cdot, T)-u(\cdot, T)\|_{L^1(\Omega)}\le Ch^{1/2}.
	\end{equation}
\end{theorem}

\begin{proof}
    We adapt the proof of Theorem~\ref{thm:shock}, with an additional estimate near the codimension-two interaction set $\Sigma$.

	Since $\theta^*$ minimizes the loss,
	\begin{equation}\label{eq:md_loss_min}
		\L(u_{\theta^*})\le \L(u_{\hat\theta})\lesssim h.
	\end{equation}
	In particular,
	\begin{equation}\label{eq:md_J_min}
		\J_{\mathrm{ent}}(u_{\theta^*};k_h)\lesssim h,
		\quad \forall k_h\in V_h^c, \quad \int_{\Omega_T}|\nabla\cdot\mathbf F(u_{\theta^*})|\,\d\z
		\le h^{-1}\L_{\rm reg}(u_{\theta^*})
		\lesssim 1.
	\end{equation}

	\noindent\textbf{Step 1: Construction of a DPwP test function.}
	We now construct a DPwP test function \(k_h=\tilde u_h\) that plays the
	same role as the shifted-interface test function in the one-dimensional
	shock proof. The basic obstruction is that the physical discontinuity set
	\(\Gamma=\cup_i\Gamma_i\) is not, in general, aligned with the background
	mesh \(\Lambda_h\). Since a DPwP function may jump only across mesh faces,
	a direct cellwise interpolant of the exact solution cannot have its jump
	set exactly on \(\Gamma\). Instead, its jump set must be replaced by a
	mesh-aligned hypersurface, which we denote by $\Gamma^*:=\bigcup_i \Gamma_i^*$.
	Consequently, the entropy condition cannot be applied directly on
	\(\Gamma^*\), because the entropy admissibility condition is known only on
	the true shock hypersurfaces \(\Gamma_i\).

	To overcome this mismatch, we use the same shifted-interface mechanism as
	in Theorem~\ref{thm:shock}. For each smooth component of the shock set, we
	choose a nearby mesh-aligned hypersurface \(\Gamma_i^*\) and an auxiliary
	surface \(\widetilde\Gamma_i\) parallel to the true shock surface
	\(\Gamma_i\). The surface \(\Gamma_i^*\) is chosen by a strip-selection
	argument so that the residual of \(u_{\theta^*}\) is small in the thin
	region enclosed between \(\Gamma_i^*\) and \(\widetilde\Gamma_i\). We then
	define \(\tilde u_h\) by interpolating the appropriate smooth one-sided
	extensions of \(u\) on the two sides of \(\Gamma_i^*\). This makes
	\(\tilde u_h\in V_h^c\) and ensures that all its jumps occur on mesh faces.
	Later, the jump integral over \(\Gamma_i^*\) will be transferred through
	the thin region to \(\widetilde\Gamma_i\), and then compared with the
	entropy-admissible jump on the true surface \(\Gamma_i\).

	The main additional difficulty is the presence of the intersection set \(\Sigma\), where several
	discontinuity surfaces may meet. To separate this singular region from the shifted-interface construction, we fix a constant
	\(\kappa_{\rm sh}>0\) that bounds the maximal normal shift used below. Choose \(\kappa_{\rm sh}\) so that
	\[
	2\kappa_{\rm sh}<c_{\rm sep},
	\]
	where \(c_{\rm sep}\) is the separation constant in Assumption~\ref{assump:shock_3d}. For \(h\)
	sufficiently small so that \(\delta_h:=h^{1/2}\le r_0\), define
	\begin{align}\label{eq:def_tilde_Sigma_revised2}
		\widetilde\Sigma^\varepsilon
		=
		\big\{\z\in\Omega_T:\operatorname{dist}(\z,\Sigma)<\delta_h\big\}.
	\end{align}
	Then \(\Gamma\setminus\widetilde\Sigma^\varepsilon\) is a finite union of pairwise disjoint smooth
	hypersurface components. Moreover, for any two distinct components \(S_i\) and \(S_j\) of
	\(\Gamma\setminus\widetilde\Sigma^\varepsilon\), Assumption~\ref{assump:shock_3d} gives
	\[
	\operatorname{dist}(S_i,S_j)\ge c_{\rm sep}\delta_h
	= c_{\rm sep}h^{1/2}.
	\]
	In the shifted-interface construction below, all normal shifts are chosen to have size at most
	\(\kappa_{\rm sh}h^{1/2}\). Hence, after shifting both components, their distance remains bounded below by
	\[
	(c_{\rm sep}-2\kappa_{\rm sh})h^{1/2}>0.
	\]
	Therefore, the shifted hypersurfaces remain disjoint.

	For each component \(\Gamma_i\), we choose a nearby hypersurface
	\(\widetilde\Gamma_i\) parallel to \(\Gamma_i\) by the same mesh-strip
	pigeonhole argument as in Theorem~\ref{thm:shock}. More precisely, within the
	admissible normal-offset band of width \(O(\kappa_{\rm sh} h^{1/2})\) around
	\(\Gamma_i\), one can select
	\[
	M \sim h^{-1/2}
	\]
	pairwise disjoint candidate strips of thickness comparable to \(h\), each bounded by its own mesh-aligned hypersurface and a nearby parallel copy of \(\Gamma_i\). Since
	\[
	\int_{\Omega_T} |\nabla\cdot \mathbf F(u_{\theta^*})|\,d\z \lesssim 1,
	\]
	one of these strips, denoted by \(\Omega_i\), satisfies
	\[
	\int_{\Omega_i} |\nabla\cdot \mathbf F(u_{\theta^*})|\,d\z
	\lesssim \frac{1}{M}
	\lesssim h^{1/2}.
	\]
	We denote the two boundaries of the selected strip by \(\Gamma_i^*\) and \(\widetilde\Gamma_i\), respectively  (see Figure~\ref{fig:construction_K_3d}).

	Let
	\[
	\Omega_\Sigma
	:=
	\bigcup\big\{K\in\Lambda_h:\ K\cap\widetilde\Sigma^\varepsilon\neq\emptyset\big\}.
	\]
	On \(\Omega_T\setminus\Omega_\Sigma\), we define \(\tilde u_h\) exactly as in Section~\ref{sec:shock_wave}: we modify \(u\) near each \(\Gamma_i\setminus\Omega_\Sigma\) so that the jump set aligns with \(\Gamma_i^*\setminus\Omega_\Sigma\), and then take the cellwise \(\mathbb Q^1\) nodal interpolant on \(\Lambda_h\).
	Inside \(\Omega_\Sigma\), we define \(\tilde u_h\equiv0\).
	Since \(\Sigma\) is piecewise \(C^2\) with bounded \(\mathcal H^{d-1}\)-measure, the measure of \(\Omega_\Sigma\) is small, so its contribution is negligible (see Figure~\ref{fig:3d_shock_intersect} for a schematic illustration of the construction of \(\tilde u_h\)).

	We now record the endpoint estimate for \(\tilde u_h\). Let \(\tau\in\{0,T\}\). Away from \(\Omega_\Sigma\), the function
	\(\tilde u_h(\cdot,\tau)\) is obtained exactly as in the proof of Theorem~\ref{thm:shock}: the smooth-region
	interpolation error is \(O(h)\), while the shifted-interface strips have thickness \(O(h^{1/2})\),
	hence contribute \(O(h^{1/2})\) in \(L^1(\Omega)\). On the other hand, in \(\Omega_\Sigma\) we have
	\(\tilde u_h\equiv 0\). At \(t=\tau\), the spatial section of \(\overline{\Omega_\Sigma}\) is contained in the \(Ch^{1/2}\)-neighborhood of the spatial projection of \(\Sigma\). Since \(\Sigma\) is a finite union of compact \(C^2\) manifolds of dimension \(d-1\), a covering argument shows that this neighborhood has \(d\)-dimensional measure \(O(h^{1/2})\). Hence, by the boundedness of \(u\), the contribution from \(\Omega_\Sigma\) is \(O(h^{1/2})\). Therefore
	\begin{equation}\label{eq:u_h_3d}
		\|\tilde u_h(\cdot,\tau)-u(\cdot,\tau)\|_{L^1(\Omega)}\lesssim h^{1/2},
		\qquad \tau=0,T.
	\end{equation}

	\noindent\textbf{Step 2: Error identity and conclusion.}
	Taking $k_h=\tilde u_h$ in \eqref{eq:md_J_min} and applying the cellwise identity \eqref{eq:cellwise_entropy_ibp}, we obtain
	\begin{equation}\label{eq:md_error_identity_raw}
		\sum_{K\in\Lambda_h}\int_{\partial K}(\mathbf F(u_{\theta^*})-\mathbf F(\tilde u_h))\sgn(u_{\theta^*}-\tilde u_h)\cdot\mathbf n\,\d s
		+\int_{\Omega_T}\nabla\cdot\mathbf F(\tilde u_h)\,\sgn(u_{\theta^*}-\tilde u_h)\,\d\z
		\lesssim h.
	\end{equation}
	Hence, the contributions over those interior faces across which \(\tilde u_h\) is continuous cancel pairwise. The only remaining face terms are the jumps across \(\Gamma^*:=\bigcup_i\Gamma_{i}^*\) and \(\partial\Omega_\Sigma\), which give rise to \(I_1\) and \(I_2\), respectively. Therefore \eqref{eq:md_error_identity_raw} becomes
	\begin{equation}\label{eq:md_error_identity}
		\begin{aligned}
			&\|u_{\theta^*}(\cdot, T)-\tilde u_h(\cdot, T)\|_{L^1(\Omega)} \\
			&\le  \|u_{\theta^*}(\cdot, 0)-\tilde u_h(\cdot, 0)\|_{L^1(\Omega)}
			+\left|\int_{\Omega_T}\nabla\cdot\mathbf F(\tilde u_h)\,\sgn(u_{\theta^*}-\tilde u_h)\,\d\z\right|\\
			&\quad\, +I_1+I_2 + I_3 +Ch,
		\end{aligned}
	\end{equation}
	where
	\begin{equation*}
		\begin{aligned}
			I_1:= & \sum_i\int_{\Gamma_{i}^*}\llbracket(\mathbf F(\tilde u_h)-\mathbf F(u_{\theta^*}))\sgn(\tilde u_h-u_{\theta^*})\rrbracket\cdot\mathbf n_{\Gamma^*_i}\,\d s, \\
			I_2:= & \sum_{e\subset\partial\Omega_{\Sigma}}\left|\int_e\llbracket(\mathbf F(\tilde u_h)-\mathbf F(u_{\theta^*}))\sgn(\tilde u_h-u_{\theta^*})\rrbracket\cdot\mathbf n_e\,\d s\right|, \\
			I_3:= & \left|\int_{\partial\Omega\times(0,T)}
			\left(\mathbf f(u_{\theta^*})-\mathbf f(\tilde u_h)\right)
			\sgn(u_{\theta^*}-\tilde u_h)\cdot\mathbf n_{\partial\Omega}\,\d s \right| \\
            &\le
			C\|u_{\theta^*}\|_{L^1(\partial\Omega\times(0,T))}
			\le
			C\mathcal L(u_{\theta^*})
			\lesssim h.
		\end{aligned}
	\end{equation*}

	We first estimate the volume term. Outside $\Omega_{\Sigma}$, the same argument as in \eqref{eq:estimate_2nd_term} yields
	\[
	\int_{\Omega_T\setminus\Omega_{\Sigma}}|\nabla\cdot\mathbf F(\tilde u_h)|\,\d\z\lesssim h+h^{\frac{1}{2}}.
	\]
	Inside $\Omega_{\Sigma}$ we have $|\nabla\!\cdot\!\mathbf{F}(\tilde{u}_h)| \equiv 0$.
	Therefore,
	\begin{equation}\label{eq:md_volume_term}
		\left|\int_{\Omega_T}\nabla\cdot\mathbf F(\tilde u_h)\,\sgn(u_{\theta^*}-\tilde u_h)\,\d\z\right|
		\lesssim h^{\frac{1}{2}}.
	\end{equation}

	Next, we estimate \(I_1\).
	For each component \(i\), we transfer the integral from \(\Gamma_i^*\)
	to the shifted shock surface $\widetilde{\Gamma}_i$ by integrating over the strip enclosed between them (see the schematic diagram in Figure~\ref{fig:construction_K_3d}).
	Proceeding as in Step~5 of the proof of Theorem~\ref{thm:shock}
	(corresponding to~\eqref{eq:estimate_int_Gamma_new}),
	we obtain the same estimate, with the only additional point that in several
	dimensions, the lateral boundary of the strip also contains side pieces near
	\(\Sigma\), in addition to the parts at \(t=0,T\). Both types of boundary
	pieces have \(d\)-dimensional measure \(O(h^{1/2})\): the former because the
	strip thickness is \(O(h^{1/2})\), and the latter because they are obtained by
	thickening \((d-1)\)-dimensional boundary pieces by a distance \(O(h^{1/2})\).
	Since the jump integrand is uniformly bounded, the corresponding boundary terms
	are still \(O(h^{1/2})\). Hence
	\begin{equation}\label{eq:md_I1}
		I_1\lesssim h^{\frac{1}{2}}.
	\end{equation}

	For \(I_2\), the jump integrand is uniformly bounded as in  \eqref{eq:estimate_jump}, so it remains to estimate \(\mathcal H^d(\partial\Omega_\Sigma)\). Here \(\Omega_\Sigma\) is a Cartesian-grid approximation of the tubular neighborhood \(\widetilde\Sigma^\varepsilon\) of radius \(\delta_h=h^{1/2}\) around the \(C^2\) manifold \(\Sigma\). Since \(\mathcal H^{d-1}(\Sigma)\lesssim 1\), for all sufficiently small \(h\), the boundary of this tubular neighborhood satisfies \[
    \mathcal H^d(\partial\widetilde\Sigma^\varepsilon)
    \lesssim \delta_h
    \lesssim h^{1/2}.
    \]
    Because the mesh is uniform Cartesian, a standard covering estimate for grid-aligned approximations of smooth hypersurfaces gives
    \[
    \mathcal H^d(\partial\Omega_\Sigma) \lesssim \mathcal  H^d(\partial\widetilde\Sigma^\varepsilon)\lesssim h^{1/2},
    \]
    which yields
    \begin{equation}\label{eq:md_I2}
    I_2 \lesssim h^{1/2}.
    \end{equation}
    This is precisely where the codimension-two nature of \(\Sigma\) is
    used: the \(d\)-dimensional measure of the boundary of a tubular
    neighborhood of radius \(h^{1/2}\) is \(O(h^{1/2})\), with the same
    scaling in every fixed dimension \(d\).

	Substituting \eqref{eq:md_volume_term}-\eqref{eq:md_I2} into \eqref{eq:md_error_identity} yields
	\[
	\|u_{\theta^*}(\cdot, T)-\tilde u_h(\cdot, T)\|_{L^1(\Omega)}
	\lesssim
	\|u_{\theta^*}(\cdot, 0)-\tilde u_h(\cdot, 0)\|_{L^1(\Omega)}
	+h^{\frac{1}{2}}.
	\]
	By the endpoint estimate \eqref{eq:u_h_3d} and
	\[
	\|u_{\theta^*}(\cdot, 0)-u(\cdot, 0)\|_{L^1(\Omega)}\le \L(u_{\theta^*})\lesssim h,
	\]
	the triangle inequality gives
	\[
	\int_\Omega |u_{\theta^*}-u|(\x,T)\,\d\x\lesssim h^{1/2}.
	\]
	This completes the proof of Theorem~\ref{thm:more_general_case}.
\end{proof}

\begin{remark}
	The proof shows that, in several space dimensions, the only additional ingredient beyond the one-dimensional analysis is the excision of a tubular neighborhood of the codimension-two interaction set. Therefore, whenever the discontinuity set consists of finitely many pairwise disjoint smooth hypersurfaces, the corresponding one-dimensional arguments apply without change. By combining those local constructions with the present surgery near $\Sigma$, one can treat finitely many multidimensional rarefaction regions and mixed shock-rarefaction configurations as well. We omit the routine details.
\end{remark}

\section{Numerical experiments}\label{sec:numerics}
We implement Algorithm~\ref{alg-1} in JAX. The experiments comprise
 standing and moving shocks, rarefaction waves, the interaction of two shocks, a sine-wave initial condition, the Buckley--Leverett equation, conservation laws with cubic and sinusoidal fluxes, and a two-dimensional three-shock configuration for Burgers' equation.

Let \(u_{\theta^*}\) denote the numerical approximation obtained from Algorithm \ref{alg-1}. To evaluate performance, we compute the relative \(L^1\) error at the final time \(T\) and the space--time relative \(L^1\) error, defined as
\[
\mathcal{E}_r^T(\theta^*) = \frac{\int_{\Omega} |u_{\theta^*}(\x, T) - u(\x, T)| \, \d \x}{\int_{\Omega} |u(\x, T)| \, \d \x}, \quad \mathcal{E}_r(\theta^*) = \frac{\int_{\Omega_T} |u_{\theta^*}(\x, t) - u(\x, t)| \, \d\z}{\int_{\Omega_T} |u(\x, t)| \, \d\z},
\]
respectively. These metrics assess accuracy at the final time and over the entire space--time domain. The integrals are approximated using standard quadrature rules.

For the loss function \eqref{eq:loss_func}, all integrals are approximated
 by the composite trapezoidal rule on uniform grids. To approximate the maximization over \(k_h\in V_h^c\), let
 $\{\ell_\alpha\}_{\alpha\in\mathcal I}$ be the local
 $\mathbb Q^1(K)$ Lagrange basis on each cell $K$, fix a perturbation
 bound $0<b<\frac{c}{4(d+1)+2}$ and a sample count $N_{\rm pert}\in\mathbb N$, and set
\begin{equation}\label{eq:avg_u}
 \bar u_K:=\frac{1}{|K|}\int_Ku_\theta\,\d\z,\qquad
	 (\operatorname{Avg}_hu_\theta)|_K
	 =\bar u_K=\sum_{\alpha\in\mathcal I}\bar u_K\ell_\alpha .
\end{equation}
We generate the candidate functions $k_{h,j}$ as follows:
\begin{equation}\label{eq:gen_pert_for_k_h}
	\begin{aligned}
	k_{h,j}|_K :=\bar u_K+
	 \sum_{\alpha\in\mathcal I}\epsilon_{j,K}^\alpha\ell_\alpha,
	 \quad \epsilon_{j,K}^\alpha\sim\mathcal U(-b,b),
	 \quad j=1,\ldots,N_{\rm pert}.
	\end{aligned}
\end{equation}
The variables \(\epsilon_{j,K}^\alpha\) are sampled independently.

\begin{algorithm}
	\caption{Our algorithm}\label{alg-1}
	\begin{algorithmic}[1]
		\Require Initial datum $u_0$, lateral trace $g$, comparison bound $c$, $N_{\rm train}$, $N_{\rm pert}$, perturbation bound $b>0$, learning rate $\eta$, $u_\theta$
		\Ensure Best neural network $u_{\theta^*}$
		\State Randomly initialize the network parameters $\theta$ and generate partition $\Lambda_h$ for $\Omega_T$.
		\State Initialize best loss $\L_{\min} \leftarrow \infty$
		\For{$i = 1$ \textbf{to} $N_{\rm train}$ }
		\State  Compute \(\bar u_K\) as in \eqref{eq:avg_u} and generate \(k_{h,j}\) by \eqref{eq:gen_pert_for_k_h}.
		\State  Compute $$\hat{\mathscr{L}}(\theta) := \J_{\mathrm{ent}} (u_\theta; k_{h}^*) + \mathcal L_{\mathrm{ibc}}^g(u_\theta) + \L_{\mathrm{reg}}(u_\theta), \quad
		k_h^*\in\underset{k_h\in\{k_{h,j}:1\le j\le N_{\rm pert}\}}{\arg\max}
		\J_{\mathrm{ent}}(u_\theta;k_h). $$
		\If{$i>0.9N_{\rm train}$ and $\hat{\mathscr{L}}(\theta) < \L_{\min}$ }
		\State $\L_{\min} \leftarrow \hat{\mathscr{L}}(\theta)$, $\theta^* \leftarrow \theta$
		\EndIf
		\State Take one Adam step with learning rate \(\eta\), using \(\hat{\mathscr L}\) and treating \(k_h^*\) as fixed.
		\EndFor
	\end{algorithmic}
\end{algorithm}
In Table~\ref{tab:hyperparameter}, when \(N_T>1\), we divide \([0,T]\) into \(N_T\) successive strips and apply
 Algorithm~\ref{alg-1} sequentially, using the terminal network trace on
 one strip as the initial datum on the next; \(N_t\) is the number of
 temporal cells on each strip. The common hyperparameters are \(b=5\)
 and \(\eta=0.001\); the same fixed \(c\) chosen in
 Section~\ref{subsec:notation} is used in the clipping and projection
 steps. Here, \(l_\theta\) and \(n_\theta\) denote the width and number of hidden
 layers, and \(N_x\) denotes the number of spatial cells.
	\begin{table}[h!]
		\centering
		\caption{Training hyperparameters and relative $L^1$ errors}\label{tab:hyperparameter}
		\renewcommand{\arraystretch}{1.0} 
		\setlength{\tabcolsep}{4pt}
		\small
		\begin{tabular}{c | c c c c c c c c c}
		\hline
		Experiment   & $l_{\theta}$   &  $n_{\theta}$   &  $N_T$   & $N_t$   &  $N_x$   &  $N_{\rm train}$   & $N_{\rm pert}$   &  $\mathcal{E}_r^T\left(\theta^*\right)$   &  $\mathcal{E}_r\left(\theta^*\right)$  \\
		\hline
		Standing shock & 64  & 4 & 1 & 64 & 128 & 1E4 & 5E4 &  0.513\% & 0.434\%\\
		\hline
		Moving shock  & 64  & 4 & 2 & 64 & 128 & 1E4 & 5E4 & 0.272\% &  0.431\% \\
		\hline
		Rarefaction wave &  64  & 4 & 1 & 32 & 64 & 1E4 & 2E4 & 0.294\% & 0.272\% \\
		\hline
		Interaction of two shocks & 64  & 4 & 2 & 64 & 128 & 2E4 & 5E4 & 0.252\% & 0.217\%\\
		\hline
		Sinusoidal initial data &  64  & 4 & 2 & 64 & 128 & 5E4 & 5E4 & 0.248\% & 0.205\%  \\
		\hline
		Cubic flux & 64  & 4 & 2 & 32 & 128 & 2E4 & 2E4 & 0.270\% & 0.199\% \\
		\hline
		Buckley-Leverett equation & 64  & 4 & 2 & 32 & 128 & 5E4 & 5E4 & 0.372\% & 0.358\% \\
		\hline
		Sine flux & 64  & 4 & 2 & 32 & 128 & 5E4 & 5E4 & 0.541\% & 0.318\% \\
		\hline
		2D Burgers & 64  & 4 & 3 & 20 & $40^2$ & 2E4 & 5E4 & 0.519\% & 0.493\%\\
		\hline
		\end{tabular}
	\end{table}

\subsection{Inviscid Burgers equation}
\subsubsection{Standing and moving shocks}
We consider the one-dimensional HCL \eqref{eq:scalar-HCL} with convex flux $f(u)=\frac{u^2}{2}$, i.e. Burgers' equation in $[-1,1] \times[0,0.5]$ for a standing shock at $x=0$ and a
 shock moving with speed $1$:
\[
 u^{\rm st}(x,t)=
\begin{cases}
 1,   &  x\le0,  \\
 -1,   &  x>0,
\end{cases}
 \qquad
 u^{\rm mov}(x,t)=
\begin{cases}
 2,   &  x\le t,  \\
 0,   &  x>t.
\end{cases}
\]
\begin{figure}[h]
	\centering
	\begin{subfigure}[b]{0.47\textwidth}
\centering
\includegraphics[width=\linewidth]{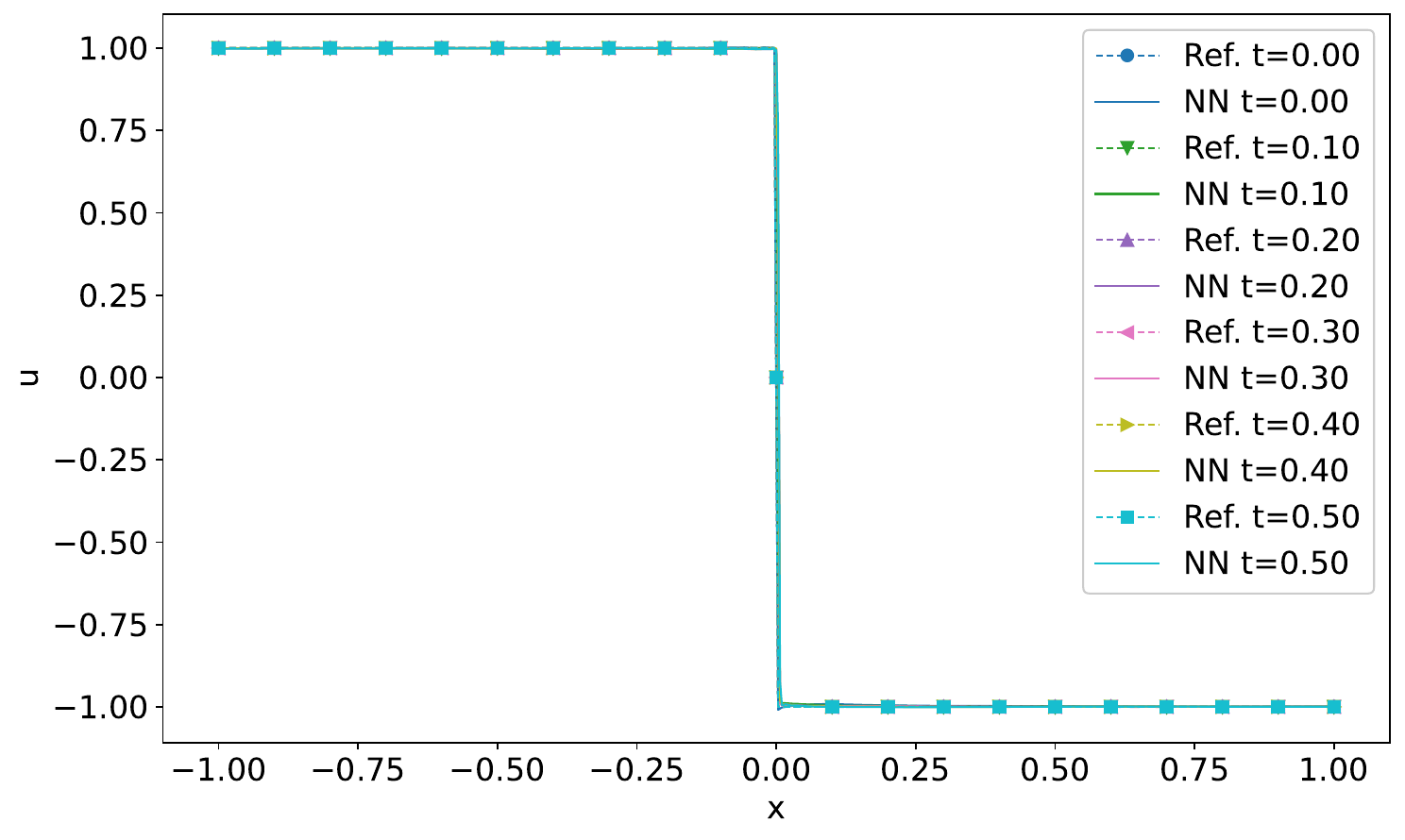}
\caption{Standing shock solution up to $T=0.5$.}
\end{subfigure}
	\hfill
	\begin{subfigure}[b]{0.47\textwidth}
\centering
\includegraphics[width=\linewidth]{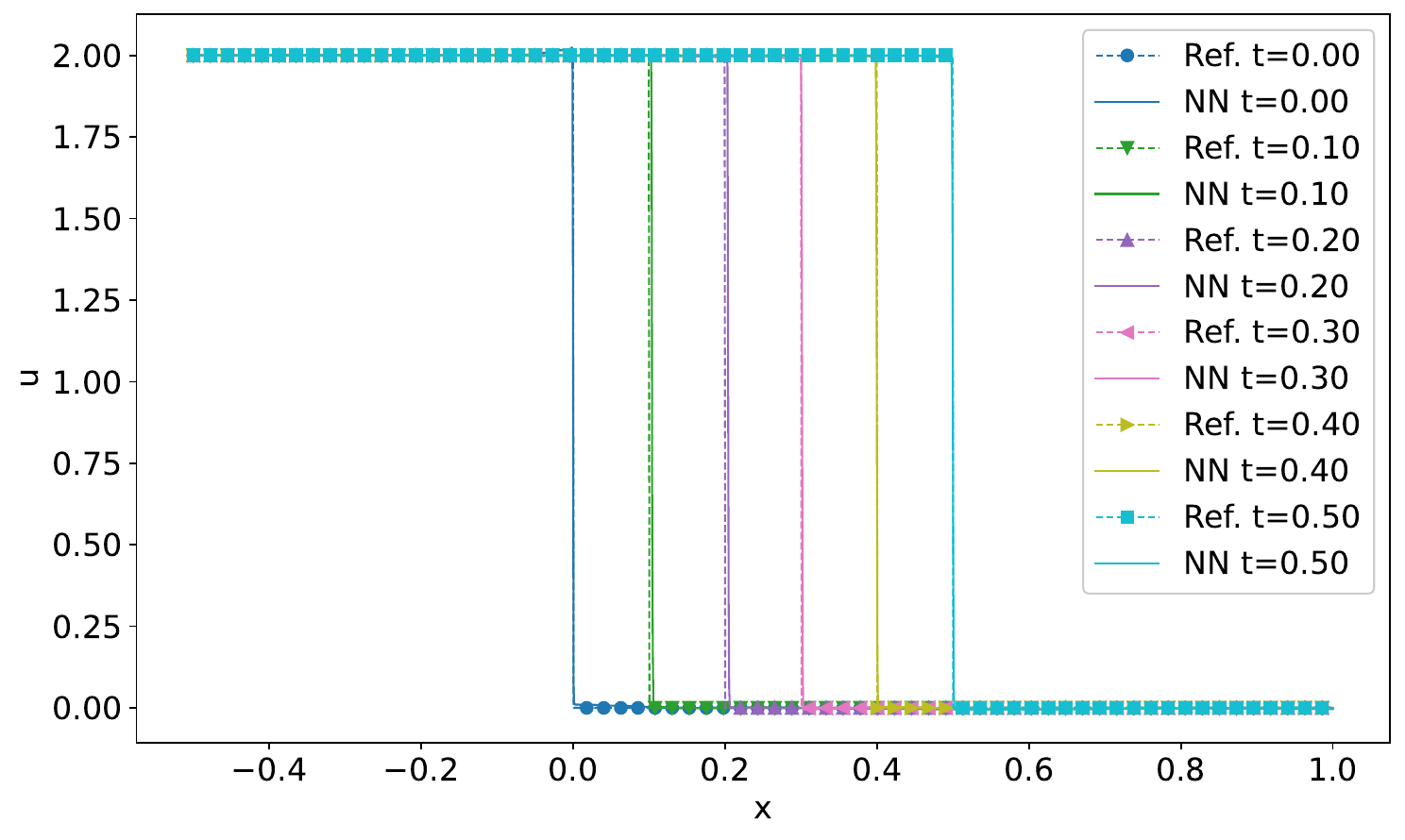}
\caption{Moving shock solution up to $T=0.5$.}
\end{subfigure}
	\par\medskip
	\begin{subfigure}[b]{0.47\textwidth}
\centering
\includegraphics[width=\linewidth]{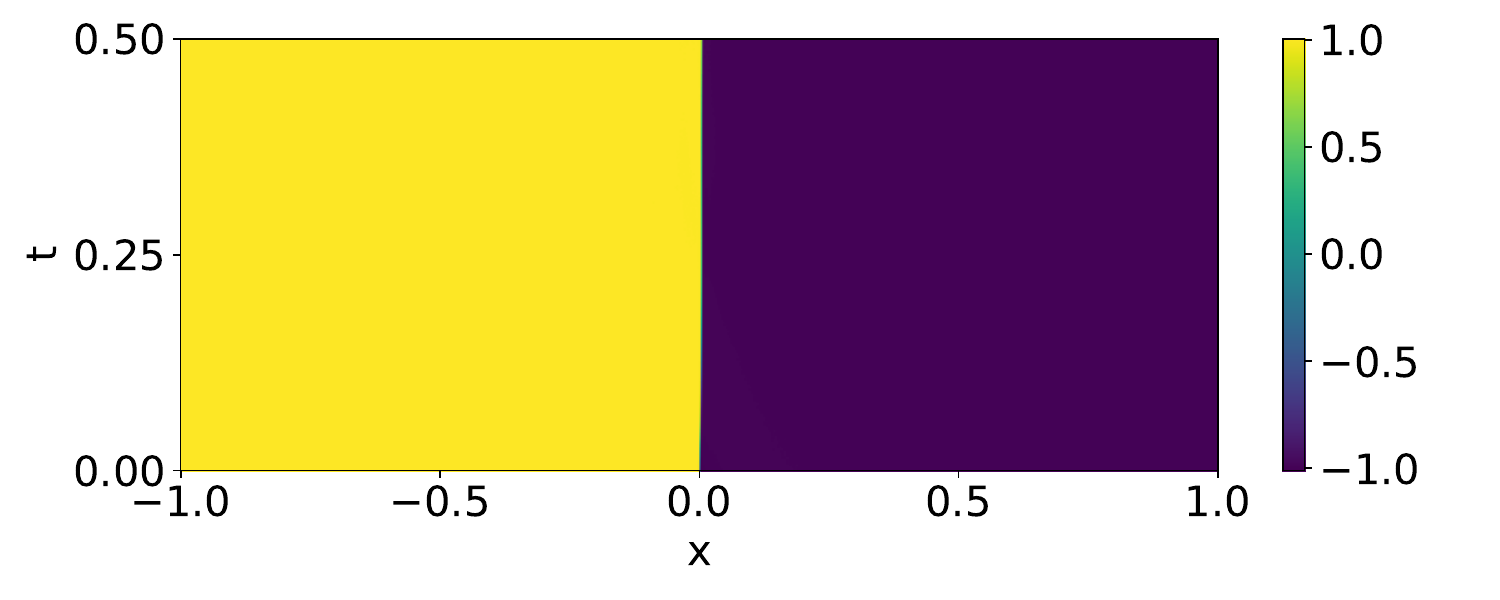}
\caption{NN solution to the standing shock}
\end{subfigure}
	\hfill
	\begin{subfigure}[b]{0.47\textwidth}
\centering
\includegraphics[width=\linewidth]{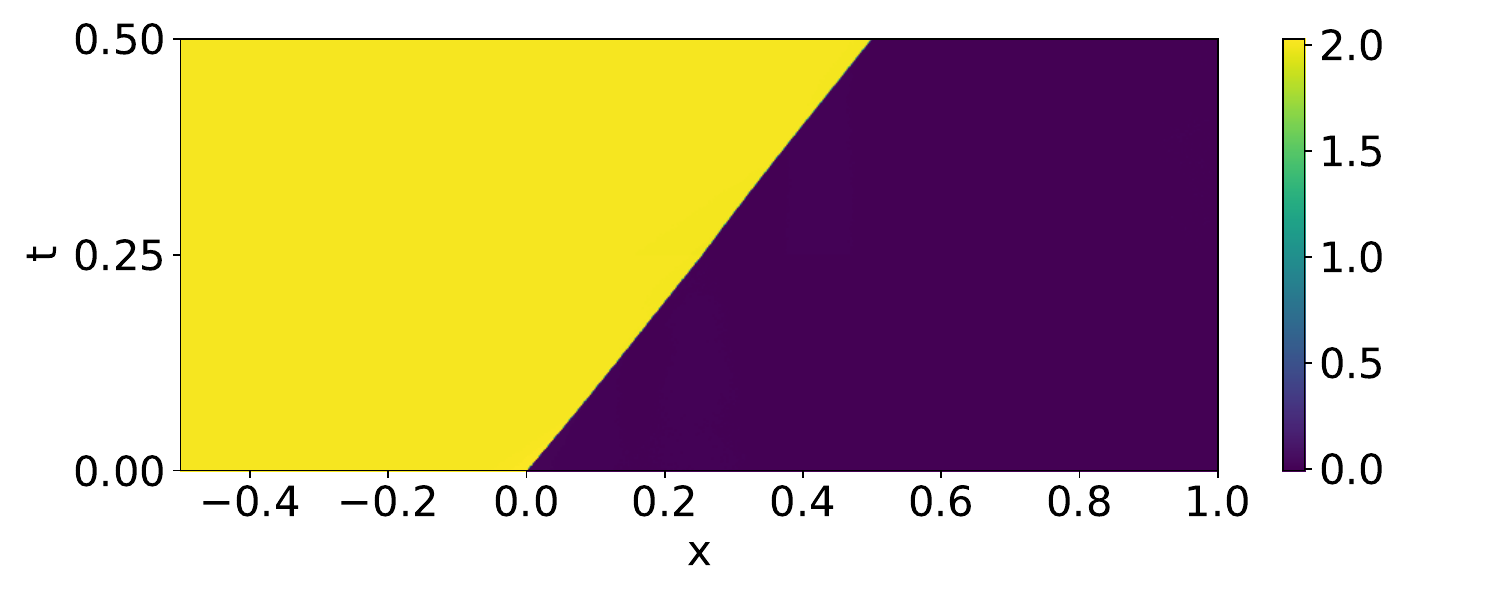}
\caption{NN solution to a moving shock}
\end{subfigure}
        \caption{Exact and neural network solutions for Burgers' equation with shocks.}
	\label{fig:standing_shock_strong_moving_shock}
\end{figure}
With the hyperparameter settings reported in Table~\ref{tab:hyperparameter} and the training procedure described in Algorithm~\ref{alg-1}, Figure~\ref{fig:standing_shock_strong_moving_shock} shows that the neural network accurately captures both the standing and moving shocks, despite the highly nonconvex min--max structure of the training objective.

\subsubsection{Rarefaction wave}
We consider Burgers' equation with initial data and exact solution given,
$$
u_0(x)= \begin{cases}-1 & x \leq 0 \\ 1 & x>0\end{cases}, \quad u(x, t)= \begin{cases}-1 & x \leq-t \\ \frac{x}{t} & -t<x \leq t \\ 1 & x>t\end{cases}.
$$
This corresponds to a rarefaction wave. Figure~\ref{fig:rarefaction} shows the NN solution, and Table~\ref{tab:hyperparameter} reports the final time error.
\begin{figure}[h!]
\centering
\begin{subfigure}[b]{0.6\textwidth}
\includegraphics[width=\linewidth]{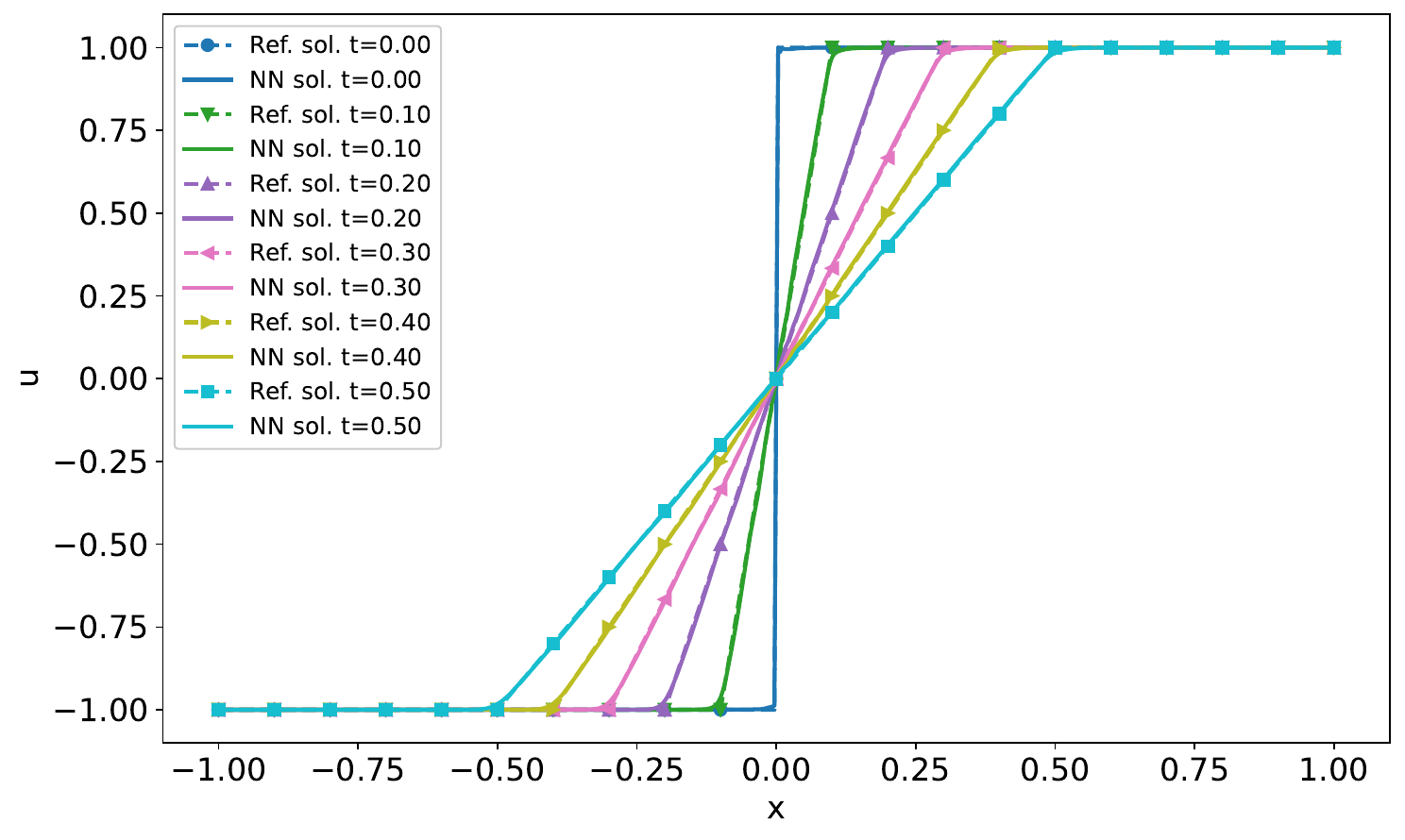}
\caption{Rarefaction wave up to $T=0.5$}
\end{subfigure}\\
\hspace{1.6cm}
\begin{subfigure}[b]{0.59\textwidth}
\includegraphics[width=\linewidth]{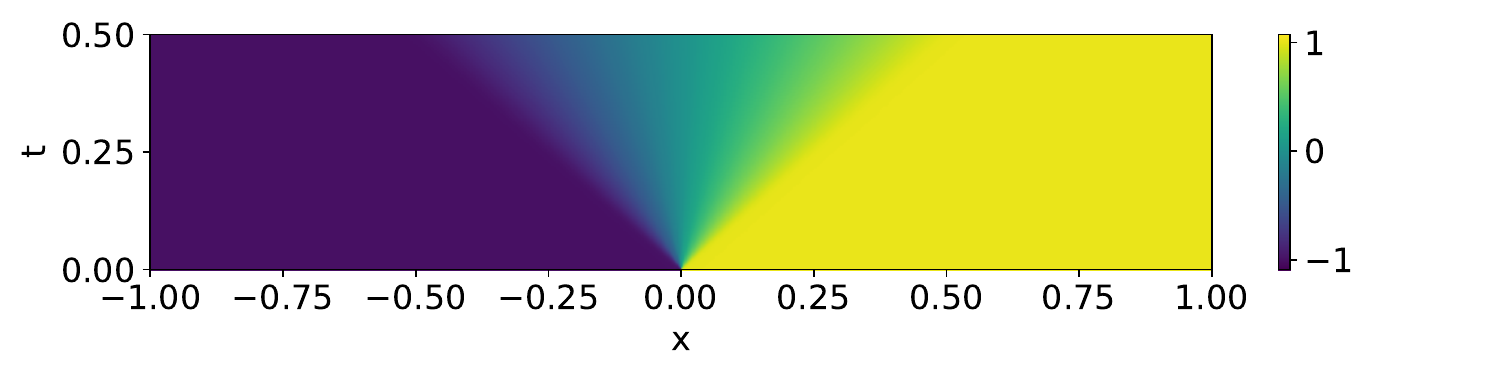}
\caption{NN solution to a rarefaction}
\end{subfigure}
\caption{Exact solutions and NN solutions for the Burgers’ equation with rarefaction.}
\label{fig:rarefaction}
\end{figure}

\subsubsection{Interaction of two shocks}
We consider Burgers' equation on \([0,1]\times[0,0.5]\) with exact solution:
\[
u(x, t)=  \begin{cases}
	\begin{cases}
		0.8 & x \leq 0.3 + 0.3t \\
		-0.2 & 0.3 + 0.3t < x \leq 0.7 - 0.9t \\
		-1.6 & x > 0.7 - 0.9t
	\end{cases} & t \leq \frac{1}{3} \\[2ex]
	\begin{cases}
		0.8 & x \leq 0.4 - 0.4\left(t - \frac{1}{3}\right) \\
		-1.6 & x > 0.4 - 0.4\left(t - \frac{1}{3}\right)
	\end{cases} & t > \frac{1}{3}
\end{cases}.
\]
This solution describes the interaction of two shocks. The NN solutions are plotted in Figure~\ref{fig:two_shocks}. The final time error is reported in Table~\ref{tab:hyperparameter}.
\begin{figure}[htbp!]
\includegraphics[width=0.45\textwidth]{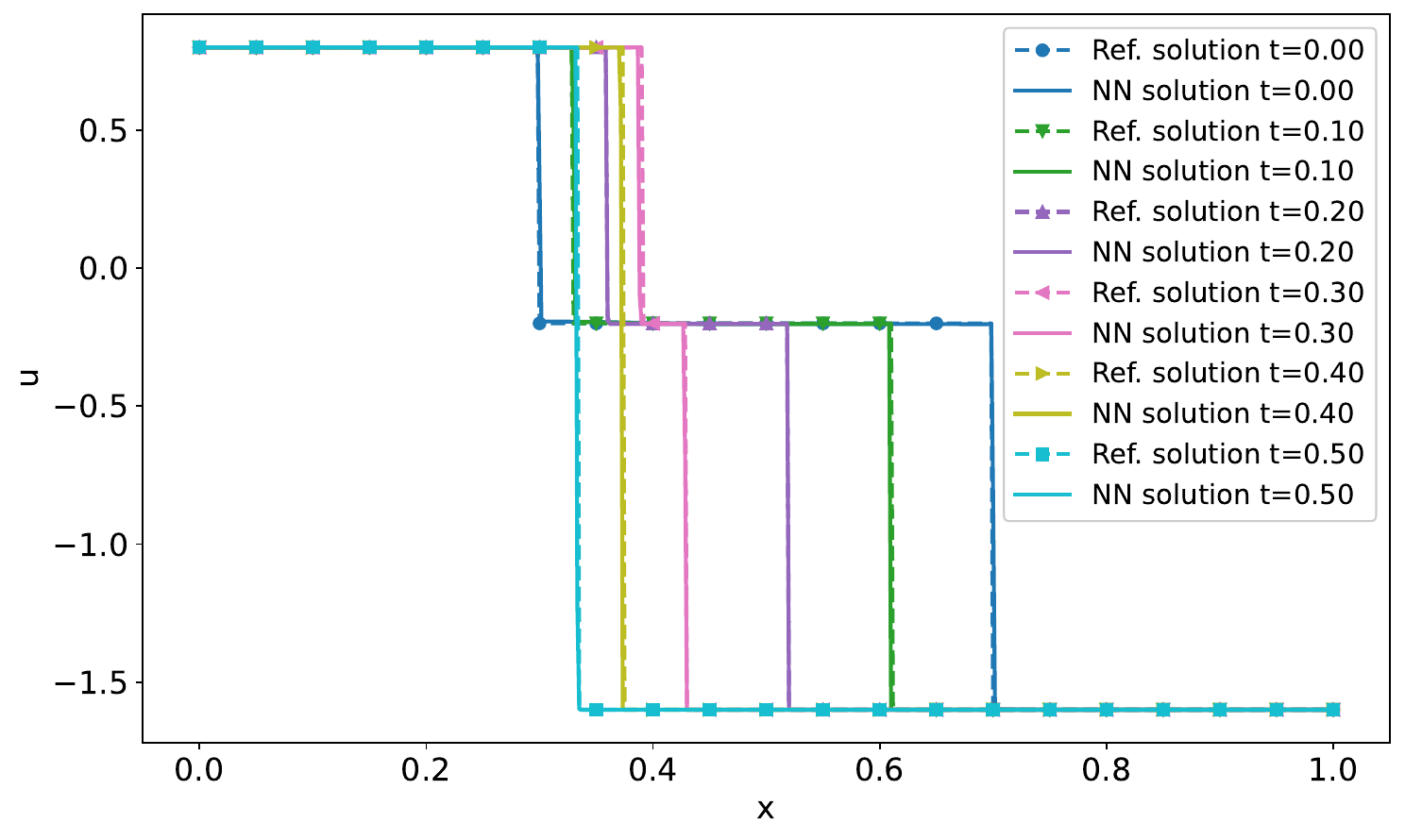}%
\includegraphics[width=0.55\textwidth]{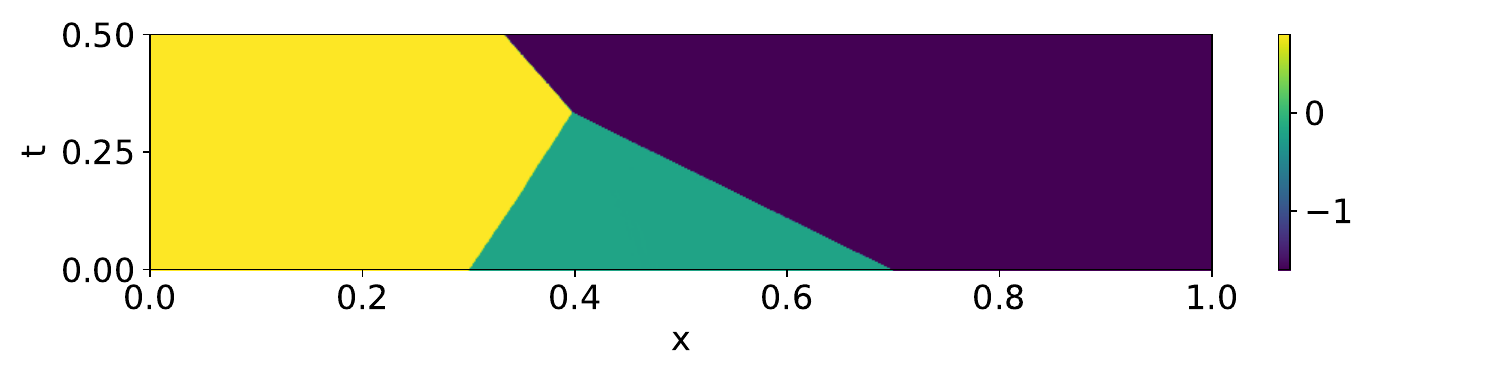}
        \caption{Exact solutions and NN solutions for Burgers' equation with two shocks.}
	\label{fig:two_shocks}
\end{figure}

\subsubsection{Sine wave initial datum}\label{sec:sine-wave}
We consider Burgers' equation with the initial data $u_0(x)=-\sin (\pi x)$
and zero Dirichlet boundary conditions in the space--time domain $[-1,1] \times[0,1]$. The solution develops a stationary shock at $x=0$ for $t>1/\pi$. The NN solutions are plotted in Figure~\ref{fig:sine_wave}. The final time error is reported in Table~\ref{tab:hyperparameter}.
\begin{figure}[h!]
\centering
\includegraphics[width=0.49\linewidth]{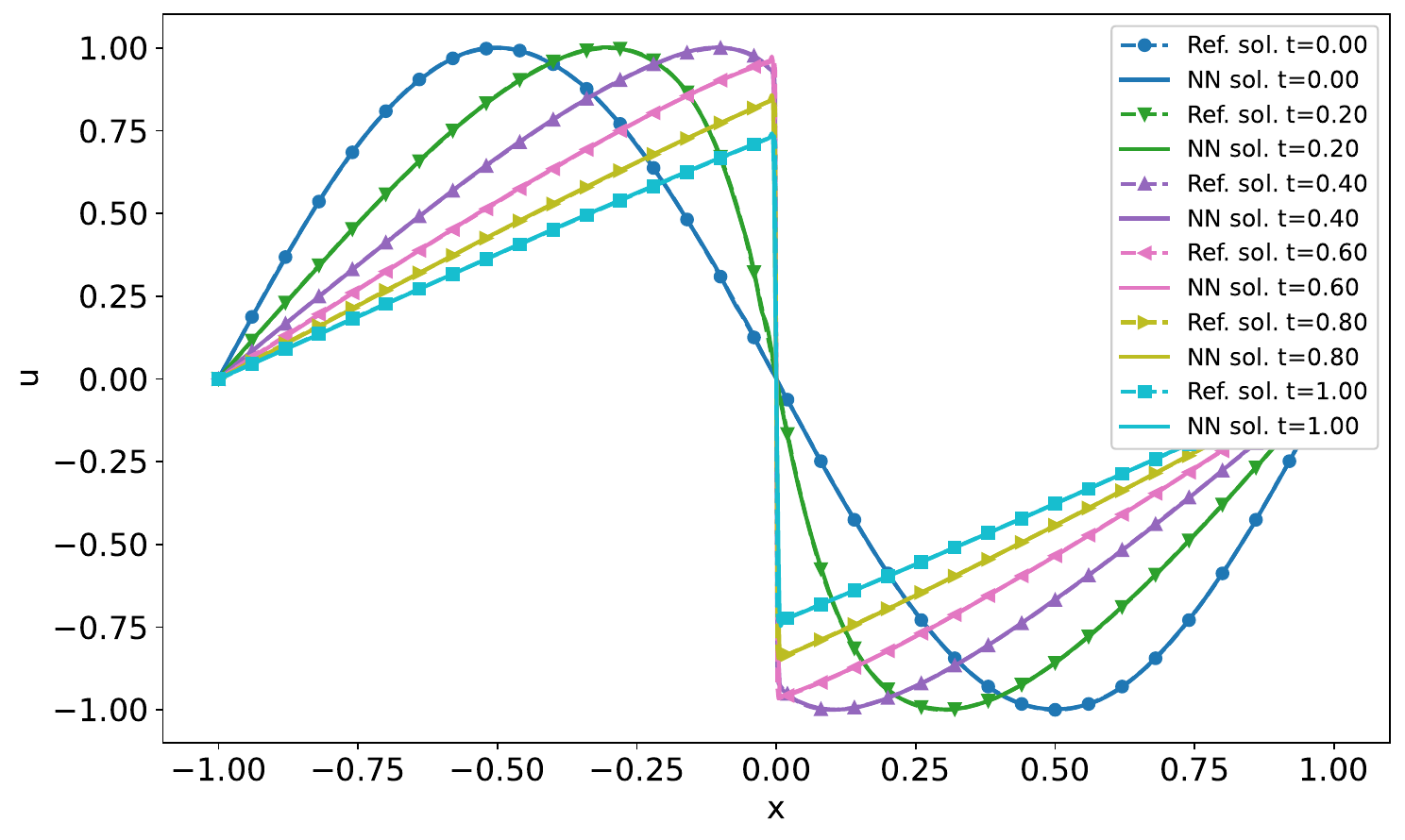}%
\includegraphics[width=0.49\linewidth]{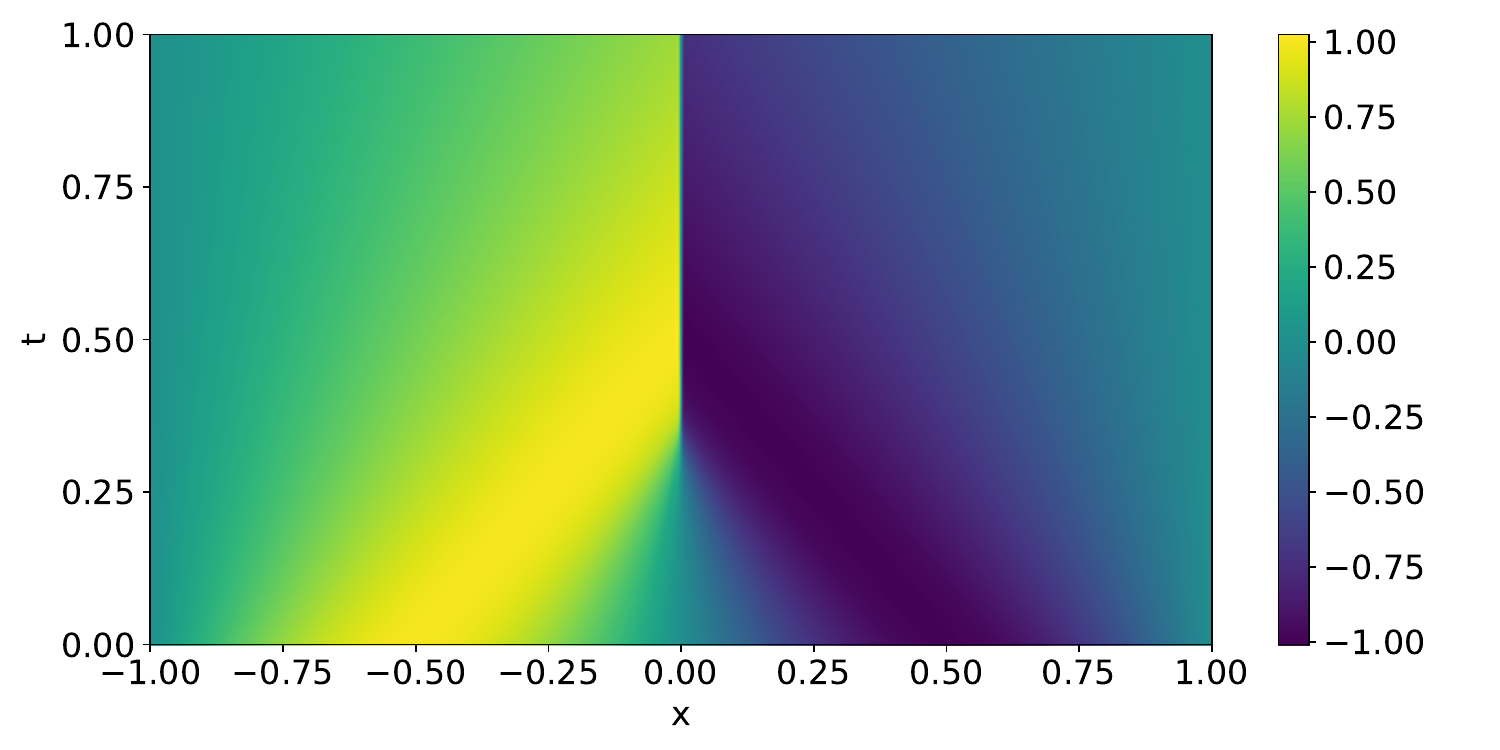}
\caption{Reference and neural network solutions for the Burgers’ equation with sinusoidal initial data.}
\label{fig:sine_wave}
\end{figure}

To examine empirical convergence under mesh refinement,
we train the neural network until the loss stabilizes.
The final-time and space--time errors are plotted in
Figure~\ref{fig:convergence}.
\begin{figure}[h!]
\centering
\includegraphics[scale=0.35]{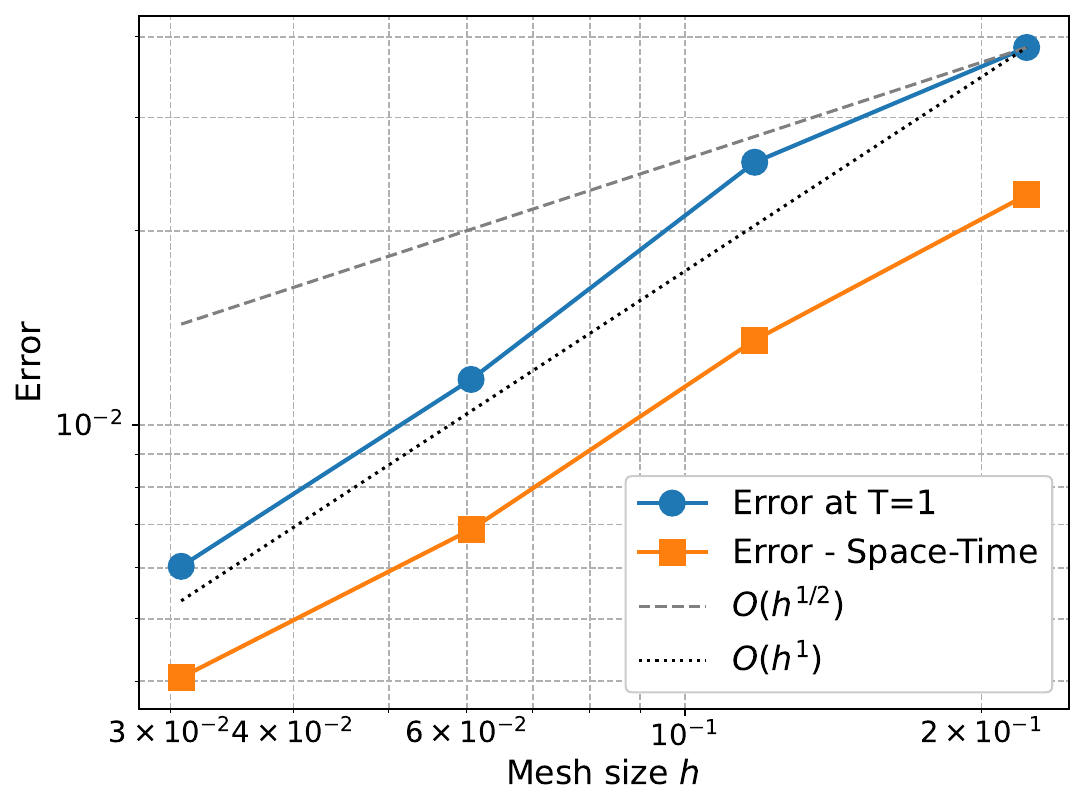}
\caption{Space--time relative \(L^1\)-error and final-time relative \(L^1\)-error at \(T=1\) versus mesh size.}
\label{fig:convergence}
\end{figure}

\begin{remark}\label{rem:convergence}
The \(O(h^{1/2})\) estimate in Theorem~\ref{thm:shock} is a consequence of the present stability analysis.
Once the exact shock interface is replaced by a mesh-aligned one, no direct estimate is available for the induced interface discrepancy term, which is why the shifted-interface argument is introduced.
For the constructed \(u_{\hat\theta}\)  in Lemma~\ref{lem:Lu_shock}, the corresponding term is of order \(O(h)\). Therefore
\[
\|u_{\hat\theta}(\cdot,T)-u(\cdot,T)\|_{L^1(\Omega)} \le Ch.
\]
This indicates that the rigorous \(O(h^{1/2})\) bound may not be sharp.
The convergence experiments in Figure~\ref{fig:convergence} are consistent with this observation.
\end{remark}

\subsection{Scalar conservation laws with non-convex flux}

We now turn from Burgers' equation, whose flux is convex, to scalar 
conservation laws with non-convex fluxes. In this regime, the entropy 
solution generally exhibits compound waves combining rarefactions and 
shocks, the shock speeds are determined by the convex-hull (Oleinik) 
construction rather than by Rankine--Hugoniot conditions alone, and the 
entropy admissibility condition plays an essential role in selecting the 
physically relevant solution. These features make non-convex problems a 
particularly demanding test for the proposed entropy-based neural network 
method.

\subsubsection{Cubic flux}

As a first non-convex example, we consider the one-dimensional 
HCL with cubic flux $f(u) \;=\; \tfrac{1}{3}\,u^{3}$, posed on the spatial domain $[-1,1]$ with Riemann initial data
\[
u_{0}(x) \;=\;
\begin{cases}
\phantom{-}1, & x < 0, \\
-1, & x \geq 0.
\end{cases}
\]
Since the jump $u_{L}\to u_{R}$ straddles the inflection point $u = 0$ 
of the flux, neither a single shock nor a single rarefaction is admissible.
The entropy solution is therefore a compound wave consisting of a 
rarefaction attached to a trailing shock; for $t > 0$, the shock 
propagates at speed $s = 0.25$ and connects the states $u = 1$ and 
$u = -0.5$.

A reference solution is computed using a third-order strong stability 
preserving Runge--Kutta scheme combined with a fifth-order WENO 
reconstruction on a very fine mesh. The neural network approximation is 
displayed in Figure~\ref{fig:cubic_flux}, and the corresponding error at 
the final time $T = 0.5$ is reported in Table~\ref{tab:hyperparameter}. 
Both the rarefaction part and the trailing shock are accurately resolved,
without spurious oscillations near the discontinuity.

\begin{figure}[h!]
\centering
\begin{subfigure}[b]{0.6\textwidth}
\includegraphics[width=\linewidth]{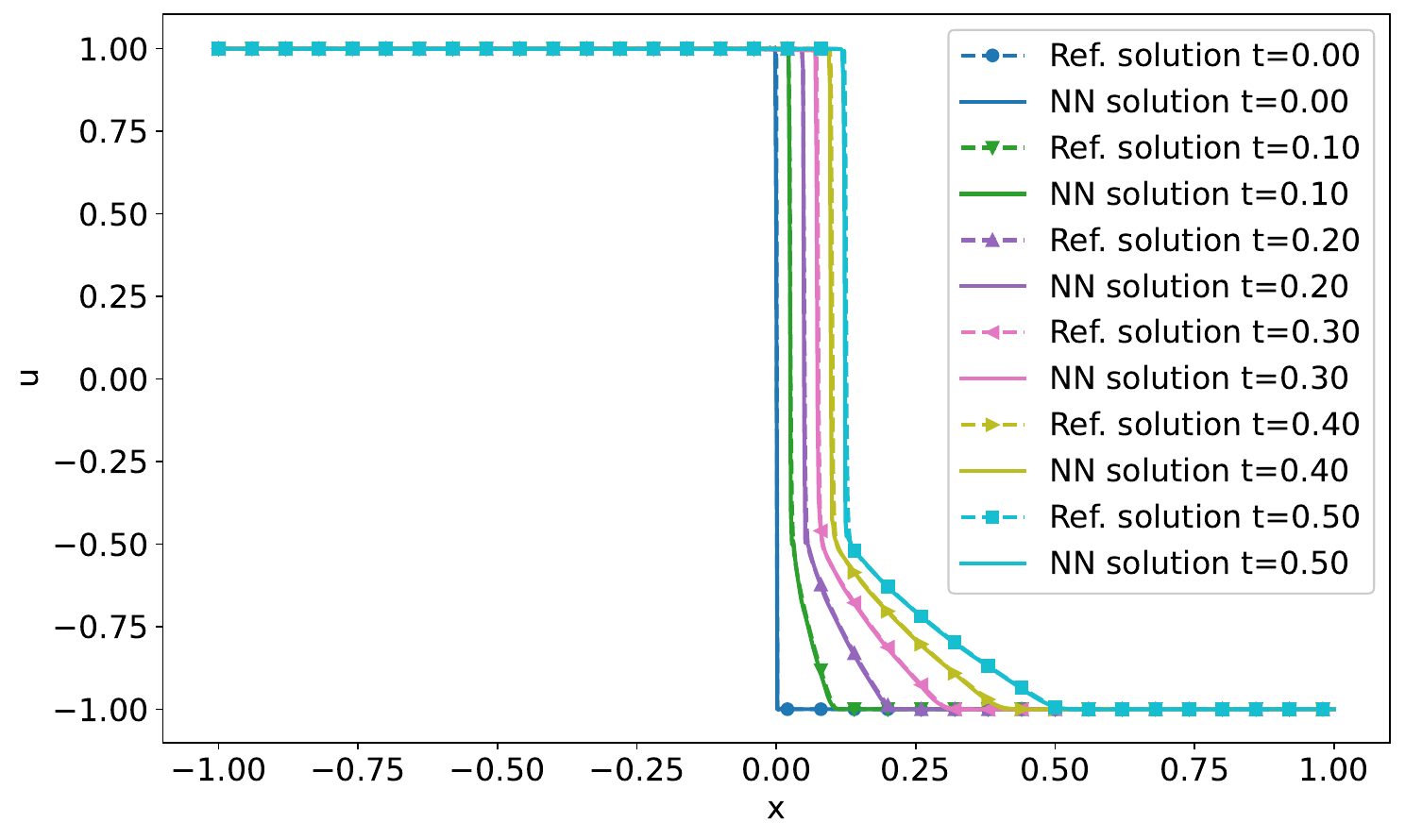}
\caption{Reference and NN solutions up to $T=0.5$.}
\end{subfigure}\\
\hspace{1.6cm}
\begin{subfigure}[b]{0.59\textwidth}
\includegraphics[width=\linewidth]{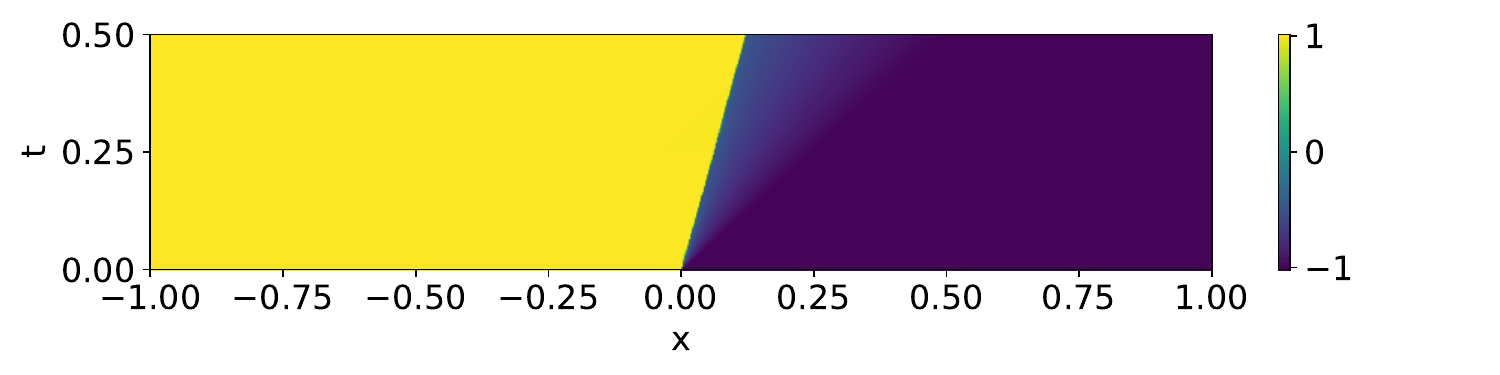}
\caption{Space--time NN solution.}
\end{subfigure}
\caption{Reference solutions and NN solutions for the conservation law with cubic flux.}
\label{fig:cubic_flux}
\end{figure}

\subsubsection{The Buckley--Leverett equation}

We next consider the Buckley--Leverett equation, a one-dimensional scalar HCL with the non-convex flux $f(u) =\frac{u^{2}}{u^{2} + \tfrac{1}{2}(1-u)^{2}}$ on $[-1,1]$ with Riemann initial data $u_{0}(x) \;=\;
\begin{cases}
1, & x < 0, \\
0, & x \geq 0.
\end{cases}$.
The entropy solution consists of a rarefaction fan attached
to a leading shock, determined by the upper concave envelope
of \(f\) on \([0,1]\). As such,
this test is a stringent benchmark for non-convex scalar conservation laws.
A reference solution is computed using a third-order strong stability
preserving Runge--Kutta scheme combined with a fifth-order WENO
reconstruction on a very fine mesh. The neural network approximation is
shown in Figure~\ref{fig:BLE}, and the corresponding error at the final
time $T=0.5$ is reported in Table~\ref{tab:hyperparameter}.
\begin{figure}[h!]
\includegraphics[width=0.45\linewidth]{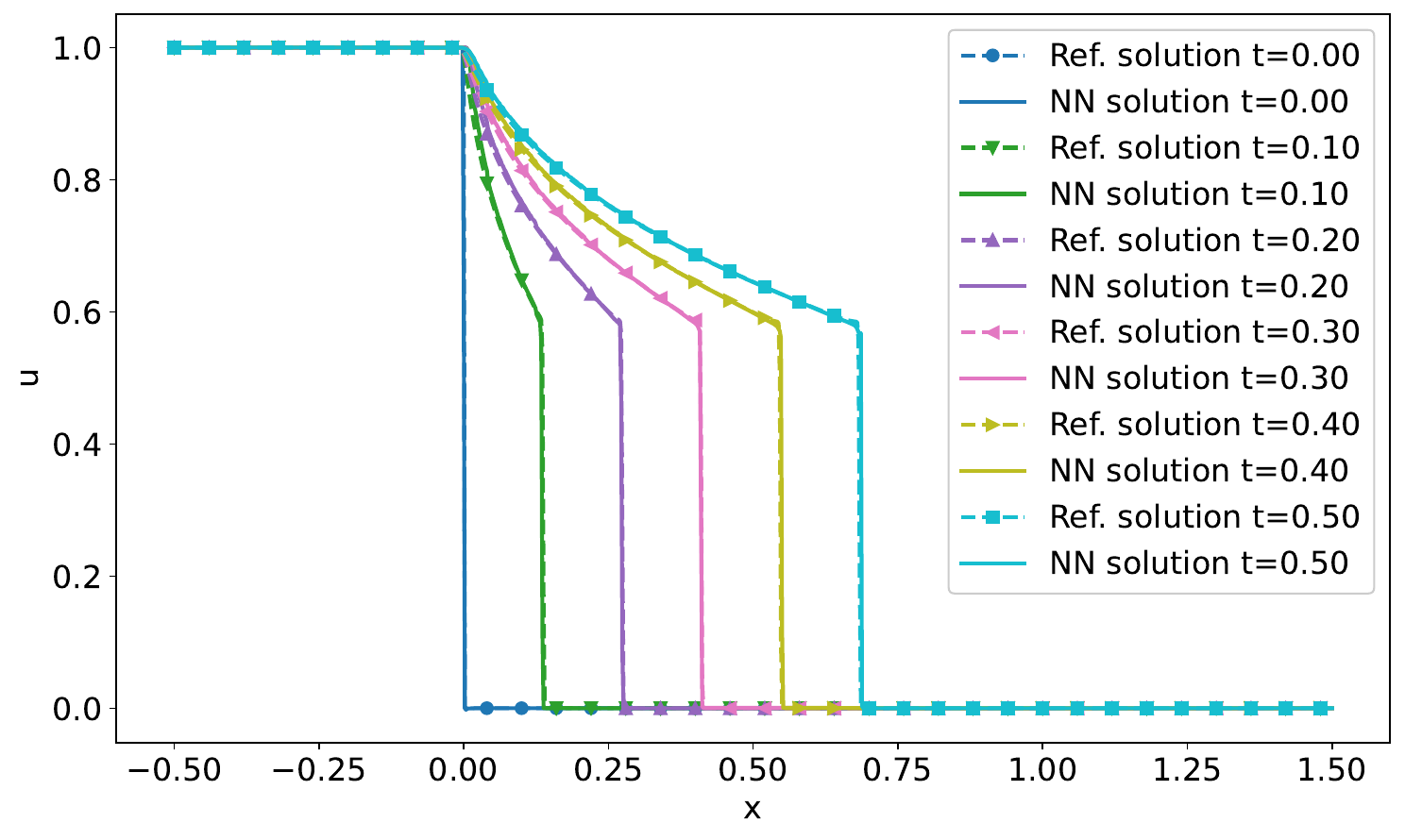}%
\includegraphics[width=0.55\linewidth]{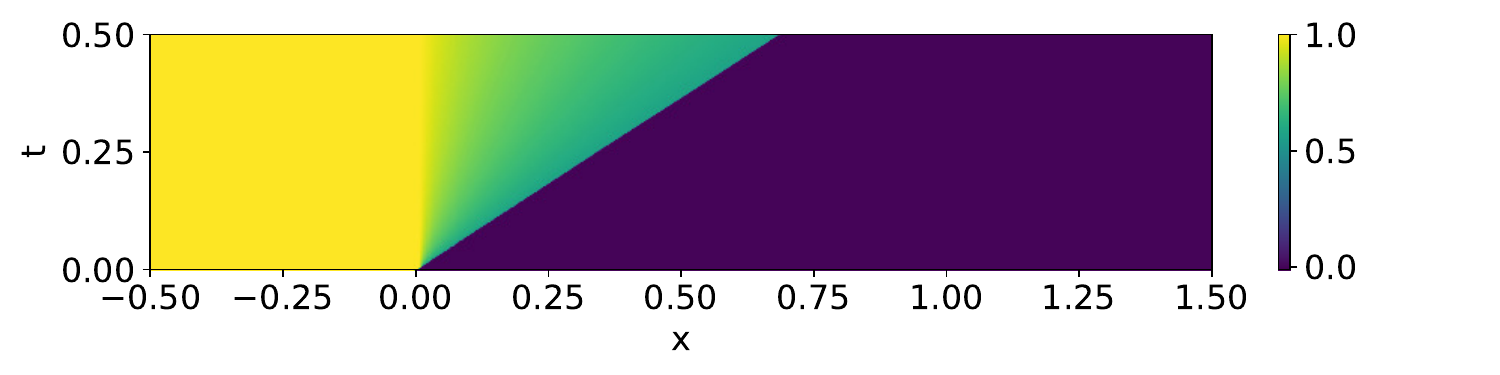}
\caption{Reference and neural network solutions for the Buckley--Leverett equation.}
\label{fig:BLE}
\end{figure}

\subsubsection{Sine flux}

As a final non-convex example, we consider the scalar HCL
with the highly non-convex sinusoidal flux $f(u) \;=\; \sin(\pi u)$ posed on the spatial domain $[-1.5,1.5]$, with Riemann initial data
\[
u_{0}(x) \;=\;
\begin{cases}
\tfrac{1}{2}, & x < 0, \\
\tfrac{5}{2}, & x \geq 0.
\end{cases}
\]
Since the flux possesses several inflection points between the two 
constant states, the entropy solution is a compound wave consisting of a 
central rarefaction fan flanked by two shocks of opposite orientation. 
This configuration cannot be captured by Rankine--Hugoniot conditions 
alone and crucially relies on the entropy admissibility condition, making 
it a particularly challenging test for the proposed method.

The reference solution is again computed by a third-order strong stability
preserving Runge--Kutta scheme combined with a fifth-order WENO 
reconstruction on a very fine mesh. The neural network approximation is 
displayed in Figure~\ref{fig:sin_flux}, and the corresponding error at 
the final time $T=0.5$ is reported in Table~\ref{tab:hyperparameter}. The 
two shocks and the connecting rarefaction fan are all accurately resolved.

\begin{figure}[h!]
\centering
\begin{subfigure}[b]{0.6\textwidth}
\includegraphics[width=\linewidth]{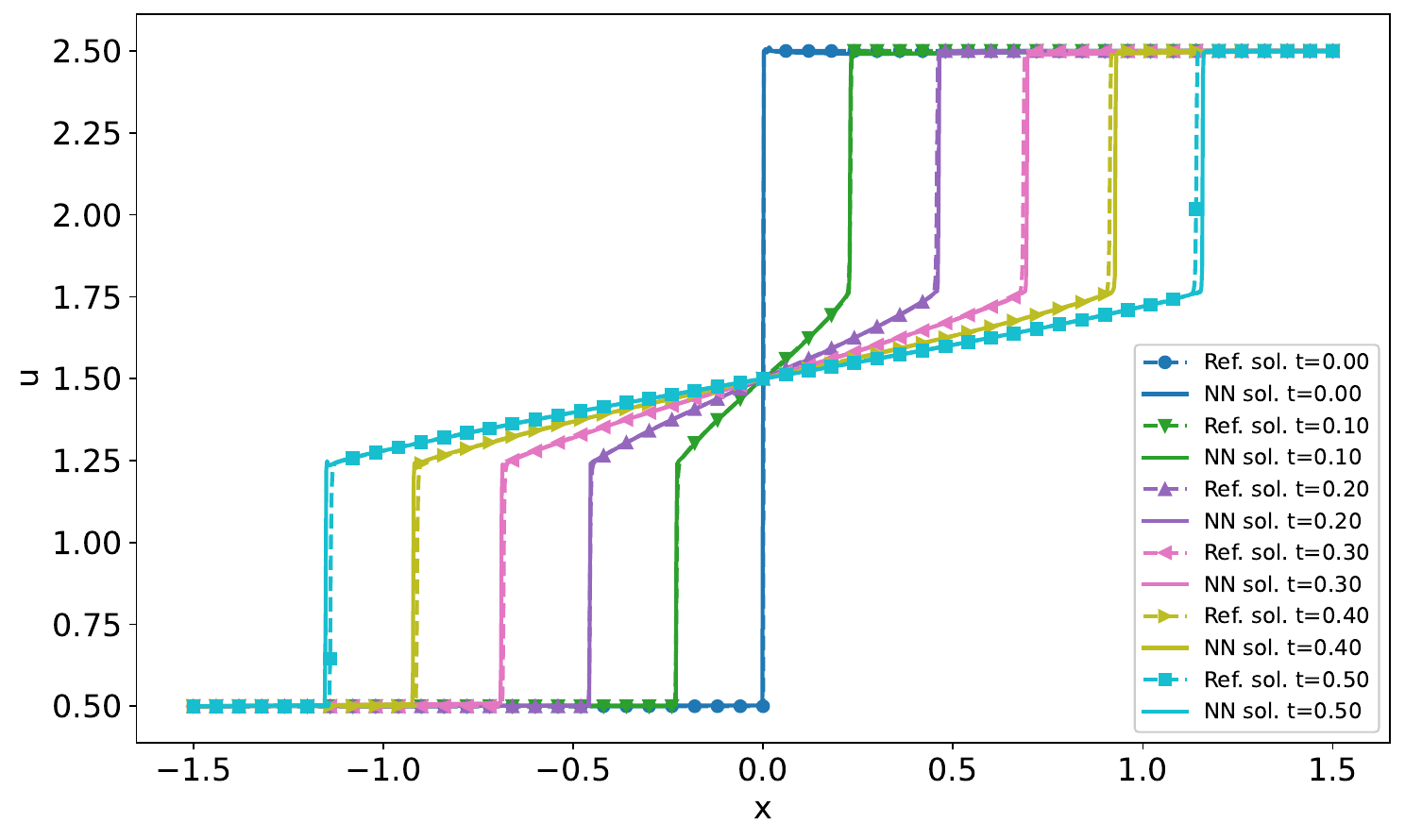}
\caption{Reference and NN solutions up to $T=0.5$.}
\end{subfigure} \\
\hspace{1.6cm}
\begin{subfigure}[b]{0.59\textwidth}
\includegraphics[width=\linewidth]{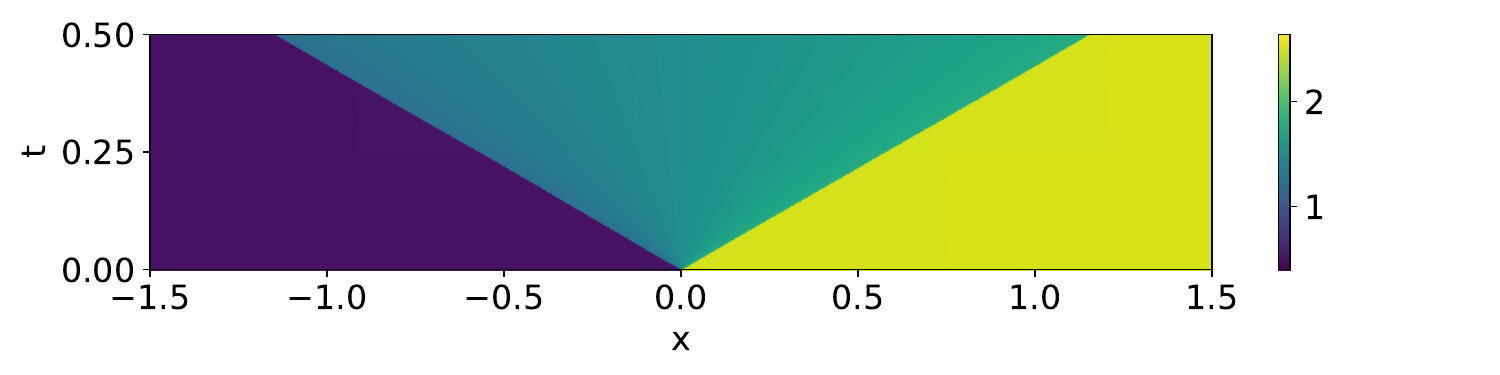}
\caption{Space--time NN solution.}
\end{subfigure}
\caption{Reference solutions and NN solutions for the conservation law with sine flux.}
\label{fig:sin_flux}
\end{figure}

\subsection{Two-dimensional Burgers equation}

To illustrate the performance of the proposed method in several space
dimensions, we consider the two-dimensional Burgers' equation
\[
\partial_t u + \partial_x\!\left(\frac{u^2}{2}\right) + \partial_y\!\left(\frac{u^2}{2}\right) = 0,
\qquad (x,y)\in[0,1]^2,\ t\in[0,T],
\]
with exact solution
\[
u(x,y,t) =
\begin{cases}
	\begin{cases}
		0, & y \geq \tfrac{1}{2} + t, \\
		2, & \text{otherwise},
	\end{cases}
	& x \leq \tfrac{1}{2} - t, \\[1.2em]
	\begin{cases}
		-2, & y \geq 1-x, \\
		\phantom{-}2, & \text{otherwise},
	\end{cases}
	& x > \tfrac{1}{2} - t.
\end{cases}
\]
\begin{figure}[htbp!]
	\centering
	\includegraphics[width=\textwidth]{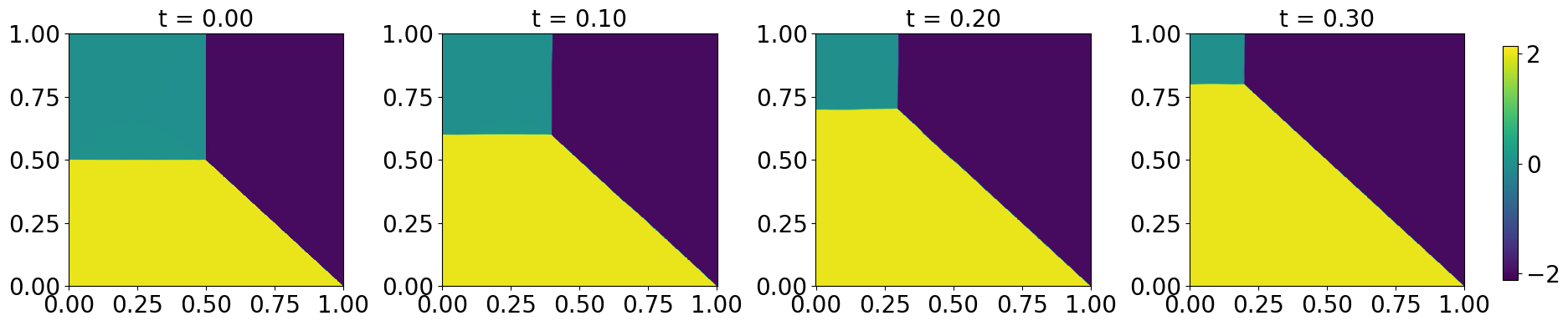}
	\caption{Snapshots of the neural network solution to the 2D Burgers
		equation up to $T = 0.3$.}
	\label{fig:2d_burgers}
\end{figure}
Snapshots of the neural network approximation are shown in
Figure~\ref{fig:2d_burgers}, and the corresponding errors at the final
time $T = 0.3$ are reported in Table~\ref{tab:hyperparameter}. The
network accurately captures all three shock fronts as well as their interaction, with sharp transitions and no visible spurious oscillations.

\section{Conclusion}

We introduced an entropy-compatible neural network method for scalar hyperbolic conservation laws based on a computable surrogate of the Kru\v{z}kov entropy residual, in which the constant entropy parameter is replaced by a bounded space--time comparison function from a finite-dimensional
discontinuous polynomial space. A central ingredient of the analysis is a constructive bridge from shock-adapted continuous piecewise linear approximations to the entropy solution producing small loss thanks to the reformulated entropy condition. Combined with an entropy-based stability argument, this yields explicit $L^1$ convergence rates for entropy
solutions containing a single regular shock in one space dimension and transversal interactions of finitely many smooth shock hypersurfaces in several space dimensions. In particular, for networks of size $O(h^{-(d+1)})$, we obtain the rate $O(h^{1/2})$.

The numerical experiments exhibit sharp shock resolution and sub-percent
relative \(L^1\) errors in the reported test cases. Important directions
for future work include a sharper treatment of the interface discrepancy,
the incorporation of optimization, sampling, and quadrature errors, and
extensions to systems of conservation laws. The additional one-dimensional
results in Sections~\ref{sec:rarefaction_wave}--\ref{sec:shock-formation} and Theorem~\ref{thm:main_1d} cover rarefaction and compound waves, regular shock interactions, and nondegenerate shock formation from smooth initial data. The rates are $O(h|\ln h|)$ in the rarefaction case and $O(h^{1/2}|\ln h|)$ in the shock-formation case, as stated there. Extending the shock-formation analysis to several space dimensions remains a direction for future work.

\section{Acknowledgements}
The work of the authors is partially supported by the National Natural Science Foundation of China (project no.~12525111) and the CAS AMSS-PolyU Joint Laboratory of Applied Mathematics.
\bibliography{ref}


\begin{thebibliography}{38}
\ifx \bisbn   \undefined \def \bisbn  #1{ISBN #1}\fi
\ifx \binits  \undefined \def \binits#1{#1}\fi
\ifx \bauthor  \undefined \def \bauthor#1{#1}\fi
\ifx \batitle  \undefined \def \batitle#1{#1}\fi
\ifx \bjtitle  \undefined \def \bjtitle#1{#1}\fi
\ifx \bvolume  \undefined \def \bvolume#1{\textbf{#1}}\fi
\ifx \byear  \undefined \def \byear#1{#1}\fi
\ifx \bissue  \undefined \def \bissue#1{#1}\fi
\ifx \bfpage  \undefined \def \bfpage#1{#1}\fi
\ifx \blpage  \undefined \def \blpage #1{#1}\fi
\ifx \burl  \undefined \def \burl#1{\textsf{#1}}\fi
\ifx \doiurl  \undefined \def \doiurl#1{\url{https://doi.org/#1}}\fi
\ifx \betal  \undefined \def \betal{\textit{et al.}}\fi
\ifx \binstitute  \undefined \def \binstitute#1{#1}\fi
\ifx \binstitutionaled  \undefined \def \binstitutionaled#1{#1}\fi
\ifx \bctitle  \undefined \def \bctitle#1{#1}\fi
\ifx \beditor  \undefined \def \beditor#1{#1}\fi
\ifx \bpublisher  \undefined \def \bpublisher#1{#1}\fi
\ifx \bbtitle  \undefined \def \bbtitle#1{#1}\fi
\ifx \bedition  \undefined \def \bedition#1{#1}\fi
\ifx \bseriesno  \undefined \def \bseriesno#1{#1}\fi
\ifx \blocation  \undefined \def \blocation#1{#1}\fi
\ifx \bsertitle  \undefined \def \bsertitle#1{#1}\fi
\ifx \bsnm \undefined \def \bsnm#1{#1}\fi
\ifx \bsuffix \undefined \def \bsuffix#1{#1}\fi
\ifx \bparticle \undefined \def \bparticle#1{#1}\fi
\ifx \barticle \undefined \def \barticle#1{#1}\fi
\bibcommenthead
\ifx \bconfdate \undefined \def \bconfdate #1{#1}\fi
\ifx \botherref \undefined \def \botherref #1{#1}\fi
\ifx \url \undefined \def \url#1{\textsf{#1}}\fi
\ifx \bchapter \undefined \def \bchapter#1{#1}\fi
\ifx \bbook \undefined \def \bbook#1{#1}\fi
\ifx \bcomment \undefined \def \bcomment#1{#1}\fi
\ifx \oauthor \undefined \def \oauthor#1{#1}\fi
\ifx \citeauthoryear \undefined \def \citeauthoryear#1{#1}\fi
\ifx \endbibitem  \undefined \def \endbibitem {}\fi
\ifx \bconflocation  \undefined \def \bconflocation#1{#1}\fi
\ifx \arxivurl  \undefined \def \arxivurl#1{\textsf{#1}}\fi
\csname PreBibitemsHook\endcsname

\bibitem[\protect\citeauthoryear{Harten et~al.}{1987}]{Harten1987}
\begin{barticle}
\bauthor{\binits{A.} \bsnm{Harten}},
\bauthor{\binits{B.} \bsnm{Engquist}},
\bauthor{\binits{S.} \bsnm{Osher}} and
\bauthor{\binits{S.R.} \bsnm{Chakravarthy}}:
\batitle{Uniformly high order accurate essentially non-oscillatory schemes, {III}}.
\bjtitle{Journal of Computational Physics}
\bvolume{71}(\bissue{2}),
\bfpage{231}--\blpage{303}
(\byear{1987})
\end{barticle}
\endbibitem

\bibitem[\protect\citeauthoryear{Liu et~al.}{1994}]{liu1994weighted}
\begin{barticle}
\bauthor{\binits{X.-D.} \bsnm{Liu}},
\bauthor{\binits{S.} \bsnm{Osher}} and
\bauthor{\binits{T.} \bsnm{Chan}}:
\batitle{Weighted essentially non-oscillatory schemes}.
\bjtitle{Journal of Computational Physics}
\bvolume{115}(\bissue{1}),
\bfpage{200}--\blpage{212}
(\byear{1994})
\end{barticle}
\endbibitem

\bibitem[\protect\citeauthoryear{Jiang and Shu}{1996}]{JiangShu1996}
\begin{barticle}
\bauthor{\binits{G.-S.} \bsnm{Jiang}} and
\bauthor{\binits{C.-W.} \bsnm{Shu}}:
\batitle{Efficient implementation of weighted {ENO} schemes}.
\bjtitle{Journal of Computational Physics}
\bvolume{126}(\bissue{1}),
\bfpage{202}--\blpage{228}
(\byear{1996})
\end{barticle}
\endbibitem

\bibitem[\protect\citeauthoryear{Cockburn and Shu}{2001}]{CockburnShu2001}
\begin{barticle}
\bauthor{\binits{B.} \bsnm{Cockburn}} and
\bauthor{\binits{C.-W.} \bsnm{Shu}}:
\batitle{{Runge--Kutta discontinuous Galerkin methods for convection-dominated problems}}.
\bjtitle{Journal of Scientific Computing}
\bvolume{16}(\bissue{3}),
\bfpage{173}--\blpage{261}
(\byear{2001})
\end{barticle}
\endbibitem

\bibitem[\protect\citeauthoryear{Tadmor}{2016}]{tadmor2016entropy}
\begin{bchapter}
\bauthor{\binits{E.} \bsnm{Tadmor}}:
\bctitle{Entropy stable schemes}.
In: \bbtitle{Handbook of Numerical Analysis}
vol. \bseriesno{17},
pp. \bfpage{467}--\blpage{493}.
\bpublisher{Elsevier},
\blocation{Amsterdam}
(\byear{2016})
\end{bchapter}
\endbibitem

\bibitem[\protect\citeauthoryear{Kru{\v{z}}kov}{1970}]{kruvzkov1970first}
\begin{barticle}
\bauthor{\binits{S.N.} \bsnm{Kru{\v{z}}kov}}:
\batitle{First order quasilinear equations in several independent variables}.
\bjtitle{Mathematics of the USSR-Sbornik}
\bvolume{10}(\bissue{2}),
\bfpage{217}--\blpage{243}
(\byear{1970})
\end{barticle}
\endbibitem

\bibitem[\protect\citeauthoryear{Kuznetsov}{1976}]{kuznetsov1976accuracy}
\begin{barticle}
\bauthor{\binits{N.N.} \bsnm{Kuznetsov}}:
\batitle{Accuracy of some approximate methods for computing the weak solutions of a first-order quasi-linear equation}.
\bjtitle{USSR Computational Mathematics and Mathematical Physics}
\bvolume{16}(\bissue{6}),
\bfpage{105}--\blpage{119}
(\byear{1976})
\end{barticle}
\endbibitem

\bibitem[\protect\citeauthoryear{Sanders}{1983}]{sanders1983convergence}
\begin{barticle}
\bauthor{\binits{R.} \bsnm{Sanders}}:
\batitle{On convergence of monotone finite difference schemes with variable spatial differencing}.
\bjtitle{Mathematics of Computation}
\bvolume{40}(\bissue{161}),
\bfpage{91}--\blpage{106}
(\byear{1983})
\end{barticle}
\endbibitem

\bibitem[\protect\citeauthoryear{Cockburn and Gremaud}{1996}]{cockburn1996priori}
\begin{barticle}
\bauthor{\binits{B.} \bsnm{Cockburn}} and
\bauthor{\binits{P.-A.} \bsnm{Gremaud}}:
\batitle{{A priori error estimates for numerical methods for scalar conservation laws. Part I: The general approach}}.
\bjtitle{Mathematics of Computation}
\bvolume{65}(\bissue{214}),
\bfpage{533}--\blpage{573}
(\byear{1996})
\end{barticle}
\endbibitem

\bibitem[\protect\citeauthoryear{Makridakis and Perthame}{2003}]{makridakis2003optimal}
\begin{barticle}
\bauthor{\binits{C.} \bsnm{Makridakis}} and
\bauthor{\binits{B.} \bsnm{Perthame}}:
\batitle{Optimal rate of convergence for anisotropic vanishing viscosity limit of a scalar balance law}.
\bjtitle{SIAM Journal on Mathematical Analysis}
\bvolume{34}(\bissue{6}),
\bfpage{1300}--\blpage{1307}
(\byear{2003})
\end{barticle}
\endbibitem

\bibitem[\protect\citeauthoryear{Nessyahu and Tadmor}{1992}]{nessyahu1992convergence}
\begin{barticle}
\bauthor{\binits{H.} \bsnm{Nessyahu}} and
\bauthor{\binits{E.} \bsnm{Tadmor}}:
\batitle{The convergence rate of approximate solutions for nonlinear scalar conservation laws}.
\bjtitle{SIAM Journal on Numerical Analysis}
\bvolume{29}(\bissue{6}),
\bfpage{1505}--\blpage{1519}
(\byear{1992})
\end{barticle}
\endbibitem

\bibitem[\protect\citeauthoryear{DiPerna}{1985}]{diperna1985measure}
\begin{barticle}
\bauthor{\binits{R.J.} \bsnm{DiPerna}}:
\batitle{Measure-valued solutions to conservation laws}.
\bjtitle{Archive for Rational Mechanics and Analysis}
\bvolume{88}(\bissue{3}),
\bfpage{223}--\blpage{270}
(\byear{1985})
\end{barticle}
\endbibitem

\bibitem[\protect\citeauthoryear{Szepessy}{1989}]{szepessy1989convergence}
\begin{barticle}
\bauthor{\binits{A.} \bsnm{Szepessy}}:
\batitle{Convergence of a shock-capturing streamline diffusion finite element method for a scalar conservation law in two space dimensions}.
\bjtitle{Mathematics of Computation}
\bvolume{53}(\bissue{188}),
\bfpage{527}--\blpage{545}
(\byear{1989})
\end{barticle}
\endbibitem

\bibitem[\protect\citeauthoryear{Raissi et~al.}{2019}]{Raissi2019}
\begin{barticle}
\bauthor{\binits{M.} \bsnm{Raissi}},
\bauthor{\binits{P.} \bsnm{Perdikaris}} and
\bauthor{\binits{G.E.} \bsnm{Karniadakis}}:
\batitle{Physics-informed neural networks: A deep learning framework for solving forward and inverse problems involving nonlinear partial differential equations}.
\bjtitle{Journal of Computational Physics}
\bvolume{378},
\bfpage{686}--\blpage{707}
(\byear{2019})
\end{barticle}
\endbibitem

\bibitem[\protect\citeauthoryear{Akrivis et~al.}{2025}]{akrivis2025runge}
\begin{barticle}
\bauthor{\binits{G.} \bsnm{Akrivis}},
\bauthor{\binits{C.G.} \bsnm{Makridakis}} and
\bauthor{\binits{C.} \bsnm{Smaragdakis}}:
\batitle{{Runge--Kutta physics informed neural networks: formulation and analysis}}.
\bjtitle{Numerische Mathematik}
\bvolume{157}(\bissue{6}),
\bfpage{1975}--\blpage{2016}
(\byear{2025})
\end{barticle}
\endbibitem

\bibitem[\protect\citeauthoryear{E and Yu}{2018}]{E2018}
\begin{barticle}
\bauthor{\binits{W.} \bsnm{E}} and
\bauthor{\binits{B.} \bsnm{Yu}}:
\batitle{{The Deep Ritz method: a deep learning-based numerical algorithm for solving variational problems}}.
\bjtitle{Communications in Mathematics and Statistics}
\bvolume{6}(\bissue{1}),
\bfpage{1}--\blpage{12}
(\byear{2018})
\end{barticle}
\endbibitem

\bibitem[\protect\citeauthoryear{Cai et~al.}{2025}]{cai2025efficient}
\begin{botherref}
\oauthor{\binits{Z.} \bsnm{Cai}},
\oauthor{\binits{A.} \bsnm{Doktorova}},
\oauthor{\binits{R.D.} \bsnm{Falgout}} and
\oauthor{\binits{C.} \bsnm{Herrera}}:
Efficient shallow {Ritz} method for one-dimensional diffusion-reaction problems.
SIAM Journal on Scientific Computing,
414--435
(2025)
\end{botherref}
\endbibitem

\bibitem[\protect\citeauthoryear{Li et~al.}{2021}]{Li2021FNO}
\begin{bchapter}
\bauthor{\binits{Z.} \bsnm{Li}},
\bauthor{\binits{N.} \bsnm{Kovachki}},
\bauthor{\binits{K.} \bsnm{Azizzadenesheli}},
\bauthor{\binits{B.} \bsnm{Liu}},
\bauthor{\binits{K.} \bsnm{Bhattacharya}},
\bauthor{\binits{A.M.} \bsnm{Stuart}} and
\bauthor{\binits{A.} \bsnm{Anandkumar}}:
\bctitle{Fourier neural operator for parametric partial differential equations}.
In: \bbtitle{International Conference on Learning Representations}
(\byear{2021}).
\burl{https://openreview.net/forum?id=c8P9NQVtmnO}
\end{bchapter}
\endbibitem

\bibitem[\protect\citeauthoryear{Lu et~al.}{2021}]{Lu2021}
\begin{barticle}
\bauthor{\binits{L.} \bsnm{Lu}},
\bauthor{\binits{P.} \bsnm{Jin}},
\bauthor{\binits{G.} \bsnm{Pang}},
\bauthor{\binits{Z.} \bsnm{Zhang}} and
\bauthor{\binits{G.E.} \bsnm{Karniadakis}}:
\batitle{Learning nonlinear operators via {DeepONet} based on the universal approximation theorem of operators}.
\bjtitle{Nature Machine Intelligence}
\bvolume{3}(\bissue{3}),
\bfpage{218}--\blpage{229}
(\byear{2021})
\end{barticle}
\endbibitem

\bibitem[\protect\citeauthoryear{Du and Zaki}{2021}]{du2021evolutional}
\begin{barticle}
\bauthor{\binits{Y.} \bsnm{Du}} and
\bauthor{\binits{T.A.} \bsnm{Zaki}}:
\batitle{Evolutional deep neural network}.
\bjtitle{Physical Review E}
\bvolume{104}(\bissue{4}),
\bfpage{045303}
(\byear{2021})
\end{barticle}
\endbibitem

\bibitem[\protect\citeauthoryear{Feischl et~al.}{2026}]{feischl2024regularized}
\begin{barticle}
\bauthor{\binits{M.} \bsnm{Feischl}},
\bauthor{\binits{C.} \bsnm{Lasser}},
\bauthor{\binits{C.} \bsnm{Lubich}} and
\bauthor{\binits{J.} \bsnm{Nick}}:
\batitle{Regularized dynamical parametric approximation}.
\bjtitle{Numerische Mathematik}
\bvolume{158}(\bissue{5}),
\bfpage{2079}--\blpage{2123}
(\byear{2026})
\doiurl{10.1007/s00211-026-01551-5}
\end{barticle}
\endbibitem

\bibitem[\protect\citeauthoryear{Su et~al.}{2025}]{su2025spike}
\begin{botherref}
\oauthor{\binits{H.} \bsnm{Su}},
\oauthor{\binits{L.} \bsnm{Zhang}} and
\oauthor{\binits{J.} \bsnm{Zhao}}:
{SPIKE}: Stable physics-informed kernel evolution method for solving hyperbolic conservation laws.
arXiv preprint arXiv:2510.18266
(2025)
\end{botherref}
\endbibitem

\bibitem[\protect\citeauthoryear{De~Ryck and Mishra}{2024}]{de2024numerical}
\begin{barticle}
\bauthor{\binits{T.} \bsnm{De~Ryck}} and
\bauthor{\binits{S.} \bsnm{Mishra}}:
\batitle{Numerical analysis of physics-informed neural networks and related models in physics-informed machine learning}.
\bjtitle{Acta Numerica}
\bvolume{33},
\bfpage{633}--\blpage{713}
(\byear{2024})
\end{barticle}
\endbibitem

\bibitem[\protect\citeauthoryear{Gazoulis et~al.}{2025}]{gazoulis2025stability}
\begin{barticle}
\bauthor{\binits{D.} \bsnm{Gazoulis}},
\bauthor{\binits{I.} \bsnm{Gkanis}} and
\bauthor{\binits{C.G.} \bsnm{Makridakis}}:
\batitle{On the stability and convergence of physics informed neural networks}.
\bjtitle{IMA Journal of Numerical Analysis}
(\byear{2025})
\doiurl{10.1093/imanum/draf090} .
\bcomment{Article draf090}
\end{barticle}
\endbibitem

\bibitem[\protect\citeauthoryear{Chaumet and Giesselmann}{2024}]{chaumet2024improving}
\begin{barticle}
\bauthor{\binits{A.} \bsnm{Chaumet}} and
\bauthor{\binits{J.} \bsnm{Giesselmann}}:
\batitle{Improving weak {PINNs} for hyperbolic conservation laws: Dual norm computation, boundary conditions and systems}.
\bjtitle{The SMAI Journal of Computational Mathematics}
\bvolume{10},
\bfpage{373}--\blpage{401}
(\byear{2024})
\end{barticle}
\endbibitem

\bibitem[\protect\citeauthoryear{Cai et~al.}{2023}]{cai2023least}
\begin{barticle}
\bauthor{\binits{Z.} \bsnm{Cai}},
\bauthor{\binits{J.} \bsnm{Chen}} and
\bauthor{\binits{M.} \bsnm{Liu}}:
\batitle{{Least-squares neural network (LSNN) method for scalar nonlinear hyperbolic conservation laws: Discrete divergence operator}}.
\bjtitle{Journal of Computational and Applied Mathematics}
\bvolume{433},
\bfpage{115298}
(\byear{2023})
\end{barticle}
\endbibitem

\bibitem[\protect\citeauthoryear{Patel et~al.}{2022}]{patel2022thermodynamically}
\begin{barticle}
\bauthor{\binits{R.G.} \bsnm{Patel}},
\bauthor{\binits{I.} \bsnm{Manickam}},
\bauthor{\binits{N.A.} \bsnm{Trask}},
\bauthor{\binits{M.A.} \bsnm{Wood}},
\bauthor{\binits{M.} \bsnm{Lee}},
\bauthor{\binits{I.} \bsnm{Tomas}} and
\bauthor{\binits{E.C.} \bsnm{Cyr}}:
\batitle{Thermodynamically consistent physics-informed neural networks for hyperbolic systems}.
\bjtitle{Journal of Computational Physics}
\bvolume{449},
\bfpage{110754}
(\byear{2022})
\end{barticle}
\endbibitem

\bibitem[\protect\citeauthoryear{De~Ryck and Mishra}{2024}]{de2024error}
\begin{barticle}
\bauthor{\binits{T.} \bsnm{De~Ryck}} and
\bauthor{\binits{S.} \bsnm{Mishra}}:
\batitle{Error analysis for deep neural network approximations of parametric hyperbolic conservation laws}.
\bjtitle{Mathematics of Computation}
\bvolume{93}(\bissue{350}),
\bfpage{2643}--\blpage{2677}
(\byear{2024})
\end{barticle}
\endbibitem

\bibitem[\protect\citeauthoryear{Oubarka et~al.}{2026}]{oubarka2026weak}
\begin{botherref}
\oauthor{\binits{I.} \bsnm{Oubarka}},
\oauthor{\binits{I.} \bsnm{Kissami}},
\oauthor{\binits{M.} \bsnm{Boubekeur}},
\oauthor{\binits{F.} \bsnm{Benkhaldoun}},
\oauthor{\binits{A.} \bsnm{Madrane}} and
\oauthor{\binits{Z.} \bsnm{Saadi}}:
Weak and entropy physics-informed neural networks for conservation laws.
arXiv preprint arXiv:2603.24819
(2026)
\end{botherref}
\endbibitem

\bibitem[\protect\citeauthoryear{Bouchut and Perthame}{1998}]{bouchut1998kruzkov}
\begin{barticle}
\bauthor{\binits{F.} \bsnm{Bouchut}} and
\bauthor{\binits{B.} \bsnm{Perthame}}:
\batitle{Kruzkov’s estimates for scalar conservation laws revisited}.
\bjtitle{Transactions of the American Mathematical Society}
\bvolume{350}(\bissue{7}),
\bfpage{2847}--\blpage{2870}
(\byear{1998})
\end{barticle}
\endbibitem

\bibitem[\protect\citeauthoryear{De~Ryck et~al.}{2024}]{de2024wpinns}
\begin{barticle}
\bauthor{\binits{T.} \bsnm{De~Ryck}},
\bauthor{\binits{S.} \bsnm{Mishra}} and
\bauthor{\binits{R.} \bsnm{Molinaro}}:
\batitle{{wPINNs: Weak physics informed neural networks for approximating entropy solutions of hyperbolic conservation laws}}.
\bjtitle{SIAM Journal on Numerical Analysis}
\bvolume{62}(\bissue{2}),
\bfpage{811}--\blpage{841}
(\byear{2024})
\end{barticle}
\endbibitem

\bibitem[\protect\citeauthoryear{Wang et~al.}{2021}]{wang2021learning}
\begin{barticle}
\bauthor{\binits{S.} \bsnm{Wang}},
\bauthor{\binits{H.} \bsnm{Wang}} and
\bauthor{\binits{P.} \bsnm{Perdikaris}}:
\batitle{Learning the solution operator of parametric partial differential equations with physics-informed {DeepONets}}.
\bjtitle{Science Advances}
\bvolume{7}(\bissue{40}),
\bfpage{8605}
(\byear{2021})
\end{barticle}
\endbibitem

\bibitem[\protect\citeauthoryear{He et~al.}{2020}]{he2020relu}
\begin{barticle}
\bauthor{\binits{J.} \bsnm{He}},
\bauthor{\binits{L.} \bsnm{Li}},
\bauthor{\binits{J.} \bsnm{Xu}} and
\bauthor{\binits{C.} \bsnm{Zheng}}:
\batitle{{ReLU} deep neural networks and linear finite elements}.
\bjtitle{Journal of Computational Mathematics}
\bvolume{38}(\bissue{3}),
\bfpage{502}--\blpage{527}
(\byear{2020})
\end{barticle}
\endbibitem

\bibitem[\protect\citeauthoryear{He et~al.}{2022}]{he2022hierarchical}
\begin{barticle}
\bauthor{\binits{J.} \bsnm{He}},
\bauthor{\binits{L.} \bsnm{Li}} and
\bauthor{\binits{J.} \bsnm{Xu}}:
\batitle{{ReLU} deep neural networks from the hierarchical basis perspective}.
\bjtitle{Computers \& Mathematics with Applications}
\bvolume{120},
\bfpage{105}--\blpage{114}
(\byear{2022})
\doiurl{10.1016/j.camwa.2022.06.006}
\end{barticle}
\endbibitem

\bibitem[\protect\citeauthoryear{Jin}{2026}]{jin2026two}
\begin{barticle}
\bauthor{\binits{P.} \bsnm{Jin}}:
\batitle{Two-hidden-layer {ReLU} neural networks and finite elements}.
\bjtitle{Neural Networks}
\bvolume{198},
\bfpage{108559}
(\byear{2026})
\doiurl{10.1016/j.neunet.2026.108559}
\end{barticle}
\endbibitem

\bibitem[\protect\citeauthoryear{Mishra and Molinaro}{2023}]{mishra2023estimates}
\begin{barticle}
\bauthor{\binits{S.} \bsnm{Mishra}} and
\bauthor{\binits{R.} \bsnm{Molinaro}}:
\batitle{Estimates on the generalization error of physics-informed neural networks for approximating {PDEs}}.
\bjtitle{IMA Journal of Numerical Analysis}
\bvolume{43}(\bissue{1}),
\bfpage{1}--\blpage{43}
(\byear{2023})
\end{barticle}
\endbibitem

\bibitem[\protect\citeauthoryear{Yin and Zhu}{2022}]{huicheng2022shock}
\begin{barticle}
\bauthor{\binits{H.} \bsnm{Yin}} and
\bauthor{\binits{L.} \bsnm{Zhu}}:
\batitle{The shock formation and optimal regularities of the resulting shock curves for 1d scalar conservation laws}.
\bjtitle{Nonlinearity}
\bvolume{35}(\bissue{2}),
\bfpage{954}--\blpage{997}
(\byear{2022})
\end{barticle}
\endbibitem

\bibitem[\protect\citeauthoryear{Longo et~al.}{2023}]{longo2023rham}
\begin{barticle}
\bauthor{\binits{M.} \bsnm{Longo}},
\bauthor{\binits{J.A.A.} \bsnm{Opschoor}},
\bauthor{\binits{N.} \bsnm{Disch}},
\bauthor{\binits{C.} \bsnm{Schwab}} and
\bauthor{\binits{J.} \bsnm{Zech}}:
\batitle{{De Rham compatible deep neural network FEM}}.
\bjtitle{Neural Networks}
\bvolume{165},
\bfpage{721}--\blpage{739}
(\byear{2023})
\end{barticle}
\endbibitem

\end{thebibliography}

\appendix
\section{Appendix}\label{appendix:proof_lem_tanh_nn}
\subsection{Neural network approximation}\label{app:nn-approx}
In this appendix, we adapt the min--max construction underlying the ReLU representability of CPwL functions on simplicial meshes; see, e.g., \cite{he2020relu,longo2023rham}. The key point is that nodal hat basis functions admit representations by finitely many binary $\min$ and $\max$ operations applied to affine functions. We replace these nonsmooth operations by smooth $\tanh$-based surrogates. For $\tau>0$, define
\begin{equation}\label{eq:smooth-minmax}
	S^{-}_{\tau}(r,s)
	:=
	\frac{r+s}{2}
	-\frac{r-s}{2}\tanh\bigl(\tau(r-s)\bigr),
	\qquad
	S^{+}_{\tau}(r,s)
	:=
	\frac{r+s}{2}
	+\frac{r-s}{2}\tanh\bigl(\tau(r-s)\bigr).
\end{equation}
These are smooth approximations of $\min(r,s)$ and $\max(r,s)$.

\begin{figure}[htbp!]
	\centering
	\begin{subfigure}[b]{0.4\textwidth}
		\includegraphics[width=\linewidth]{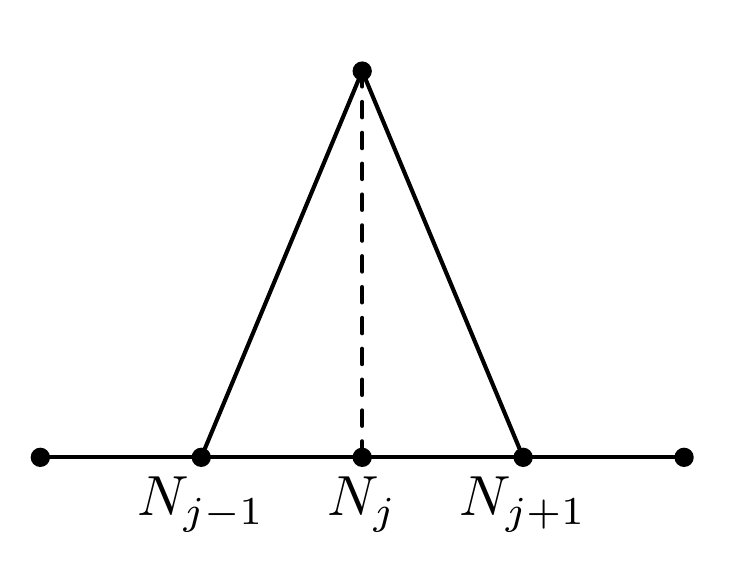}
	\end{subfigure}
	\qquad
	\begin{subfigure}[b]{0.4\textwidth}
		\includegraphics[width=\linewidth]{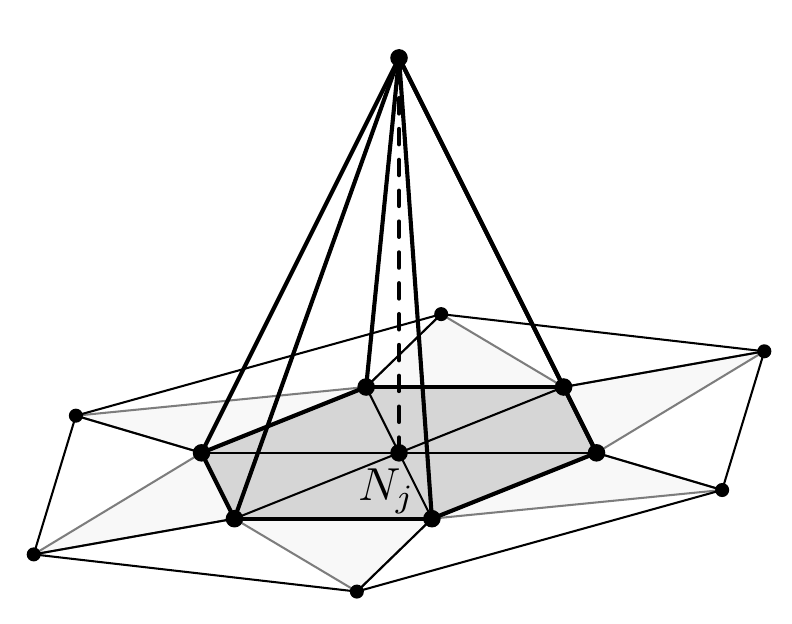}
	\end{subfigure}
	\caption{Finite-element hat basis functions in one and two dimensions.}
	\label{fig:base_function}
\end{figure}

We will repeatedly use the function
\begin{equation}\label{eq:def_psi}
	\Psi_\tau(r):=\tanh(\tau r)+\tau r\,\text{sech}^2(\tau r).
\end{equation}
Throughout the appendix, we use the convention \(\sgn(0)=0\).
\begin{lemma}[Smoothing of one min--max operation]
	\label{lem:binary_smoothing}
	Let $\mathcal O\subset\mathbb R^n$ be a bounded Lipschitz domain,
	and let $G,H$ be continuous piecewise affine functions on a finite
	common polyhedral partition of $\mathcal O$. Set
	\[
	m^-(G,H):=\min\{G,H\},
	\qquad
	m^+(G,H):=\max\{G,H\}.
	\]
	Then, for either sign,
	\begin{align}
	\|S_\tau^\pm(G,H)-m^\pm(G,H)\|_{L^\infty(\mathcal O)}
	+
	\|\nabla S_\tau^\pm(G,H)
	-\nabla m^\pm(G,H)\|_{L^1(\mathcal O)}
	\longrightarrow0,
	\end{align}
	as $\tau\to\infty$.
\end{lemma}

\begin{proof}
We first note the following elementary estimates:
\begin{equation}\label{eq:element_estimates}
    \begin{aligned}
        \sup_{r\in\mathbb R}
|r|\bigl(1-\tanh(\tau|r|)\bigr)
\le C\tau^{-1},\qquad\qquad\\
\sup_{r\in\mathbb R} |\Psi_\tau(r)|
\le C,
\qquad
|\Psi_\tau(r)-\sgn(r)|\le C e^{-c\tau|r|}.
    \end{aligned}
\end{equation}
These estimates follow directly from the exponential decay of
\(1-\tanh y\) and \(\sech^2 y\) for \(y\ge 0\). Set $D:=G-H$. Since
	\[
	m^\pm(G,H)
	=
	\frac{G+H}{2}\pm\frac{|D|}{2},
	\]
	we have pointwise
	\begin{align}\label{eq:S_app_m}
	|S_\tau^\pm(G,H)-m^\pm(G,H)|
	=
	\frac12 |D|
	\bigl(1-\tanh(\tau|D|)\bigr)
	\le C\tau^{-1}.
	\end{align}
	
	Consider now one polyhedron $P$ of the common partition and write $D(x)=p\cdot x+\alpha$ on $P$. Then, almost everywhere on $P$,
	\[
	\nabla S_\tau^\pm(G,H)
	=
	\frac{\nabla G+\nabla H}{2}
	\pm\frac12\Psi_\tau(D)p,
	\]
	whereas
	\[
	\nabla m^\pm(G,H)
	=
	\frac{\nabla G+\nabla H}{2}
	\pm\frac12\sgn(D)p.
	\]
	If \(p=0\), then the two gradients agree identically. If \(p\neq 0\), we slice \(P\) in the direction \(p/|p|\) and invoke \eqref{eq:element_estimates} to obtain
	\begin{align*}
	    \int_P
	|p|e^{-c\tau|p\cdot x+\alpha|}\,dx
	\le \frac{C_P}{\tau}.
	\end{align*}
	Consequently,
	\begin{align}\label{eq:S_app_m_W11}
	    \|\nabla S_\tau^\pm(G,H)
	-\nabla m^\pm(G,H)\|_{L^1(P)}
	\le \frac{C_P}{\tau}.
	\end{align}
	Summing over the finite partition proves the result.
\end{proof}

\begin{lemma}[Smoothing of a finite min--max tree]
	\label{lem:tree_smoothing}
	Let $F$ be a finite expression obtained from affine functions by
	finitely many binary min and max operations. Let $F_\tau$ be
	obtained by replacing every min and max by $S_\tau^-$ and
	$S_\tau^+$, respectively. Then
	\[
	\|F_\tau-F\|_{L^\infty(\mathcal O)}
	+
	\|\nabla F_\tau-\nabla F\|_{L^1(\mathcal O)}
	\longrightarrow0
	\qquad(\tau\to\infty).
	\]
\end{lemma}

\begin{proof}
	We argue by induction on the number of binary operations. The
	claim is immediate for affine functions. Suppose, for example,
	that
	\[
	F=\min\{G,H\},
	\qquad
	F_\tau=S_\tau^-(G_\tau,H_\tau);
	\]
	the max case is identical.
	
	Since $\Psi_\tau$ is uniformly bounded, $S_\tau^\pm$ is uniformly
	Lipschitz in its two scalar arguments. Hence the induction
	hypothesis and Lemma~\ref{lem:binary_smoothing} imply
	\[
	\|F_\tau-F\|_{L^\infty(\mathcal O)}\longrightarrow0.
	\]
	
	For the gradients, set
	\[
	d:=G-H,
	\qquad
	e_\tau:=(G_\tau-G)-(H_\tau-H),
	\qquad
	\delta_\tau:=\|e_\tau\|_{L^\infty(\mathcal O)}.
	\]
	The induction hypothesis gives $\delta_\tau\to0$. Using the
	gradient formula from the proof of
	Lemma~\ref{lem:binary_smoothing}, all terms are controlled by the
	induction hypothesis and the uniform boundedness of $\Psi_\tau$,
	except
	\[
	D_\tau
	:=
	\bigl(
	\Psi_\tau(d+e_\tau)-\Psi_\tau(d)
	\bigr)\nabla d.
	\]
	
	Take a finite common polyhedral refinement on which $G$ and $H$
	are affine. On one of its cells $P$, write
	$d(x)=p\cdot x+\alpha$. If $p=0$, then $D_\tau=0$ on $P$.
	Otherwise, decompose
	\[
	N_{\tau,P}
	:=
	P\cap\{|d|\le2\delta_\tau\},
	\qquad
	E_{\tau,P}:=P\setminus N_{\tau,P}.
	\]
	By slicing in the direction $p/|p|$, there holds $|N_{\tau,P}|
	\le C_P\frac{\delta_\tau}{|p|}$,
	and, by the induction hypothesis, $\|e_{\tau}\|_{L^{\infty}}\to 0$ as $\tau\to \infty$. Hence,
	\begin{align}\label{eq:D_1}
	    \int_{N_{\tau,P}}|D_\tau|\,dx
	\le C_P\delta_\tau\longrightarrow 0,\quad \text{as}\quad \tau\to\infty.
	\end{align}
	On $E_{\tau,P}$, the functions $d$ and $d+e_\tau$ have the same
	sign and $|d+e_\tau|\ge|d|/2$. Hence 
	\[
	|\Psi_\tau(d+e_\tau)-\Psi_\tau(d)|\leq |\Psi_\tau(d+e_\tau)-\sgn(d+e_\tau)|+|\Psi_\tau(d)-\sgn(d)|
	\le C e^{-c\tau|d|},
	\]
	after changing $c>0$ if necessary, and another slicing argument
	gives $\int_{E_{\tau,P}}|D_\tau|\,dx
	\le \frac{C_P}{\tau}.$ 
	Thus
	\begin{align}\label{eq:D_2}
	    \|D_\tau\|_{L^1(\mathcal O)}
	\le C\bigl(\delta_\tau+\tau^{-1}\bigr)
	\longrightarrow0.
	\end{align}
	By combining \eqref{eq:D_1} and \eqref{eq:D_2}, we complete the induction.
\end{proof}

\begin{lemma}[Fixed-size tanh approximation of arithmetic operations]
	\label{lem:tanh_arithmetic}
	Let $\sigma=\tanh$, and fix $b\in\mathbb R$ such that
	$\sigma''(b)\neq0$. For $A,B\ge1$, define
	\[
	I_A(x):=A\sigma(x/A),
	\]
	\[
	Q_B(x)
	:=
	\frac{B^2}{\sigma''(b)}
	\left[
	\sigma\left(b+\frac{x}{B}\right)
	+\sigma\left(b-\frac{x}{B}\right)
	-2\sigma(b)
	\right],
	\]
	and
	\[
	M_B(x,y)
	:=
	\frac14\bigl(Q_B(x+y)-Q_B(x-y)\bigr).
	\]
	For every $R>0$,
	\[
	\|I_A-\operatorname{id}\|_{W^{1,\infty}([-R,R])}
	\le C_R A^{-2},
	\]
	\[
	\|Q_B-(\operatorname{id})^2\|_{W^{1,\infty}([-R,R])}
	\le C_{R,b}B^{-2},
	\]
	and
	\[
	\|M_B-m\|_{W^{1,\infty}([-R,R]^2)}
	\le C_{R,b}B^{-2},
	\qquad
	m(x,y):=xy.
	\]
	Moreover, $I_A$, $Q_B$, and $M_B$ are one-hidden-layer tanh
	networks of widths $1$, $2$, and $4$, respectively.
	
	For fixed $\tau>0$, define
	\[
	\mathfrak S_{\tau,A,B}^{\pm}(r,s)
	:=
	\frac12 I_A\bigl(I_A(r+s)\bigr)
	\pm
	\frac12
	M_B\left(
	I_A(r-s),
	\sigma\bigl(\tau(r-s)\bigr)
	\right).
	\]
	Then $\mathfrak S_{\tau,A,B}^{\pm}$ is a two-hidden-layer tanh
	network of uniformly bounded width, and
	\[
	\|\mathfrak S_{\tau,A,B}^{\pm}-S_\tau^\pm
	\|_{W^{1,\infty}([-R,R]^2)}
	\le
	C_{R,b,\tau}\bigl(A^{-2}+B^{-2}\bigr).
	\]
\end{lemma}

\begin{proof}
	Taylor's formula gives, uniformly on bounded intervals,
	\[
	I_A(x)=x+O_R(A^{-2}),
	\qquad
	I_A'(x)=1+O_R(A^{-2}),
	\]
	and
	\[
	Q_B(x)=x^2+O_{R,b}(B^{-2}),
	\qquad
	Q_B'(x)=2x+O_{R,b}(B^{-2}).
	\]
	The polarization identity
	\[
	xy=\frac14\bigl((x+y)^2-(x-y)^2\bigr)
	\]
	therefore yields the stated $W^{1,\infty}$ estimate for $M_B$.
	
	To estimate $\mathfrak S_{\tau,A,B}^{\pm}$, set
	\[
	p:=r+s,
	\qquad
	d:=r-s,
	\qquad
	y:=\sigma(\tau d).
	\]
	On $[-R,R]^2$, these functions and their first derivatives are
	bounded by a constant depending only on $R$ and $\tau$. The
	preceding estimates and the chain rule give
	\[
	\|I_A\circ I_A-\operatorname{id}
	\|_{W^{1,\infty}([-2R,2R])}
	\le C_R A^{-2},
	\]
	and
	\[
	\|M_B(I_A(d),y)-dy
	\|_{W^{1,\infty}([-R,R]^2)}
	\le
	C_{R,b,\tau}(A^{-2}+B^{-2}),
	\]
	which proves the claim. The network-size statements follow
	directly from the displayed formulas.
\end{proof}

\begin{proposition}[Uniform \(\tanh\)-approximation of CPwL functions]\label{prop:represent}
	Let \(\Omega_T\subset\mathbb R^n\) be bounded, and let \(\hat u:\Omega_T\to\mathbb R\) be continuous and piecewise affine over a conforming simplicial partition \(\mathcal T\) with \(N:=\#\mathcal T<\infty\). Assume that each vertex of \(\mathcal T\) belongs to at most \(N_{\rm patch}\) simplices, where \(N_{\rm patch}\) depends only on \(d\). Then \(\forall \varepsilon\in(0,1)\) there exists a \(\tanh\) neural network \(u_{\hat\theta}\) of the form \eqref{def:NN} such that
	\begin{align}
		\|\hat u-u_{\hat\theta}\|_{L^\infty(\Omega_T)}&\le \varepsilon,\label{eq:tanh_approx_L_infty}\\
		\|\nabla\hat u-\nabla u_{\hat\theta}\|_{L^1(\Omega_T)}&\le \varepsilon.\label{eq:tanh_approx_W11}
	\end{align}
	Moreover,
	\begin{itemize}
		\item \emph{Depth:} \(L \le C_1\lceil \log_2(N_{\rm patch})\rceil + C_2\);
		\item \emph{Width:} each hidden layer has at most \(O(N_{\rm patch}N)\) neurons,
	\end{itemize}
	where \(C_1,C_2\) are independent of \(\varepsilon\) and \(N\).
\end{proposition}

\begin{proof}
	For functions on $\Omega_T$, set
	\[
	\mathcal E(v,w)
	:=
	\|v-w\|_{L^\infty(\Omega_T)}
	+
	\|\nabla v-\nabla w\|_{L^1(\Omega_T)}.
	\]
	Let $\mathcal V=\{z_j\}_{j=1}^{N_{\mathrm v}}$ be the vertices
	of $\mathcal T$, with nodal hat basis
	$\{\varphi_j\}_{j=1}^{N_{\mathrm v}}$. Since every simplex has
	$n+1$ vertices, $N_{\mathrm v}\le (n+1)N.$ 
    
	By \cite[Theorem~4.3]{longo2023rham}
	(see also \cite[Theorem~3.1]{he2020relu} for locally convex
	meshes), each $\varphi_j$ admits a binary min--max representation
	with affine leaves, at most $C N_{\mathrm{patch}}$ binary nodes,
	and depth
	\begin{equation}\label{eq:A7_depth_condition}
	D_j
	\le
	C_1\lceil\log_2N_{\mathrm{patch}}\rceil+C_2.
	\end{equation}

	Let $\varphi_{j,\tau}$ be obtained by replacing every minimum
	and maximum in this tree by $S_\tau^-$ and $S_\tau^+$,
	respectively. Lemma~\ref{lem:tree_smoothing} gives
	\begin{equation}\label{eq:A8_phi_tau_approx}
	\mathcal E(\varphi_{j,\tau},\varphi_j)
	\longrightarrow0
	\qquad(\tau\to\infty).
	\end{equation}

	Fix $\tau$. Pad shorter branches with identity nodes, replace
	each binary node $S_\tau^\pm$ by
	$\mathfrak S_{\tau,A,B}^\pm$, and replace each inserted identity
	by $I_A\circ I_A$. Denote the resulting tanh network by
	$\Phi_{j,\tau,A,B}$. We claim that
	\begin{equation}\label{eq:A9_NN_approx}
	\|\Phi_{j,\tau,A,B}-\varphi_{j,\tau}
	\|_{W^{1,\infty}(\Omega_T)}
	\longrightarrow0
	\qquad
	\text{as }\min\{A,B\}\to\infty.
	\end{equation}
	Indeed, this follows by induction over the levels of the finite
	tree. It is immediate at the affine leaves and at the inserted
	identity nodes by Lemma~\ref{lem:tanh_arithmetic}. Suppose that
	$G_{A,B}\to G$ and $H_{A,B}\to H$ in
	$W^{1,\infty}(\Omega_T)$ at two child nodes. Since
	$S_\tau^\pm(r,s)$ lies between $r$ and $s$, all exact
	intermediate values lie in a fixed bounded interval; the
	approximate ones do so for all sufficiently large $A,B$.
	Therefore,
	\[
	\begin{aligned}
	&
	\|\mathfrak S_{\tau,A,B}^\pm(G_{A,B},H_{A,B})
	-S_\tau^\pm(G,H)\|_{W^{1,\infty}(\Omega_T)}
	\\
	&\quad\le
	\|\mathfrak S_{\tau,A,B}^\pm(G_{A,B},H_{A,B})
	-S_\tau^\pm(G_{A,B},H_{A,B})\|_{W^{1,\infty}(\Omega_T)}
	\\
	&\qquad+
	\|S_\tau^\pm(G_{A,B},H_{A,B})
	-S_\tau^\pm(G,H)\|_{W^{1,\infty}(\Omega_T)}
	\longrightarrow0.
	\end{aligned}
	\]
	The first term tends to zero by
	Lemma~\ref{lem:tanh_arithmetic} and the chain rule, while the
	second does so because $S_\tau^\pm$ is $C^2$ on bounded sets.
	This proves \eqref{eq:A9_NN_approx}.

	Write
	\[
	\widehat u
	=
	\sum_{j=1}^{N_{\mathrm v}}
	\widehat u(z_j)\varphi_j,
	\qquad
	C_{\widehat u}
	:=
	1+\sum_{j=1}^{N_{\mathrm v}}|\widehat u(z_j)|.
	\]
	Given $\varepsilon>0$, the finiteness of $N_{\mathrm v}$,
	\eqref{eq:A8_phi_tau_approx}, and
	\eqref{eq:A9_NN_approx} allow us to choose first $\tau$ and then
	$A,B$ such that
	\[
	\max_{1\le j\le N_{\mathrm v}}
	\mathcal E(\varphi_{j,\tau},\varphi_j)
	<
	\frac{\varepsilon}{2C_{\widehat u}},
	\qquad
	\max_{1\le j\le N_{\mathrm v}}
	\mathcal E(\Phi_{j,\tau,A,B},\varphi_{j,\tau})
	<
	\frac{\varepsilon}{2C_{\widehat u}}.
	\]
	Define
	\[
	U_{\widehat\theta}
	:=
	\sum_{j=1}^{N_{\mathrm v}}
	\widehat u(z_j)\Phi_{j,\tau,A,B},
	\]
	where the summation is absorbed into the affine output layer.
	Then
	\[
	\mathcal E(U_{\widehat\theta},\widehat u)
	\le
	\sum_{j=1}^{N_{\mathrm v}}
	|\widehat u(z_j)|
	\left[
	\mathcal E(\Phi_{j,\tau,A,B},\varphi_{j,\tau})
	+
	\mathcal E(\varphi_{j,\tau},\varphi_j)
	\right]
	\le\varepsilon,
	\]
	which proves the asserted approximation estimates.

	Finally, every binary or identity node is replaced by a
	fixed-size tanh module. After merging consecutive affine maps,
	padding shorter branches, and evaluating all
	$N_{\mathrm v}=O(N)$ trees in parallel, the resulting network
	has depth \(O(\log N_{\mathrm{patch}})\) 
	and width \(O(NN_{\mathrm{patch}}).\) The approximation parameters $\tau,A,B$ affect only the weights
	and biases, not the architecture.
\end{proof}

\subsubsection*{Proof of Lemma~\ref{lem:loss_cpwl_approx}}\label{app:loss-cpwl}
Under the assumptions of Lemma~\ref{lem:loss_cpwl_approx}, the function \(\hat u\) satisfies
\[
\|\hat u\|_{L^\infty(\Omega_T)} \le \frac{c}{2}-\delta_c
\]
for some \(\delta_c>0\). Therefore, for any approximation tolerance
\(\varepsilon\le \delta_c/2\), Proposition~\ref{prop:represent} yields a raw tanh network
\(U_{\hat\theta}^{\rm raw}\) of depth \(O(\log_2 N_{\rm patch})\) and width
\(O(NN_{\rm patch})\) such that
\[
\max\!\left\{\|U_{\hat\theta}^{\rm raw}-\hat u\|_{L^\infty(\Omega_T)},\|\nabla U_{\hat\theta}^{\rm raw}-\nabla\hat u\|_{L^1(\Omega_T)}\right\}\le \varepsilon.
\]
Consequently,
\[
|U_{\hat\theta}^{\rm raw}(z)|
\le \|\hat u\|_{L^\infty(\Omega_T)}+\varepsilon
\le \frac{c}{2}-\frac{\delta_c}{2}
< \frac{c}{2},
\qquad \forall z\in\Omega_T.
\]
Hence, the clipping layer in \eqref{def:clipped-NN} is inactive, and thus
\[
u_{\hat\theta}:=\Pi_c(U_{\hat\theta}^{\rm raw})=U_{\hat\theta}^{\rm raw}.
\]
Therefore, in what follows we do not distinguish between \(u_{\hat\theta}\) and
\(U_{\hat\theta}^{\rm raw}\).

Since $\mathbf f\in C^2(\mathbb R;\mathbb R^d)$, the associated space--time flux $\mathbf F$
is $C^2$ as well. Moreover, in the argument below,
the functions under consideration, in particular $u_{\hat\theta}$, $\hat u$, and
$k_h\in V_h^c$, all take values in the bounded interval $[-c,c]$. Therefore, there
exists a constant $C_F>0$ such that for all a.e. differentiable functions $v,w$ on
$\Omega_T$ satisfying
$\|v\|_{L^\infty(\Omega_T)},\|w\|_{L^\infty(\Omega_T)}\le c$,
\begin{align}
    |\mathbf F(v)-\mathbf F(w)|&\le C_F |v-w|
	\quad\text{a.e. in }\Omega_T,\label{eq:proof_flux_lip}\\[2mm]
    |\nabla\!\cdot \mathbf F(v)-\nabla\!\cdot \mathbf F(w)|
	&\le C_F\bigl(|\nabla v-\nabla w|+|v-w|\,|\nabla w|\bigr)
	\quad\text{a.e. in }\Omega_T.\label{eq:proof_divflux_lip}
\end{align}
We now fix
\begin{align}\label{eq:choose_epsilon}
	\varepsilon:=\min\{h^2,C_F^{-1}h^2, {\delta_c}/{2},1/2\}.
\end{align}

We first estimate the boundary contribution. Recalling the definition of $\mathcal L_{ibc}$ in~\eqref{eq:loss_func}, the pointwise bound in Proposition~\ref{prop:represent} yields
\begin{equation}\label{eq:proof_bc}
	\begin{aligned}
		\mathcal L_{ibc}(u_{\hat\theta})
		&\le \|\hat u(0)-u_0\|_{L^1(\Omega)}
		+ |\Omega|\,\|u_{\hat\theta}-\hat u\|_{L^\infty(\Omega_T)}  + \|\hat u\|_{L^1(\partial\Omega\times(0,T))}\\
		&\qquad+ |\partial\Omega\times(0,T)|\,
		\|u_{\hat\theta}-\hat u\|_{L^\infty(\Omega_T)}
		\\
		&\le \mathcal L_{ibc}(\hat u)+C\varepsilon
		\le \mathcal L_{ibc}(\hat u)+Ch^2.
	\end{aligned}
\end{equation}

Next we consider the term $\mathcal L_{\mathrm{reg}}$. Since $\hat u$ is continuous and piecewise affine on the finite partition $\mathcal T$, we have $\|\nabla \hat u\|_{L^1(\Omega_T)}<\infty$. Using \eqref{eq:proof_divflux_lip} with
$v=u_{\hat\theta}$ and $w=\hat u$, we obtain
\begin{align}
	\mathcal L_{\mathrm{reg}}(u_{\hat\theta})
	&\le h\int_{\Omega_T} |\nabla\!\cdot\mathbf F(\hat u)|\,\d\mathbf z
	+ h\int_{\Omega_T} |\nabla\!\cdot\mathbf F(u_{\hat\theta})-\nabla\!\cdot\mathbf F(\hat u)|\,\d\mathbf z
	\notag\\
	&\le \mathcal L_{\mathrm{reg}}(\hat u)
	+ hC_F\|\nabla u_{\hat\theta}-\nabla\hat u\|_{L^1(\Omega_T)}
	+ hC_F\|u_{\hat\theta}-\hat u\|_{L^\infty(\Omega_T)}\|\nabla \hat u\|_{L^1(\Omega_T)}
	\notag\\
	&\le \mathcal L_{\mathrm{reg}}(\hat u)+Ch\varepsilon
	\le \mathcal L_{\mathrm{reg}}(\hat u)+Ch^3.
	\label{eq:proof_reg}
\end{align}

It remains to control the entropy part. Fix any $k_h\in V_h^c$. Write
\begin{align*}
	\mathcal J_{\mathrm{ent}}(u_{\hat\theta};k_h)
	&=
	\int_{\Omega_r}\nabla\!\cdot\mathbf F(u_{\hat\theta})\,\sgn(u_{\hat\theta}-k_h)\,\d\mathbf z
	\\
	&\quad+\int_{\Omega_s}\nabla\!\cdot\mathbf F(u_{\hat\theta})\,\sgn(u_{\hat\theta}-k_h)\,\d\mathbf z
	=: I_1+I_2.
\end{align*}

For the regular region $\Omega_r$, using $|\sgn|\le 1$ and \eqref{eq:proof_divflux_lip},
\begin{align}
	I_1
	&\le \int_{\Omega_r} |\nabla\!\cdot\mathbf F(\hat u)|\,\d\mathbf z
	+ \int_{\Omega_r} |\nabla\!\cdot\mathbf F(u_{\hat\theta})-\nabla\!\cdot\mathbf F(\hat u)|\,\d\mathbf z
	\notag\\
	&\le \int_{\Omega_r} |\nabla\!\cdot\mathbf F(\hat u)|\,\d\mathbf z
	+ C_F\|\nabla u_{\hat\theta}-\nabla\hat u\|_{L^1(\Omega_T)}
	+ C_F\|u_{\hat\theta}-\hat u\|_{L^\infty(\Omega_T)}\|\nabla\hat u\|_{L^1(\Omega_T)}
	\notag\\
	&\le \int_{\Omega_r} |\nabla\!\cdot\mathbf F(\hat u)|\,\d\mathbf z + C\varepsilon
	\le \int_{\Omega_r} |\nabla\!\cdot\mathbf F(\hat u)|\,\d\mathbf z + Ch^2.
	\label{eq:proof_I1}
\end{align}

For the singular region $\Omega_s$, we write
\begin{align*}
	I_2
	&=
	\int_{\Omega_s}
	\Bigl(
	\nabla\!\cdot\mathbf F(u_{\hat\theta})\,\sgn(u_{\hat\theta}-k_h)
	-\nabla\!\cdot\mathbf F(\hat u)\,\sgn(\hat u-k_h)
	\Bigr)\,\d\mathbf z\\
	&\quad+\int_{\Omega_s}\nabla\!\cdot\mathbf F(\hat u)\,\sgn(\hat u-k_h)\,\d\mathbf z
	\\
	&=
	\int_{\Omega_s}
	\Bigl(
	\nabla\!\cdot\mathbf F(u_{\hat\theta})\,\sgn(u_{\hat\theta}-k_h)
	-\nabla\!\cdot\mathbf F(\hat u)\,\sgn(\hat u-k_h)
	\Bigr)\,\d\mathbf z\\
	&\quad+\mathcal J_{\mathrm{ent}}(\hat u;k_h)
	-\int_{\Omega_r}\nabla\!\cdot\mathbf F(\hat u)\,\sgn(\hat u-k_h)\,\d\mathbf z.
\end{align*}
Hence,
\begin{equation}\label{eq:proof_I2_reduction}
	I_2-\mathcal J_{\mathrm{ent}}(\hat u;k_h)
	\le
	\int_{\Omega_s}
	\Bigl(
	\nabla\!\cdot\mathbf F(u_{\hat\theta})\,\sgn(u_{\hat\theta}-k_h)
	-\nabla\!\cdot\mathbf F(\hat u)\,\sgn(\hat u-k_h)
	\Bigr)\,\d\mathbf z
	+\int_{\Omega_r} |\nabla\!\cdot\mathbf F(\hat u)|\,\d\mathbf z.
\end{equation}

Now apply the cellwise identity \eqref{eq:cellwise_entropy_ibp} on each
$\widetilde K:=K\cap\Omega_s$, $K\in\Lambda_h$, first with $v=u_{\hat\theta}$ and then with $v=\hat u$.
Subtracting the two identities gives
\begin{equation}\label{eq:proof_T123}
	\int_{\Omega_s}
	\Bigl(
	\nabla\!\cdot\mathbf F(u_{\hat\theta})\,\sgn(u_{\hat\theta}-k_h)
	-\nabla\!\cdot\mathbf F(\hat u)\,\sgn(\hat u-k_h)
	\Bigr)\,\d\mathbf z
	= T_1+T_2+T_3,
\end{equation}
where
\begin{align*}
	T_1&:=\sum_{K\in\Lambda_h}\int_{\partial \widetilde K}
	\bigl(\mathbf F(u_{\hat\theta})-\mathbf F(\hat u)\bigr)\,\sgn(u_{\hat\theta}-k_h)\cdot \mathbf n\,\d s,\\
	T_2&:=\sum_{K\in\Lambda_h}\int_{\partial \widetilde K}
	\bigl(\mathbf F(\hat u)-\mathbf F(k_h)\bigr)\,
	\bigl(\sgn(u_{\hat\theta}-k_h)-\sgn(\hat u-k_h)\bigr)\cdot \mathbf n\,\d s,\\
	T_3&:=\sum_{K\in\Lambda_h}\int_{\widetilde K}
	\nabla\!\cdot\mathbf F(k_h)\,
	\bigl(\sgn(u_{\hat\theta}-k_h)-\sgn(\hat u-k_h)\bigr)\,\d\mathbf z.
\end{align*}

Using the Lipschitz continuity \eqref{eq:proof_flux_lip}, the bound $|\sgn|\le 1$, and the fact $\partial \widetilde K \subset (\partial K\cap\Omega_s)\cup(K\cap\partial\Omega_s)$, we obtain
\begin{align}
	|T_1|
	&\le \sum_{K\in\Lambda_h}\int_{\partial \widetilde K}
	|\mathbf F(u_{\hat\theta})-\mathbf F(\hat u)|\,\d s\leq C_F \varepsilon \sum_{K\in\Lambda_h}|\partial \widetilde K|
	\le C_F\varepsilon
	\Bigl(\sum_{K\in\Lambda_h} |\partial K| + |\partial\Omega_s|\Bigr).
	\label{eq:proof_T1a}
\end{align}
By quasi-uniformity of $\Lambda_h$, $\sum_{K\in\Lambda_h} |\partial K|\lesssim h^{-1}$, and from the assumption \eqref{eq:assum_Omega_s} on $\Omega_s$, we have $|\partial\Omega_s|\lesssim C_{\Omega}$. Therefore,
\begin{equation}\label{eq:proof_T1}
	|T_1|\le C\varepsilon h^{-1}\le Ch.
\end{equation}

For $T_2$, note that $\sgn(u_{\hat\theta}-k_h)\neq \sgn(\hat u-k_h)
\;\Longrightarrow\;
|\hat u-k_h|\le |u_{\hat\theta}-\hat u|\le \varepsilon.$
Hence, on the set where the sign changes, \eqref{eq:proof_flux_lip} yields
\[
|\mathbf F(\hat u)-\mathbf F(k_h)|
\le C_F |\hat u-k_h|
\le C_F\varepsilon.
\]
Therefore,
\begin{align}
	|T_2|
	&\le 2\sum_{K\in\Lambda_h}\int_{\partial \widetilde K\cap\{|\hat u-k_h|\le \varepsilon\}}
	|\mathbf F(\hat u)-\mathbf F(k_h)|\,\d s
	\notag\\
	&\le  2C_{F}\varepsilon \sum_{K\in\Lambda_h}|\partial \widetilde K|
	\le C\varepsilon h^{-1}
	\le Ch.
	\label{eq:proof_T2}
\end{align}

Finally, since $k_h\in V_h^c$, inverse estimates on the quasi-uniform mesh $\Lambda_h$ give
\[
\|\nabla k_h\|_{L^\infty(K)}\lesssim h^{-1}
\qquad\text{for all }K\in\Lambda_h.
\]
Because $\|k_h\|_{L^\infty(\Omega_T)}\le c$, we also have
$\|\mathbf f'(k_h)\|_{L^\infty(\Omega_T)}\le C$, and thus
\[
\|\nabla\!\cdot\mathbf F(k_h)\|_{L^\infty(\Omega_T)}\lesssim h^{-1}.
\]
Using $|\Omega_s|\le C_\Omega h^2$ in \eqref{eq:assum_Omega_s}, we infer
\begin{equation}\label{eq:proof_T3}
	|T_3|
	\le 2\int_{\Omega_s} |\nabla\!\cdot\mathbf F(k_h)|\,\d\mathbf z
	\le C h^{-1} |\Omega_s|
	\le Ch.
\end{equation}

Combining \eqref{eq:proof_I2_reduction}--\eqref{eq:proof_T3}, we obtain
\[
I_2
\le \mathcal J_{\mathrm{ent}}(\hat u;k_h)
+ \int_{\Omega_r} |\nabla\!\cdot\mathbf F(\hat u)|\,\d\mathbf z
+ Ch.
\]
Together with \eqref{eq:proof_I1}, this yields
\[
\mathcal J_{\mathrm{ent}}(u_{\hat\theta};k_h)
\le
\mathcal J_{\mathrm{ent}}(\hat u;k_h)
+ C h
+ C\int_{\Omega_r} |\nabla\!\cdot\mathbf F(\hat u)|\,\d\mathbf z,
\qquad \forall\,k_h\in V_h^c,
\]
after absorbing the $O(h^2)$ term into $Ch$. Taking the supremum over $k_h\in V_h^c$ gives
\begin{equation}\label{eq:proof_ent_final}
	\sup_{k_h\in V_h^c}\mathcal J_{\mathrm{ent}}(u_{\hat\theta};k_h)
	\le
	\sup_{k_h\in V_h^c}\mathcal J_{\mathrm{ent}}(\hat u;k_h)
	+ C h
	+ C\int_{\Omega_r} |\nabla\!\cdot\mathbf F(\hat u)|\,\d\mathbf z.
\end{equation}

Finally, by \eqref{eq:proof_bc}, \eqref{eq:proof_reg}, and \eqref{eq:proof_ent_final},
\begin{align*}
	\mathcal L(u_{\hat\theta})
	&=
	\mathcal L_{ibc}(u_{\hat\theta})
	+\mathcal L_{\mathrm{reg}}(u_{\hat\theta})
	+\sup_{k_h\in V_h^c}\mathcal J_{\mathrm{ent}}(u_{\hat\theta};k_h)
	\\
	&\le
	\mathcal L_{ibc}(\hat u)
	+\mathcal L_{\mathrm{reg}}(\hat u)
	+\sup_{k_h\in V_h^c}\mathcal J_{\mathrm{ent}}(\hat u;k_h)
	+ Ch
	+ C\int_{\Omega_r} |\nabla\!\cdot\mathbf F(\hat u)|\,\d\mathbf z
	\\
	&=
	\mathcal L(\hat u)
	+ Ch
	+ C\int_{\Omega_r} |\nabla\!\cdot\mathbf F(\hat u)|\,\d\mathbf z.
\end{align*}
After a harmless redefinition of the constant \(C\), this is precisely the asserted estimate.
\hfill\qed

\subsection{Proof of Theorem~\ref{thm:shock_smooth_initial} (shock formation with smooth initial data)}
\label{appendix:proof_smooth_init}

\begin{proof}
Let \((x_q,\tau_q)\) be the shock-formation point from Assumption~\ref{assump:piecewise_smooth}\textup{(iv)}. Set
\[
u_q:=u(x_q,\tau_q),\qquad
a_q:=f'(u_q),\qquad
\overline\gamma_q(t):=x_q+a_q(t-\tau_q),
\]
and let
\[
\Gamma:=\{(x,t)\in\Omega_T:\ x=\gamma_q^b(t),\ \tau_q<t\le T\}
\]
be the newly formed shock curve. Away from \(\Gamma\), the solution is classical and satisfies \(\nabla\!\cdot\mathbf F(u)=0\); across \(\Gamma\), the one-sided traces satisfy the Rankine--Hugoniot relation and the entropy condition \eqref{eq:reformulated_entropy_condition}.

\medskip
\noindent\textbf{Step 1: singular bounds near the birth point.}
Set
\[
\hat t:=t-\tau_q,\qquad
\hat x:=x-\overline\gamma_q(t),
\]
and define
\[
G(\hat t,\hat x):=\bigl(|\hat t|^3+\hat x^2\bigr)^{-5/6},
\qquad
H(\hat t,\hat x):=\bigl(|\hat t|^3+\hat x^2\bigr)^{-1/3}.
\]
By Assumption~\ref{assump:piecewise_smooth}\textup{(iv)}, and by the uniform \(W^{2,\infty}\) regularity away from the birth point,
\begin{equation}\label{eq:shock_birth_local_bounds}
|\nabla^2u(x,t)|\le C\bigl(G(\hat t,\hat x)+1\bigr)
\qquad\text{in }\Omega_T\setminus\Gamma,
\end{equation}
and, for the post-birth one-sided states,
\begin{equation}\label{eq:shock_birth_local_bounds_grad}
|\nabla u^\pm(x,t)|\le C\bigl(H(\hat t,\hat x)+1\bigr),
\qquad t>\tau_q .
\end{equation}

Let
\[
E_h:=\Bigl([x_q-c_0h,x_q+c_0h]\times[\tau_q-h,\tau_q+h]\Bigr)\cap\Omega_T,
\]
where \(c_0>|a_q|+\sup_{|v|\le \|u\|_{L^\infty(\Omega_T)}}|f'(v)|+1\). Enlarge \(E_h\), if necessary, to the union of cells of \(\Lambda_h\) intersecting it. Then \(|E_h|\lesssim h^2\), and on \(\Omega_T\setminus E_h\) one has either \(|\hat t|\ge h\) or \(|\hat x|\ge h\). Therefore,
\begin{equation}\label{eq:shock_birth_D2_log}
\int_{\Omega_T\setminus E_h}|\nabla^2u|\,\d\z
\lesssim
\int_{\Omega_T\setminus E_h}\bigl(G(\hat t,\hat x)+1\bigr)\,\d\z
\lesssim |\ln h|.
\end{equation}
Similarly, since \(H(\hat t,\hat x)\lesssim |\hat t|^{-1}\) on \(\Gamma\),
\begin{equation}\label{eq:shock_birth_D1_log}
\int_{\Gamma\setminus E_h}
\bigl(|\nabla u^-|+|\nabla u^+|\bigr)\,\d s
\lesssim |\ln h|.
\end{equation}

\medskip
\noindent\textbf{Step 2: CPwL competitor and small loss.}
Set \(\varepsilon:=h^2\). As in Lemma~\ref{lem:Lu_shock}, after an \(O(h)\) translation if needed, assume that \(\Gamma\setminus E_h\) does not significantly intersect the mesh skeleton. Build an \(\varepsilon\)-thin strip \(\widetilde\Lambda_\Gamma\) around \(\Gamma\setminus E_h\), with piecewise linear boundaries \(\widetilde\Gamma^\pm\), and triangulate it by anisotropic triangles with tangential length \(O(h)\) and normal thickness \(O(\varepsilon)\). Outside \(E_h\cup\widetilde\Lambda_\Gamma\), use a shape-regular mesh of size \(O(h)\). Inside \(E_h\), use a bounded CPwL extension matching the boundary nodal values.

Define \(\hat u\) by the nodal \(P^1\) interpolant of the corresponding smooth branch of \(u\) outside \(E_h\cup\widetilde\Lambda_\Gamma\), by linear connection of the two one-sided traces across \(\widetilde\Lambda_\Gamma\), and by the bounded extension inside \(E_h\). Then
\[
\|\hat u(\cdot,0)-u_0\|_{L^1(\Omega)}\lesssim h^2,
\qquad
\|\hat u\|_{L^1(\partial\Omega\times(0,T))}=0.
\]
Let
\[
\Omega_s:=E_h\cup\widetilde\Lambda_\Gamma,\qquad
\Omega_r:=\Omega_T\setminus\Omega_s.
\]
Using \eqref{eq:shock_birth_D2_log} and the standard interpolation estimate on \(\Omega_r\),
\begin{equation}\label{eq:shock_birth_regular_residual}
\int_{\Omega_r}|\nabla\!\cdot\mathbf F(\hat u)|\,\d\z
\lesssim h|\ln h|.
\end{equation}
Moreover, the strip and \(E_h\) have total measure \(O(h^2)\), and the possible large gradients are supported only there; hence
\[
h\int_{\Omega_s}|\nabla\!\cdot\mathbf F(\hat u)|\,\d\z\lesssim h.
\]
Using the cellwise identity \eqref{eq:cellwise_entropy_ibp}, the entropy condition on \(\Gamma\), and the trace bound \eqref{eq:shock_birth_D1_log}, the same argument as in Lemma~\ref{lem:Lu_shock} yields
\begin{equation}\label{eq:shock_birth_hat_loss}
\sup_{k_h\in V_h^c}\J_{\mathrm{ent}}(\hat u;k_h)
\lesssim h|\ln h|.
\end{equation}
Combining these bounds gives
\[
\L(\hat u)\lesssim h|\ln h|.
\]
By Lemma~\ref{lem:loss_cpwl_approx}, there exists a clipped \(\tanh\) neural network \(u_{\hat\theta}\) of the stated size such that
\[
\L(u_{\hat\theta})
\le
\L(\hat u)+Ch+\int_{\Omega_r}|\nabla\!\cdot\mathbf F(\hat u)|\,\d\z
\lesssim h|\ln h|.
\]
Since \(\theta^*\) minimizes the loss,
\begin{equation}\label{eq:shock_birth_min_loss}
\J_{\mathrm{ent}}(u_{\theta^*};k_h)
\le
\mathscr L(\theta^*)
\le
\L(u_{\hat\theta})
\lesssim h|\ln h|,
\qquad \forall k_h\in V_h^c,
\end{equation}
and
\begin{equation}\label{eq:shock_birth_residual_bound}
\int_{\Omega_T}|\nabla\!\cdot\mathbf F(u_{\theta^*})|\,\d\z
\lesssim |\ln h|.
\end{equation}

\medskip
\noindent\textbf{Step 3: DPwP comparison and shifted-interface estimate.}
Let \(t_h:=\tau_q+h\). Starting from the portion of \(\Gamma\) with \(t\ge t_h\), construct \(M\) disjoint shifted strips of width \(O(h)\), as in the proof of Theorem~\ref{thm:shock}. By \eqref{eq:shock_birth_residual_bound}, one strip, denoted by \(\Omega_j\), satisfies
\begin{equation}\label{eq:shock_birth_selected_strip}
\int_{\Omega_j}|\nabla\!\cdot\mathbf F(u_{\theta^*})|\,\d\z
\lesssim \frac{|\ln h|}{M}.
\end{equation}
Let \(\Gamma_j\) be its mesh-aligned left boundary and let \(\widetilde\Gamma_j\) be the shifted copy of the true shock curve inside the strip.

Define \(\tilde u_h\in V_h^c\) as follows. On the left of \(\Gamma_j\), take the cellwise \(Q_1\) interpolant of the corresponding smooth branch of \(u\). On the right of \(\Gamma_j\), take the cellwise \(Q_1\) interpolant of a smooth extension of the right post-shock branch. Inside the \(O(h)\) neighborhood of the birth point, use any bounded DPwP extension. Then \(\|\tilde u_h\|_{W_h^{1,\infty}}\le c\).

Taking \(k_h=\tilde u_h\) in \eqref{eq:shock_birth_min_loss} and applying \eqref{eq:cellwise_entropy_ibp} gives
\begin{equation}\label{eq:shock_birth_error_identity}
\begin{aligned}
\|u_{\theta^*}(\cdot,T)-\tilde u_h(\cdot,T)\|_{L^1(\Omega)}
\le\;&
\|u_{\theta^*}(\cdot,0)-\tilde u_h(\cdot,0)\|_{L^1(\Omega)}  \\
&+
\left|\int_{\Omega_T}\nabla\!\cdot\mathbf F(\tilde u_h)
\sgn(u_{\theta^*}-\tilde u_h)\,\d\z\right|
+ I_h
+ Ch|\ln h|,
\end{aligned}
\end{equation}
where \(I_h\) denotes the jump integral over the mesh-aligned interface \(\Gamma_j\). The lateral boundary contribution is bounded by \(C\L_{\mathrm{ibc}}(u_{\theta^*})\) and has been absorbed into the last term.

Using \eqref{eq:shock_birth_D2_log}, the volume term satisfies
\begin{equation}\label{eq:shock_birth_volume_tilde}
\left|\int_{\Omega_T}\nabla\!\cdot\mathbf F(\tilde u_h)
\sgn(u_{\theta^*}-\tilde u_h)\,\d\z\right|
\lesssim h|\ln h|+Mh|\ln h|+h.
\end{equation}
To estimate \(I_h\), transfer the jump integral from \(\Gamma_j\) to \(\widetilde\Gamma_j\) through \(\Omega_j\). By \eqref{eq:shock_birth_selected_strip}, the interpolation estimate, and \eqref{eq:shock_birth_D2_log},
\[
\text{transfer error}
\lesssim
\frac{|\ln h|}{M}+h|\ln h|+h.
\]
On \(\widetilde\Gamma_j\), compare the shifted traces with the true traces on \(\Gamma\). Since the shift is \(O(Mh)\), \eqref{eq:shock_birth_D1_log} gives
\[
\int_{\widetilde\Gamma_j}
\Bigl(
|\tilde u_h^- - u^-|+|\tilde u_h^+ - u^+|
\Bigr)\,\d s
\lesssim h|\ln h|+Mh|\ln h|.
\]
Using the entropy condition \eqref{eq:reformulated_entropy_condition} on the true shock curve, we obtain
\begin{equation}\label{eq:shock_birth_interface_term}
I_h
\lesssim
\frac{|\ln h|}{M}
+h|\ln h|
+Mh|\ln h|
+h.
\end{equation}
Substituting \eqref{eq:shock_birth_volume_tilde} and \eqref{eq:shock_birth_interface_term} into \eqref{eq:shock_birth_error_identity} yields
\[
\|u_{\theta^*}(\cdot,T)-\tilde u_h(\cdot,T)\|_{L^1(\Omega)}
\le
\|u_{\theta^*}(\cdot,0)-\tilde u_h(\cdot,0)\|_{L^1(\Omega)}
+
C\left(
\frac{|\ln h|}{M}
+h|\ln h|
+Mh|\ln h|
+h
\right).
\]
Choose \(M=\lceil h^{-1/2}\rceil\). Then
\[
\|u_{\theta^*}(\cdot,T)-\tilde u_h(\cdot,T)\|_{L^1(\Omega)}
\lesssim
\|u_{\theta^*}(\cdot,0)-\tilde u_h(\cdot,0)\|_{L^1(\Omega)}
+
h^{1/2}|\ln h|.
\]
Since \(u_0\) is smooth,
\[
\|\tilde u_h(\cdot,0)-u_0\|_{L^1(\Omega)}\lesssim h,
\qquad
\|u_{\theta^*}(\cdot,0)-u_0\|_{L^1(\Omega)}
\le \L_{\mathrm{ibc}}(u_{\theta^*})
\lesssim h|\ln h|.
\]
Hence
\[
\|u_{\theta^*}(\cdot,0)-\tilde u_h(\cdot,0)\|_{L^1(\Omega)}
\lesssim h|\ln h|.
\]
Finally, because \(T>\tau_q+\delta_*\), the terminal-time one-sided states are uniformly \(W^{2,\infty}\) near \(\Gamma\cap\{t=T\}\). Therefore,
\[
\|\tilde u_h(\cdot,T)-u(\cdot,T)\|_{L^1(\Omega)}
\lesssim Mh+h
\lesssim h^{1/2}.
\]
The triangle inequality gives
\[
\|u_{\theta^*}(\cdot,T)-u(\cdot,T)\|_{L^1(\Omega)}
\lesssim h^{1/2}|\ln h|,
\]
which proves \eqref{eq:shock_smooth_initial}.
\end{proof}
\subsection{Proof of Theorem~\ref{thm:main_1d} (main result in 1D)}\label{appendix:proof_1d_thm}
Let \(\{\mathcal P_q\}_{q=1}^Q\) be the finite family from Assumption~\ref{assump:piecewise_smooth}. The local estimates in the proofs of Theorems~\ref{thm:shock}, \ref{thm:shock_interaction}, \ref{thm:rarefaction}, \ref{thm:compound_wave}, and \ref{thm:shock_smooth_initial} depend only on the corresponding local regularity, the Rankine--Hugoniot relation, the entropy condition, and the local geometry of the wave pattern. Hence the same constructions can be applied on each piece \(\mathcal P_q\), after a translation of coordinates used only to simplify notation.

\noindent\textbf{Step 1: Construction of a global CPwL competitor.}
For each non-smooth piece $\mathcal P_q$ (that is, a regular shock piece, a rarefaction piece, a
compound-wave piece, or a regular interaction piece), choose an open set
$\mathcal U_q$ with piecewise $C^2$ boundary such that $\overline{\mathcal U_q}\subset \mathcal P_q$ 
and $\mathcal U_q$ contains the unique wave/event carried by $\mathcal P_q$. Since the number of pieces is
finite and no piece contains two distinct local events, the sets $\mathcal U_q$ can be chosen pairwise
disjoint. Set $\Omega_{\rm reg}:=\Omega_T\setminus \bigcup_q \mathcal U_q$, 
where the union runs over all non-smooth pieces. By
Assumption~\ref{assump:piecewise_smooth}, $u\in W^{2,\infty}\cap C^2$ on each connected component of
$\Omega_{\rm reg}$.

On each $\mathcal U_q$ we choose the local CPwL competitor from the corresponding earlier proof:
\begin{itemize}
\item on a regular shock piece, the construction from Lemma~\ref{lem:Lu_shock};
\item on a rarefaction piece, the construction from Lemma~\ref{lem:Lu_rarefaction};
\item on a compound-wave piece, the local construction as in 
Theorem~\ref{thm:compound_wave};
\item on a regular interaction piece, the local construction as in 
Theorem~\ref{thm:shock_interaction};
\item on a shock-birth piece, the local construction from Theorem~\ref{thm:shock_smooth_initial};
\end{itemize}
On each connected component of $\Omega_{\rm reg}$ we take the nodal $P^1$ interpolant of $u$.

Because every $\partial\mathcal U_q$ lies in a smooth region, the above local CPwL constructions can be
chosen to match the exact nodal values on $\partial\mathcal U_q$. After a finite conforming refinement along
the boundaries of the $\mathcal U_q$'s, they patch together into a global continuous CPwL function
$\hat u$ on a conforming simplicial partition $\mathcal T_h$ of $\Omega_T$. We have $\#\mathcal T_h = O(h^{-2})$, 
and the patch complexity of $\mathcal T_h$ is uniformly bounded. Indeed, in a neighborhood of each \(\partial\mathcal U_q\) the exact
solution is smooth, and the local constructions reduce there to the
ordinary nodal interpolant of the same smooth trace of \(u\). Hence the
nodal values agree on the common boundary. The conforming refinement
needed along the finitely many boundaries \(\partial\mathcal U_q\)
introduces only \(O(h^{-1})\) additional simplices, because each
\(\partial\mathcal U_q\) has uniformly bounded \(\mathcal H^1\)-measure.
This is dominated by the existing \(O(h^{-2})\) elements and does
not affect the loss estimates or the uniformly bounded patch
complexity.

Let $\Omega_s$ be the union of the local singular regions appearing in the above constructions, namely the
thin shock layers, rarefaction-tip neighborhoods, interaction patches, and shock-birth neighborhoods, and set $\Omega_r:=\Omega_T\setminus \Omega_s$. 
Each local singular region has measure $O(h^2)$, and there are only finitely many of them; hence $|\Omega_s|\lesssim h^2$. 
Moreover, $\hat u$ is piecewise linear on $\Omega_r$ over a shape-regular mesh of size $O(h)$. 

Summing the local estimates from the preceding one-dimensional constructions yields
\[
\L(\hat u)\lesssim h|\ln h|,
\quad
\int_{\Omega_r} |\nabla\!\cdot\mathbf F(\hat u)|\,\d\z \lesssim h|\ln h|.
\]
Indeed, smooth, regular-shock, and regular-interaction pieces contribute $O(h)$, while rarefaction,
compound-wave pieces contribute $O(h|\ln h|)$. By construction of the CPwL function $\hat u$, we have $\|\hat u\|_{L^\infty(\Omega_T)} \le c/2 - \delta_c$, 
so the assumptions of Lemma~\ref{lem:loss_cpwl_approx} are satisfied.
Hence, there exists a clipped \(\tanh\) neural network \(u_{\hat\theta}\) of the stated size such that
\[
\L(u_{\hat\theta})
\le \L(\hat u)+Ch+\int_{\Omega_r} |\nabla\!\cdot\mathbf F(\hat u)|\,\d\z
\lesssim h|\ln h|.
\]
Since $\theta^\ast$ minimizes the loss, we obtain
\[
\J_{\mathrm{ent}}(u_{\theta^\ast};k_h)\le \mathscr{L}(\theta^\ast)\le \L(u_{\hat\theta})\lesssim h|\ln h|,
\quad \forall\, k_h\in V_h^c.
\]

\noindent\textbf{Step 2: Construction of a global DPwP test function.}
We now construct $\tilde u_h\in V_h^c$. In every shifted-interface argument we choose the same strip parameter $M:=\left\lceil h^{-1/2}\right\rceil $. 
For each non-smooth piece $\mathcal P_q$, enlarge $\mathcal U_q$ (if necessary) to the union of the cells of
$\Lambda_h$ intersecting it; this changes the measure by at most $O(h)$, which is negligible for the final
estimate. For all sufficiently small $h$ these enlarged sets remain pairwise disjoint.

On each enlarged $\mathcal U_q$ we define $\tilde u_h$ by the local DPwP construction used in the
corresponding proof:
\begin{itemize}
\item on a regular shock piece, the shifted-interface construction from the proof of Theorem~\ref{thm:shock}; on a regular interaction piece, the corresponding local construction from the proof of Theorem~\ref{thm:shock_interaction};
\item on a rarefaction piece, the fan-adapted DPwP as in 
Theorem~\ref{thm:rarefaction};
\item on a compound-wave piece, the local construction as in 
Theorem~\ref{thm:compound_wave};
\item on a shock-birth piece, the shifted-interface construction from the proof of Theorem~\ref{thm:shock_smooth_initial};
\end{itemize}
On the remaining cells, namely on $\Omega_T\setminus\bigcup_q \mathcal U_q$, we take the cellwise $Q_1$
interpolant of $u$.

By the standing choice of \(c\) made in Section~\ref{sec:notation}, the local
DPwP constructions satisfy the uniform bound
\[
        \|\tilde u_h\|_{W_h^{1,\infty}}\le c
\] and agree with the same smooth trace data on the boundaries of the $\mathcal U_q$'s. Let $\Gamma_h$ denote the union of the mesh-aligned jump interfaces of $\tilde u_h$; the number of connected components of $\Gamma_h$ is bounded independently of $h$.

Taking $k_h=\tilde u_h$ in the previous bound and applying the cellwise identity
\eqref{eq:cellwise_entropy_ibp}, we get
\[
\begin{aligned}
& \|u_{\theta^\ast}(\cdot,T)-\tilde u_h(\cdot,T)\|_{L^1(\Omega)} \\
\le & C h|\ln h| +
\|u_{\theta^\ast}(\cdot,0)-\tilde u_h(\cdot,0)\|_{L^1(\Omega)} + 
\left|\int_{\Omega_T}\nabla\!\cdot\mathbf F(\tilde u_h)\,
\sgn\!\left(u_{\theta^\ast}-\tilde u_h\right)\,\d\z\right|
\\
+ & \left| \int_{\partial\Omega\times(0,T)}
\left(\mathbf f(u_{\theta^*})-\mathbf f(\tilde u_h)\right)
\sgn(u_{\theta^*}-\tilde u_h)\cdot\mathbf n_{\partial\Omega}\,\d s \right| +
\int_{\Gamma_h}
\llbracket (\mathbf F(\tilde u_h)-\mathbf F(u_{\theta^\ast}))
\sgn\!\left(\tilde u_h-u_{\theta^\ast}\right)\rrbracket\cdot \n\,\d s
.
\end{aligned}
\]
By the Lipschitz continuity of $\mathbf f$ and the construction of $\tilde u_h$ on the boundary,
\begin{equation*}
| \int_{\partial\Omega\times(0,T)}
\left(\mathbf f(u_{\theta^*})-\mathbf f(\tilde u_h)\right)
\sgn(u_{\theta^*}-\tilde u_h)\cdot\mathbf n_{\partial\Omega}\,\d s|
\le
C\|u_{\theta^*}\|_{L^1(\partial\Omega\times(0,T))}
\le
C\L(u_{\hat\theta})\lesssim h|\ln h| .
\end{equation*}

We now sum the local estimates from the proofs of Theorems~\ref{thm:shock}, \ref{thm:shock_interaction}, \ref{thm:rarefaction}, \ref{thm:compound_wave}, and \ref{thm:shock_smooth_initial}. On smooth pieces, we have the standard interpolation bound \(O(h)\). With the common choice \(M=\lceil h^{-1/2}\rceil\), regular shocks and regular interactions contribute \(O(h^{1/2})\), rarefaction pieces contribute \(O(h|\ln h|)\), compound-wave pieces contribute at most \(O(h^{1/2}|\ln h|)\), and shock-birth pieces contribute \(O(h^{1/2}|\ln h|)\). Hence the global contribution is \(O(h^{1/2}|\ln h|)\). Hence
\[
\left|\int_{\Omega_T}\nabla\!\cdot\mathbf F(\tilde u_h)\,
\sgn\!\left(u_{\theta^\ast}-\tilde u_h\right)\,\d\z\right|
+
\int_{\Gamma_h}
\llbracket (\mathbf F(\tilde u_h)-\mathbf F(u_{\theta^\ast}))
\sgn\!\left(\tilde u_h-u_{\theta^\ast}\right)\rrbracket\cdot \n\,\d s
\lesssim h^{1/2}|\ln h|.
\]
Therefore,
\[
\|u_{\theta^\ast}(\cdot,T)-\tilde u_h(\cdot,T)\|_{L^1(\Omega)}
\le
\|u_{\theta^\ast}(\cdot,0)-\tilde u_h(\cdot,0)\|_{L^1(\Omega)}
+ C h^{1/2}|\ln h|.
\]

\noindent\textbf{Step 3: Endpoint estimates and conclusion.}
By the same local endpoint estimates, together with the standard \(Q_1\) interpolation estimate on \(\Omega_T\setminus\bigcup_q \mathcal U_q\), we have
\[
\|\tilde u_h(\cdot,0)-u_0\|_{L^1(\Omega)}
+
\|\tilde u_h(\cdot,T)-u(\cdot,T)\|_{L^1(\Omega)}
\lesssim h^{1/2}|\ln h|.
\]
Indeed, smooth pieces contribute \(O(h)\), regular-shock and regular-interaction pieces contribute \(O(h^{1/2})\), rarefaction pieces contribute \(O(h|\ln h|)\), compound-wave pieces contribute \(O((h|\ln h|)^{1/2})\), and shock-birth pieces contribute \(O(h^{1/2}|\ln h|)\).

Moreover,
\[
\|u_{\theta^\ast}(\cdot,0)-u_0\|_{L^1(\Omega)}
\le
\L_{\mathrm{ibc}}(u_{\theta^\ast})
\le
\mathscr L(\theta^\ast)
\lesssim h|\ln h|
\lesssim h^{1/2}|\ln h|.
\]
Therefore,
\[
\|u_{\theta^\ast}(\cdot,0)-\tilde u_h(\cdot,0)\|_{L^1(\Omega)}
\lesssim h^{1/2}|\ln h|.
\]
Combining this with the previous estimate and the triangle inequality gives
\[
\|u_{\theta^\ast}(\cdot,T)-u(\cdot,T)\|_{L^1(\Omega)}
\lesssim h^{1/2}|\ln h|.
\]
\hfill\qed

\subsection{Proof of Lemma~\ref{lem:Lu_shock_3d}}\label{appendix:proof_lemma_nd}\label{app:proof-lemma-nd}
As in Section~\ref{sec:shock_wave}, we assume that the discontinuity surface \(\Gamma\) does not significantly intersect the mesh skeleton \(\partial\Lambda_h\); more precisely,
\[
\mathcal{H}^{d}\!\left(\bigl\{(\x,t)\in \Gamma:\dist((\x,t),\partial \Lambda_h)\lesssim h^2\bigr\}\right)\lesssim h.
\]
If this condition fails, we apply the discrete spatial-shift argument
from Step~3 of the proof of Lemma~\ref{lem:Lu_shock} simultaneously
to the finitely many regular shock branches. Spatial mesh faces are
handled by averaging over $\delta\in[0,h]^d$, while temporal mesh faces
are treated exactly as in the one-dimensional proof, using the uniform
bound on the normal shock speed implied by the Rankine--Hugoniot
relation and the boundedness of $f'$ on the range of $u$. Summing the
branchwise estimates, we obtain a single shift $|\delta|\lesssim h$
for which the translated discontinuity set satisfies the required
nonalignment condition. The corresponding change in the initial-data
term is $O(h)$ by the standard $BV$ translation estimate.

By Lemma~\ref{lem:loss_cpwl_approx}, it suffices to construct a CPwL
function satisfying the hypotheses of that lemma and whose loss
\eqref{eq:loss_func} is $O(h)$.

Let
\begin{equation*}
\Sigma^\varepsilon =
\big\{(\x,t)\in\Omega_T:\operatorname{dist}((\x,t),\Sigma)\lesssim h^2\big\},
\end{equation*}
and
\begin{align*}
	\Lambda_\Sigma:= & \{K\in\Lambda_h:K\cap\Sigma^\varepsilon\neq\emptyset\},\\
	\Lambda_\Gamma:= & \{K\in\Lambda_h\setminus\Lambda_\Sigma:K\cap\Gamma\neq\emptyset\},\\
	\Lambda_h^\circ:= & \Lambda_h\setminus(\Lambda_\Sigma\cup\Lambda_\Gamma).
\end{align*}
Since each cell has diameter $O(h)$,
the union $\Omega_{\Sigma,h} := \bigcup_{K \in \Lambda_\Sigma} K$
is contained in a tubular neighborhood of $\Sigma$
with radius $Ch$.
By the piecewise $C^2$ regularity of $\Sigma$
and the boundedness of its $\mathcal{H}^{d-1}$ measure,
we obtain
\begin{equation}\label{eq:LambdaSigma_card}
	|\Omega_{\Sigma,h}| \;\lesssim\; h^2,
	\quad
	\#\Lambda_\Sigma \;\lesssim\; h^{-(d-1)}.
\end{equation}
By the same reasoning, the boundedness of $\mathcal H^d(\Gamma)$ implies $\#\Lambda_\Gamma\le C h^{-d}$.

We now define a continuous piecewise affine function $\hat u$ on a conforming simplicial refinement of $\Lambda_h$.
\begin{itemize}
	\item For \(K\in\Lambda_h^\circ\), the exact solution \(u\) is smooth in a neighborhood of \(K\) by Assumption~\ref{assump:shock_3d}; we therefore define \(\hat u\) on \(K\) as the linear nodal interpolant of \(u\).
\item If $K\in\Lambda_\Gamma$, then $K\cap\Gamma\ne\varnothing$.
In general, $K\cap\Gamma$ may contain several connected smooth
pieces. We group these pieces according to the smooth shock
branches to which they belong, and denote the corresponding
portions by $\Gamma_K^{(j)}$. Pieces belonging to the same
smooth branch are treated together, even when the branch
crosses the same cell face more than once.

For distinct connected components of
$\Gamma\setminus\Sigma^\varepsilon$,
Assumption~\ref{assump:shock_3d}\textup{(iii)} guarantees
their separation. Choosing $\varepsilon=c'h^2$
with the fixed geometric constants appropriately chosen,
we construct thin transition layers consistently along the
smooth branches, so that layers associated with distinct
branches are pairwise disjoint outside $\Omega_{\Sigma,h}$.
Let $K_\Gamma^{(j)}$ denote the local layer associated with
$\Gamma_K^{(j)}$, including its local continuation across
$\partial K$ into neighboring cells. Each such layer has
tangential size $O(h)$ and thickness $O(\varepsilon)$, and hence
\[
    |K_\Gamma^{(j)}|\lesssim \varepsilon h^d.
\]

Outside the transition layers, we take the nodal interpolants
of the corresponding one-sided traces of $u$, while inside each layer we extend the boundary nodal values by continuous piecewise affine interpolation. Layers belonging to the same branch are constructed
consistently wherever they meet. If a local continuation enters
a cell in $\Lambda_h^\circ$, we adjust the smooth-cell
interpolation defined above to accommodate the layer.

Since there are finitely many shock hypersurfaces with
compact $C^2$ closures, the number of distinct local branches
meeting any sufficiently small cell is uniformly bounded in $h$. Therefore,
\begin{align}\label{eq:Gamma_K}
    |K_\Gamma|
    =\left|\bigcup_j K_\Gamma^{(j)}\right|
    \lesssim \varepsilon h^d.
\end{align}

Finally, whenever a thin layer crosses a common face of two
neighboring cells, we use a common refinement of the induced
face partitions, extend it conformingly into both cells, and
assign identical nodal values at all vertices of the common
face refinement. The piecewise affine traces therefore
coincide on the entire common face, and the resulting
function $\hat u$ is globally continuous.
	\item On $\Omega_{\Sigma,h}$ we choose a bounded CPwL extension that matches the traces prescribed on the boundary of neighboring cells.
\end{itemize}
The resulting function $\hat u$ is continuous, piecewise linear, and is defined on a simplicial mesh with $O(h^{-(d+1)})$ elements.

We claim that $\L(\hat u)\lesssim h$. First, note that  $\hat u=0$ on $\partial\Omega\times(0,T)$ and the initial and boundary error can be bounded as
\begin{equation}\label{eq:hatu_bc_md}
	\|\hat u(\cdot,0)-u_0(\cdot)\|_{L^1(\Omega)}\lesssim h,\quad \|\hat u\|_{L^1(\partial\Omega\times(0,T))}=0.
\end{equation}
Second, the \(\mathcal L_{\mathrm{reg}}\) term satisfies
\begin{equation}\label{eq:hatu_reg_md}
	h^{-1}\L_{reg}(\hat{u})=\int_{\Omega_T}|\nabla\cdot\mathbf F(\hat u)|\,\d\z\lesssim 1.
\end{equation}
Indeed, outside $\Omega_{\Sigma,h}\cup\bigcup_{K\in\Lambda_\Gamma}K_\Gamma$ the standard interpolation estimate on smooth regions yields a total contribution $O(h)$. The transition layers contribute $O(1)$ because their total measure is $O(\varepsilon)$ (since $|K_{\Gamma}|\lesssim \varepsilon h^{d}$) while $|\nabla\hat u|\lesssim \varepsilon^{-1}$ there. Finally, on $\Omega_{\Sigma,h}$ we have $|\nabla\hat u|\lesssim \varepsilon^{-1}$, and therefore \eqref{eq:LambdaSigma_card} implies a contribution $O(1)$.

It remains to estimate the entropy residual part. Fix $k_h\in V_h^c$ and decompose
\begin{equation}\label{eq:Lu_shock_decompose_3d_revised2}
	\begin{aligned}
		&\sum_{K\in\Lambda_h}\int_K \nabla\cdot\mathbf F(\hat u)\sgn(\hat u-k_h)\,\d\z \\
		=&
		\sum_{K\in\Lambda_h^\circ}\int_K \nabla\cdot\mathbf F(\hat u)\sgn(\hat u-k_h)\,\d\z
		+\sum_{K\in\Lambda_\Gamma}\int_K \nabla\cdot\mathbf F(\hat u)\sgn(\hat u-k_h)\,\d\z \\
		&+\sum_{K\in\Lambda_\Sigma}
		\left(
		\int_{\partial K}(\mathbf F(\hat u)-\mathbf F(k_h))\sgn(\hat u-k_h)\cdot\mathbf n\,\d s
		+\int_K \nabla\cdot\mathbf F(k_h)\sgn(\hat u-k_h)\,\d\z
		\right),
	\end{aligned}
\end{equation}
where we have used the cellwise identity \eqref{eq:cellwise_entropy_ibp} for each $K\in \Lambda_{\Sigma}$.

The same argument as in Section~\ref{sec:shock_wave} gives
\begin{equation}\label{eq:Lu_3d_1_revised2}
	\sum_{K\in\Lambda_h^\circ}\int_K \nabla\cdot\mathbf F(\hat u)\sgn(\hat u-k_h)\,\d\z
	+
	\sum_{K\in\Lambda_\Gamma}\int_K \nabla\cdot\mathbf F(\hat u)\sgn(\hat u-k_h)\,\d\z
	\lesssim h.
\end{equation}
For the remaining term, using the bound $\sup_{\z\in\Omega_T}|\hat u(\z)|
\lesssim
\sup_{\z\in\Omega_T}|u(\z)|$, the piecewise \(W^{1,\infty}\)-regularity of \(k_h\), and \(\#(\Lambda_\Sigma)\lesssim h^{-(d-1)}\), we obtain
\begin{align}\label{eq:Lu_3d_2_revised2}
	&\sum_{K\in\Lambda_\Sigma}
	\left(
	\int_{\partial K}(\mathbf F(\hat u)-\mathbf F(k_h))\sgn(\hat u-k_h)\cdot\mathbf n\,\d s
	+\int_K \nabla\cdot\mathbf F(k_h)\sgn(\hat u-k_h)\,\d\z
	\right)\notag\\
	&\lesssim
	\sum_{K\in\Lambda_\Sigma}
	\left(
	\int_{\partial K}|\mathbf F(\hat u)-\mathbf F(k_h)|\,\d s
	+\int_K |\nabla\cdot\mathbf F(k_h)|\,\d\z
	\right)\notag\\
	&\lesssim \#(\Lambda_\Sigma)\Big(h^d\cdot C(\|\hat{u}\|_{L^{\infty}}, \|k_h\|_{L^{\infty}})+h^{d+1}\cdot\|\nabla k_h\|_{L^{\infty}}\Big)
	\lesssim h.
\end{align}
Combining~\eqref{eq:Lu_shock_decompose_3d_revised2}-\eqref{eq:Lu_3d_2_revised2}, we find
\begin{equation}\label{eq:md_hatJ}
	\sup_{k_h\in V_h^c}\J_{\mathrm{ent}}(\hat u;k_h)\lesssim h.
\end{equation}
Together with \eqref{eq:hatu_bc_md} and \eqref{eq:hatu_reg_md}, this proves
\begin{equation}\label{eq:md_hatloss}
	\L(\hat u)\lesssim h.
\end{equation}
Let \(\Omega_s:=\Omega_{\Sigma,h}\cup\bigcup_{K\in\Lambda_\Gamma}K_\Gamma\) and \(\Omega_r:=\Omega_T\setminus\Omega_s\).
By construction, \(\Omega_s\) consists of a tubular neighborhood of \(\Sigma\) together with a thickened piecewise linear approximation of \(\Gamma\); hence, using also \eqref{eq:LambdaSigma_card}, we have
\[
|\Omega_s|\lesssim h^2,
\qquad
\mathcal H^d(\partial\Omega_s)\lesssim 1.
\]
Moreover, the constructed CPwL function \(\hat u\) satisfies
\[
\|\nabla \hat u\|_{L^1(\Omega_T)}\le C.
\]
Indeed, this follows from the same argument as in the previous constructions: \(\hat u\) has uniformly bounded gradient on the regular region, while on the transition layers and on \(\Omega_{\Sigma,h}\) the gradient may be large but is supported on a set of measure \(O(h^2)\). Therefore Lemma~\ref{lem:loss_cpwl_approx} applies.

Since \(\Omega_r\) excludes both the transition layers and \(\Omega_{\Sigma,h}\), the standard interpolation estimate yields
\[
\int_{\Omega_r}|\nabla\cdot\mathbf F(\hat u)|\,d\mathbf z\lesssim h.
\]
By construction of the CPwL function $\hat u$, we have $\|\hat u\|_{L^\infty(\Omega_T)} \le c/2 - \delta_c$,
so the assumptions of Lemma~\ref{lem:loss_cpwl_approx} are satisfied.
Hence, there exists a clipped \(\tanh\) neural network \(u_{\hat\theta}\) of the stated size such that
\[
\L(u_{\hat\theta})\le \L(\hat u)+Ch+\int_{\Omega_r}|\nabla\cdot\mathbf F(\hat u)|\,d\mathbf z\lesssim h.
\]
This finishes the proof of Lemma~\ref{lem:Lu_shock_3d}.

\end{document}